\documentclass[10pt,reqno]{amsart}
\usepackage{amssymb,mathrsfs,color}

\usepackage{wrapfig}

\usepackage{tikz}
\usepgflibrary{fpu}
\usetikzlibrary{arrows,calc,decorations.pathreplacing}
\definecolor{light-gray1}{gray}{0.90}
\definecolor{light-gray2}{gray}{0.80}

\definecolor{deepgreen}{cmyk}{1,0,1,0.5}

\newcommand{\E}{\mathcal{E}}

\newcommand{\LL}{\mathcal{L}}

\newcommand{\HH}{\mathcal{H}}
\newcommand{\M}{\mathcal{M}}

\newcommand{\Sp}{\mathbb{S}}

\newcommand{\CC}{\mathscr{C}}
\newcommand{\NN}{\mathcal{N}}

\newcommand{\K}{\mathcal{K}}

\newcommand{\N}{\mathbb{N}}
\newcommand{\R}{\mathbb{R}}
\newcommand{\Z}{\mathbb{Z}}

\newcommand{\al}{\alpha}
\newcommand{\be}{\beta}
\newcommand{\ga}{\gamma}
\newcommand{\de}{\delta}
\newcommand{\e}{\varepsilon}
\newcommand{\fy}{\varphi}

\newcommand{\la}{\lambda}

\newcommand{\s}{\sigma}

\newcommand{\p}{\partial}

\newcommand{\supp}{\operatorname{supp}}

\makeatletter

\newcommand{\Rmnum}[1]{\expandafter\@slowromancap\romannumeral #1@}
\makeatother

\newcommand{\lec}{\lesssim}
\newcommand{\gec}{\gtrsim}

\newcommand{\I}{\infty}

\newcommand{\ti}{\widetilde}

\newcommand{\ang}[1]{\left\langle{#1}\right\rangle}
\newcommand{\abs}[1]{\left\lvert{#1}\right\rvert}

\newcommand{\ali}[1]{\begin{align}\begin{split} #1 \end{split}\end{align}}
\newcommand{\ant}[1]{\begin{align*}\begin{split} #1 \end{split}\end{align*}}

\newcommand{\EQ}[1]{\begin{equation}\begin{split} #1 \end{split}\end{equation}}
\newcommand{\Del}[1]{}

\newcommand{\pt}{&}
\newcommand{\pr}{\\ &}
\newcommand{\pq}{\quad}

\def\ti{\tilde}

\numberwithin{equation}{section}

\newtheorem{thm}{Theorem}[section]
\newtheorem{cor}[thm]{Corollary}
\newtheorem{lem}[thm]{Lemma}
\newtheorem{prop}[thm]{Proposition}

\newtheorem{claim}[thm]{Claim}
\theoremstyle{remark}
\newtheorem{rem}{Remark}
\newtheorem{defn}{Definition}

\newcommand{\mand}{{\ \ \text{and} \ \  }}
\newcommand{\mor}{{\ \ \text{or} \ \ }}
\newcommand{\mif}{{\ \ \text{if} \ \ }}
\newcommand{\mfor}{{\ \ \text{for} \ \ }}
\newcommand{\mas}{{\ \ \text{as} \ \ }}

 \def\Id{\mathrm{Id}}
 \def\nn{\nonumber}

\def\eps{\varepsilon}
\def\glei{\mathrm{eq}}

\begin{document}

\title[Equivariant Adkins-Nappi-Skyrme Wave Maps]{Conditional global existence and scattering for a semi-linear Skyrme equation with large data}
\author{Andrew Lawrie}
\begin{abstract}
We study a generalization of energy super-critical wave maps due to Adkins and Nappi that can also be viewed as a simplified version of the Skyrme model. These are maps from $1+3$ dimensional Minkowski space that take values in the $3$-sphere, and it follows that every finite energy Adkins-Nappi wave map has a fixed topological degree which is an integer. Here we initiate the study of the large data dynamics for Adkins-Nappi wave maps by proving that there is no type II blow-up in the class of maps with topological degree zero. In particular, any degree zero map whose critical norm stays bounded must be global-in-time and scatter to zero as $t \to \pm \infty$. 
\end{abstract}

\thanks{Support from the National Science Foundation, DMS-1302782 is gratefully acknowledged. The author would also like to thank Carlos Kenig and Wilhelm Schlag for many helpful conversations.}
\maketitle

\section{Introduction}
In this paper, we study a generalization of energy supercritical wave maps  due to Adkins and Nappi, \cite{AN}, which is a simplified version of the Skyrme model. Before introducing Adkins-Nappi wave maps we give a brief review of the original wave maps problem to motivate the generalization. 

The relevant wave map model in particle physics, called the nonlinear $\s$-model,  concerns  maps  $U: \R^{1+3} \to \Sp^3$. Here  a wave map is defined to be a formal critical point of the Lagrangian 
\ant{
\LL(U) = \frac{1}{2} \int_{\R^{1+3}} \eta^{\al \be} \ang{ \p_{\al} U, \p_{\be} U}_g \, dx\, dt,
} 
where $\eta$ is the Minkowski metric on $\R^{1+3}$ and $g$ is the round metric on $\Sp^3$. The Euler-Lagrange equations for  $\LL$ are given by 
\ant{
\eta^{\al \be} D_{\al} \p_{\be} U = 0,
}
where $D$ is the pull-back covariant derivative on $U^*T\Sp^3$.   Wave maps exhibit a conserved energy, 
 \ant{
 \E(\vec U)(t) = \frac{1}{2} \int_{\R^3} (\abs{\p_t U}_g^2 + \abs{\nabla U}_g^2) \, dx = \textrm{const.}
 }
and are invariant under the scaling 
\ant{
\vec U(t,x) :=(U(t, x), \p_tU(t, x)) \mapsto  \vec U_{\la}(t,x) := (U(t/ \la, x/ \la), \la^{-1} \p_tU(t/\la, x/ \la)), 
} 
for $\la>0$. We refer to the particular model at hand as energy super-critical since one can reduce the energy by concentrating the solution to a point via a rescaling. In particular, 
\ant{
\E(\vec U_{\la})  = \la \E(\vec U),
}
so when $\la \to 0$ the rescaled solution shrinks to a point  but the energy tends to zero, making it energetically favorable for the solution to concentrate. This leads one to expect that smooth finite energy initial data can lead to finite-time blow up, and this is indeed the case. One can see this in the usual $1$-equivariant formulation of this model, where $\psi$ is the azimuth angle measured from the north pole of $\Sp^3$, and the system of equations for $U$ reduce to an equation for $\psi$, viz., 
\ant{
\psi_{tt}- \psi_{rr} - \frac{2}{r} \psi_r+ \frac{\sin 2\psi}{r^2} = 0,
}
with conserved energy 
\ant{
\E(\vec \psi) = \frac{1}{2} \int_0^{\infty} \left(\psi_t^2 + \psi_r^2 + \frac{2\sin^2 \psi}{r^2} \right) \, r^2 \, dr.
}
One notes that initial data $(\psi_0, \psi_1)$ with finite energy must satisfy $\lim_{r \to \infty} \psi_0(r) = n \pi$ and that this integer $n$ is preserved by a continuous  flow. If we also require the origin to be sent to the north pole of the sphere, i.e., $\psi(t, 0)  = 0$, these two conditions give rise to a notion of a topological degree of the map. For example, degree $1$ maps have $\psi_0(0) = 0$,  $\psi_0(\infty) = \pi$, and thus wrap around the sphere once. 
Shatah, \cite{Shatah}, proved that such wave maps can indeed develop singularities in finite time  and later an explicit example of self-similar blow-up was given by Turok, Spergel, \cite{TS}, namely, 
\ant{
\psi(t, r)  = 2 \arctan \left(\frac{r}{t}\right).
}
Further,  there are no finite energy, nontrivial, stationary solutions --harmonic maps in this case-- which are what physicists would call topological solitons.

In the physics literature, there have been several attempts to modify this model to remove the possibility of finite-time blow up and to introduce topological solitons. The most famous modification is due to Skyrme, \cite{Skyrme}. The full Skyrme model involves a quasilinear system of equations, which makes it very difficult to analyze. However, one can make an equivariant assumption as above, which reduces the model to a semi-linear equation for the azimuth angle $\psi$. One of the appealing features of the equivariant Skyrme model  is the existence of nontrivial finite energy stationary solutions called Skyrmions, whose existence and uniqueness was established rigorously in~\cite{McT}. We note that finite energy solutions to the equivariant Skyrme equation exhibit a notion of topological degree similar to what is described for wave maps above. 

Recently, the stability of the Skyrmion was addressed numerically in \cite{BCR}.  Global existence and scattering for initial data that is small in the space $(\dot{B}^{\frac{5}{2}}\times \dot{B}^{\frac{3}{2}}) \cap (\dot{H}^1 \times L^2)$ was established in~\cite{GNR}. Global existence for large smooth initial data was proved in~\cite{Li}.  However, much stronger results are conjectured in the literature. In particular, it is believed, see \cite{BCR}, that all smooth finite energy degree $n$ data lead to a global solution which relaxes to a degree $n$ Skyrmion as $t \to \pm \infty$. In other words, Skyrmions are believed to be globally asymptotically stable in the energy space. This full conjecture presents many significant challenges, starting with the fact that the equation, while not scaling invariant, still exhibits features of a supercritical equation. 
 
The difficulties presented by Skyrme lead one to consider even simpler modifications of the wave maps model that still retain the many of the interesting mathematical features of Skyrme. In~\cite{KLS}, the author, together with Kenig and Schlag, considered $(1+3)$-dimensional wave maps exterior to the unit,  ball taking values in $\Sp^3$, and  with a Dirichlet boundary condition on the boundary of the ball -- a model introduced in~\cite{Biz}. There, it was shown that any finite energy data leads to a global and smooth solution which  relaxes to the unique harmonic map in its degree class as $t \to \pm \infty$. This extended the result in~\cite{LS}, which established the result in the case of degree $0$, or topologically trivial, initial data. %Cutting out the unit ball removed the supercritical nature of the problem, making the analysis significantly easier.    
  
Here we consider another semi-linear modification of the Skyrme model that was introduced by Adkins and Nappi,~\cite{AN} in the mid $80$'s. We seek to extend the techniques developed in~\cite{KLS} to Adkins-Nappi wave maps and begin the analysis of large data dynamics for this model. Before stating the main results, we provide a brief introduction to the Adkins-Nappi formulation. We also refer the reader to the recent papers, \cite{GNR, GR10b, GR10a}, for an excellent introduction which includes  physical motivation. 

\subsection{Adkins Nappi Wave Maps}
Consider maps $U: (\R^{1+3}, \eta) \to (\Sp^3, g)$  where $\eta$ is the Minkowski metric and $g$ is the round metric on $\Sp^3$. %Denote by $S_{\al \be} = g_{ij} \p_{\al} U^i \p_{\be}U^j$ the pull-back metric. 
Let $A = A_{\al} dx^{\al}$ be a $1$-form (or gauge potential) and denote by $F_{\al \be} := \p_{\al}A_{\be}- \p_{\be} A_{\al}$ the associated curvature (or electromagnetic field). Adkins-Nappi wave maps are defined as formal critical  points of the Lagrangian
\EQ{ \label{AN Lag}
\LL(U, A) &=  \frac{ 1}{2}\int_{R^{1+3}} \eta^{\al \be} \ang{ \p_{\al} U, \p_{\be} U}_g \, dx\, dt, + \frac{1}{4} \int_{\R^{1+3}} F_{\al \be} F^{\al \be} \, dt \, dx \\
& \quad - \int_{\R^{1+3}} A_{\al} j^{\al} \, dt \, dx,
}
where the flux, or baryonic current, $j$ is defined by 
\ant{
j^{\al} = c \epsilon^{\al \be \ga \de}(\eta) \p_{\be} U^i \p_{\ga} U^j \p_{\de} U^{k} \epsilon_{i j k}(g)
}
for an appropriate normalizing constant $c$. Here $\epsilon$ is the Levi-Civita symbol, i.e.,  
\ant{
&\epsilon_{ijk}(g) = \sqrt{\abs{g}}  [ i, j, k],\\
&\epsilon^{\al \be \ga \de}(\eta) = \frac{1}{ \sqrt{\abs{ \eta}} }[ \al, \be, \ga, \de],
}
and the notation $[a, b, c]$ is defined as
\ant{
[a, b, c]:= \begin{cases} +1 \mif abc \textrm{ is an even permutation of 123,}\\ -1 \mif abc \, \,  \textrm{is an odd permutation of 123,} \\\ 0  \, \, \textrm{otherwise.} \end{cases}
}

 Note that~\eqref{AN Lag} is a generalization of the Lagrangian for wave maps, where here we have now introduced a coupling between maps $U: \R^{1+3} \to \Sp^3$ and $1$-forms $A$ on $\R^{1+3}$. Following \cite{GNR, GR10b, GR10a}, we consider only a restricted class of maps $U$ and forms $A$, namely we make an equivariance assumption. Let $(t, r, \theta, \phi)$ be polar coordinates on $\R^{1+3}$. The metric $\eta$ takes the form 
\ant{
(\eta_{\al \be}) =  \textrm{diag}(-1, 1,  r^2,  r^2 \sin^2 \theta).
}
Now, consider spherical coordinates on $\Sp^3 \subset \R^4$, 
\ant{
 ( \psi, \theta, \phi) \mapsto (  \sin \psi \sin \theta \sin \phi, \sin \psi \sin \theta \cos \phi, \sin \psi \cos \theta, \cos \psi),
 }
 with the metric $g$ in these coordinates given by 
 \ant{
 (g_{ij}) =  \textrm{diag}( 1, \sin^2 \psi, \sin^2 \psi \sin^2 \theta).
 }
Restricting to $1$-equivaraint maps means that we require $U$ to commute with the action of rotation on both the domain and target. This allows us to  make the standard $1$-equivariant ansatz for $U$,  viz., 
\EQ{
U(t, r, \theta, \phi) =  ( \psi(t, r), \theta, \phi),
}
and we assume the boundary condition $\psi(t, 0) = 0$ for all $t$. We also assume that the $1$-form $A$ satisfies  
\EQ{ \label{A ans}
A(t, r, \theta, \phi) = ( V(t, r), 0, 0, 0).
}
With these assumptions we have 
\EQ{
&\eta^{\al \be} g_{ij}(U) \p_\al U^i \p_\be U^j = - \psi_t^2 + \psi_r^2 + 2 r^{-2} \sin ^2 \psi, \quad F_{\al \be} F^{\al \be} = - 2 V_{r}^2,\\
&A_{\al} j^{\al} = V j^0  = 6c V \sin^2 \psi \psi_r r^{-2} = 3c r^{-2}V \p_r( \psi- \frac{1}{2}\sin2 \psi).
}
Hence if we restrict ourselves to $1$-equivariant maps $U$ and $1$-forms $A$ as in \eqref{A ans} the Lagrangian reduces to 
\EQ{ \notag
\frac{1}{ \pi} \LL(\psi, V) &=% \frac{1}{2} \int_{\R}\int_0^{\infty}  \left( -\psi_t^2 + \psi_r^2 + 2\frac{ \sin^2 \psi}{r^2} \right) \, r^2 \, dr \, dt\\
%& \quad- \frac{1}{2} \int_{\R} \int_0^{\infty} V_r^2 \, r^2 \, dr \, dt + \int_{\R} \int_0^{\infty} 3c\,  V \, \p_r( \psi- \frac{1}{2} \sin 2 \psi) \, dr\, dt \\
%&=
 \frac{1}{2}\int_{\R}\int_0^{\infty}  \left( -\psi_t^2 + \psi_r^2 + 2\frac{ \sin^2 \psi}{r^2} \right) \, r^2 \, dr \, dt\\
& \quad- \frac{1}{2} \int_{\R} \int_0^{\infty} V_r^2 \, r^2 \, dr  \, dt- 3c \int_{\R}\int_0^{\infty} V_r ( \psi- \sin \psi \cos \psi) \, dr \, dt,
}
which we can rewrite conveniently as 
 \EQ{ \notag
\frac{1}{ \pi}   \LL( \psi, V) &= \frac{1}{2} \int_{\R}\int_0^{\infty}  \left( -\psi_t^2 + \psi_r^2 + 2\frac{ \sin^2 \psi}{r^2}  + \al^2 \frac{ (\psi- \sin \psi \cos \psi)^2}{r^4}\right) \, r^2 \, dr \, dt \\
&\quad -  \frac{1}{2}\int_{\R} \int_0^{\infty} \left( r V_r +  \al  \frac{ ( \psi- \sin \psi \cos \psi)}{r} \right)^2 \, dr \, dt,
}
where $\al = 3c$. Next, observe that if $( \psi, V)$ is a stationary point for $\LL$ then we would have 
\ant{
0=\frac{d}{d \e} \vert_{\e=0} \LL( \psi, V+ \e W) = -C\int_{ \R} \int_0^{\infty}  \p_r\left(r^2V_r + \al ( \psi- \sin \psi \cos \psi) \right) W \, dr \, dt.
}
From this we can deduce that 
\EQ{ \label{V eqn}
r^2V_r + \al  (\psi- \sin \psi \cos \psi) = 0
}
for any stationary point $(\psi, V)$. Indeed the variational equation for $V$ is given by 
\ant{
 \p_r(r^2 V_r) = -  2\al \sin^2 \psi \psi_r.
}
%\Red{Contrast this with what is in the paper by Dolbeault, in particular equation (10) there}
This leads to a decoupling of the Euler-Lagrange equations for $\psi$ and $V$. Indeed, after rescaling the coordinates~$(t, r) \mapsto ( \al t, \al r)$, we can use \eqref{V eqn} to obtain the Euler-Lagrange equation for the azimuth angle $\psi$, which we formulate with Cauchy data:
\EQ{\label{an eq}
&\psi_{tt} -\psi_{rr} -\frac{2}{r} \psi_r + \frac{\sin(2 \psi)}{r^2} + \frac{( \psi - \sin\psi \cos\psi)(1-\cos2 \psi)}{r^4} = 0,\\
& \vec \psi(0) = ( \psi_0, \psi_1).
}
We  will use the notation $\vec \psi(t):=( \psi(t), \psi_t(t))$. The conserved energy is 
\EQ{
\E( \vec \psi)(t) = \int_0^{\infty}\left[ \frac{1}{2}( \psi_t^2 + \psi_r^2) + \frac{\sin^2 \psi}{r^2} + \frac{(\psi- \sin\psi \cos\psi)^2}{2r^4} \right] \, r^2 \, dr.
}
\begin{rem}
Observe that finite energy and continuous dependence on the initial data imply that for each $t \in I$ where $I \ni 0$ is an interval on which the solution exists, we have $\psi(t, 0) = 0$ and $\lim_{r \to \infty} \psi(t, r) =  n \pi$ for a fixed integer $ n \in \Z$. We  refer to $n$ as the {\em degree} of the map. We denote by $\E_{ n}$ the space of all finite energy data of degree $n$
\EQ{
\E_n := \{ (\psi_0, \psi_1) \mid \, \E (\vec \psi) <\infty, \, \, \psi_0(0)=0, \, \, \psi_0( \infty) = n \pi\}.
}
\end{rem}
\begin{rem}
We also remark that  the equation~\eqref{an eq} is not invariant under any scaling. However, as we will see in Sections~\ref{5d reduction} and~\ref{small scale} the underlying scale-invariant equation is indeed energy super-critical.  In fact, in order to formulate a small data scattering theory for degree zero Cauchy data, i.e., $\vec \psi(0) \in \E_0$,  we are forced to work, not in the energy space, but rather  %the right space in which to consider degree zero Cauchy data, i.e., $\vec \psi(0) \in \E_0$  is 
in the inhomogeneous Sobolev space 
\EQ{
\HH:= \{ (\psi_0, \psi_1) \in \vec \E_0 \mid, \vec \psi(0) \in (\dot{H}^2 \times \dot H^1) \cap (\dot{H}^1 \times L^2) \}
}
endowed with the norm
\EQ{
\|(\psi_0, \psi_1) \|_{\HH} &:= \|(\psi_0, \psi_1)\|_{(\dot{H}^2 \times \dot H^1) \cap (\dot{H}^1 \times L^2)( \R^3)}\\
& = \| (u_0, u_1) \|_{ \dot H^1 \times L^2 (\R^3)}  +  \|(u_{0, r}, u_{1, r})\|_{\dot{H}^1 \times L^2(\R^3)}.
}
We again refer the reader to Sections~\ref{5d reduction} and \ref{small scale} for a detailed analysis on why the $\dot{H}^2 \times \dot{H}^1$ norm is critical for this problem. The main drawback here is that while the conserved energy controls the $\dot{H}^1 \times L^2$ norm of the evolution, it is not known a priori if the $\dot{H}^2 \times \dot{H}^1$ norm of the solution should remain bounded. We are thus forced to make this assumption a prioi in order to study large data dynamics, which is why we refer to  the results in this paper as ``conditional."
%The reason for working with data in this space will be explained further in the next section. %Before stating the main theorem we remark that 
\end{rem}

\begin{rem}
In this paper we will focus our attention on the case of degree zero initial data, i.e., $\vec \psi(0) \in \E_0$. The analysis is simpler in this class of topologically trivial data as one does not need linearize about a nontrivial stationary solution to establish a perturbative small data theory, and we are thus free to focus directly on the question of the asymptotic behavior of large initial data.  
\end{rem}
We can now state the main theorem. 

\begin{thm}\label{main} Let $( \psi_0, \psi_1) \in \E_0$. Then, there is a unique solution $\vec \psi(t) \in \E_0$ to \eqref{an eq} with initial data $\vec \psi(0) = ( \psi_0, \psi_1)$ defined on its maximal interval of existence $0 \in I_{\max} = (T_-, T_+)$. Assume in addition that we have 
\EQ{
\sup_{t \in [0, T_+)} \| \vec \psi(t) \|_{\HH} < \infty. 
}
Then, in fact $T_+ = + \infty$, i.e., $\vec\psi(t)$ is defined globally for positive times. Moreover, $\vec \psi(t)$ scatters to zero in $\HH$ as $t \to  \infty$. The corresponding statement holds for negative times as well. 
 \end{thm}

\begin{rem}
The statement ``$\vec \psi(t)$ scatters to zero in $\HH$ as $t \to \pm \infty$" means that there exists a  solution $\vec \fy_{L}^\pm(t)$ to the linearized wave equation: 
 \EQ{ \label{linear wave}
 \fy_{tt} - \fy_{rr}- \frac{2}{r} \fy_r + \frac{2}{r^2} \fy = 0
 }
so that 
\EQ{
\| \vec \psi(t) - \vec \fy_{L}^\pm(t)\|_{\HH} \to 0 \mas t \to \pm \infty
}
\end{rem}

\begin{rem} Theorem~\ref{main} rules out what is referred to as type-II blow-up for~\eqref{an eq}. The contrapositive statement of the theorem says that if a degree zero Adkins-Nappi wave map blows-up in finite time ($T_+ < \infty$), or if $T_+ = \infty$ but the map does not scatter, then the norm $\| \vec \psi(t) \|_{\HH} \to \infty$ as $t \to T_+$ and both of these situations are referred to as type-I behavior. This type of result has been recently established for the semi-linear wave equation with power-type nonlinearities in focusing subcritical problems, \cite{Shen}, for defocusing super-critical problems, \cite{KM11a, KM11b, KV11a, KV11b,  Bul12a, Bul12b}, and recently for the $3d$ focusing super-critical wave equation,~\cite{DKM5}.  %For work on power-type Schr\"odinger equations see~\cite{KM10, KV11a, KV11b, Mur12a, Mur12b}. 
Here we are also in a super-critical situation but the nonlinearity has both a focusing piece and a defocusing piece as we shall see in Section~\ref{5d reduction}.  
\end{rem}

 The proof of Theorem~\ref{main} follows the concentration compactness/rigidity method introduced by Kenig and Merle in~\cite{KM06, KM08}. 
The argument is performed after reducing~\eqref{an eq} to a $5d$ equation by the substitution $r u(t, r):= \psi(t, r)$. The details of this reduction is outlined in Section~\ref{5d reduction}, but here we simply note that the equation for $u$ is of the form 
\EQ{\label{u eq1}
u_{tt} - u_{rr} - \frac{4}{r} u_r + Z_1( \psi) u^3 + Z_2( \psi) u^5 =0,
} 
where $Z_1, Z_2$ are bounded functions. The nonlinearity has both cubic and quintic powers, both of which  are energy super-critical due to the fact that the equation lives in $1+5$ dimensions. This facet of the equation requires extra attention, and indeed two-types of profiles emerge in the concentration compactness argument. We call the first type of profile ``Euclidean" and these correspond to subsequences of initial data which concentrate on very small scales on which the underlying quintic semi-linear equation 
\ant{
v_{tt} -v_{rr} - \frac{4}{r} v_r  +\frac{4}{3} v^5=0
} 
serves as a good approximation to~\eqref{u eq1}, see Section~\ref{small scale} and Section~\ref{nonlinear profiles}. The other profiles emerge from subsequences of data that live at a fixed scale, and we call these profiles  Adkins-Nappi as they correspond to genuine solutions of~\eqref{u eq1}. 

The proof of Theorem~\ref{main} begins in earnest in Section~\ref{CE}. The concentration compactness argument and construction of the critical element follow roughly the procedure introduced in~\cite{KM10} with additional arguments to handle the appearance of both Euclidean and Adkins-Nappi profiles in the nonlinear profile decomposition. 

Finally, we note that the rigidity argument, where one rules out the possibility of nonzero solutions with pre-compact trajectory (a property enjoyed by the critical element), follows the techniques introduced in~\cite{KLS}. This new type of argument, based on the ``channels of energy" method introduced in the seminal works~\cite{DKM1, DKM2, DKM3} and especially in~\cite{DKM4, DKM5}, avoids using any dynamical identities or inequalities, such as virial or Morawetz, which rely heavily on the specific structure of the nonlinearity, and are in general poorly suited to the complicated nonlinearities in geometric equations.  The method used here relies on specific information about the underlying elliptic equation proved in Section~\ref{stat sols} and on the exterior energy estimates for the underlying free radial wave equation proved in~\cite{KLS}, see Proposition~\ref{linear prop}. 

%%%%%%%%%%%%%%%%%%%%%%%%%%%%%%%%%%%%%%%%%%%%%%%%%%%%%%%%%%%%%%%%%%%%%%%%
%%%%%%%%%%%%%%%%%%%%%%%%%%%%%%%%%%%%%%%%%%%%%%%%%%%%%%%%%%%%%%%%%%%%%%%%
%%%%%%%%%%%%%%%%%%%%%%%%%%%%%%%%%%%%%%%%%%%%%%%%%%%%%%%%%%%%%%%%%%%%%%%%

\section{Preliminaries}
In this section we will establish a few simple facts about finite energy solutions to~\eqref{an eq}. 
We first note that any finite energy solution is necessarily bounded. 
\begin{lem}\cite{GNR}
All finite energy solutions  $\vec \psi(t) $ to \eqref{an eq} satisfy the bound 
\EQ{
\abs{\psi(t, r)} \le C( \E(\vec \psi)),
}
where the function $C:(0, \infty) \to (0, \infty)$, satisfies  $C(\rho) \to 0$ as $\rho \to 0$. 
\end{lem}
\begin{proof}
To see this we introduce the function 
\ant{
G( \rho):= \int_0^{\rho}( \theta - \sin \theta \cos \theta) \, d \theta = \frac{1}{2}( \rho^2- \sin^2 \rho).
}
Note that $G(0)=0$, $G( \rho) >0$ for $\rho \neq 0$ and $G( \rho) \to \infty$ as $\abs{\rho} \to \infty$. For a solution $\vec \psi(t, r) \in \E_n$ to \eqref{an eq} we then have 
\ant{
&G( \psi(t, r))= G( \psi(t,r)) - G( \psi(t, 0))= \\
&=  \int_{\psi(t, 0)}^{\psi(t, r)} ( \theta - \sin \theta \cos\theta) \, d \theta= \int_0^r( \psi(t,r) - \sin \psi(t, r) \cos\psi(t, r) ) \psi_r(t, r) \, dr\\
& \le \left( \int_0^r \psi_r^2(t, r) \, r^2\, dr\right)^{\frac{1}{2}} \left( \int_0^r  \frac{( \psi(t, r)- \sin\psi(t,r) \cos \psi(t, r))^2}{r^4} \, r^2 \, dr \right)^{\frac{1}{2}}\\
& \le  \E_0^r( \vec \psi(t)) \le \E(\vec \psi).
}
This gives the desired boundedness. 
\end{proof}
Next we move to the study of stationary solutions to~\eqref{an eq} which vanish as $r \to \infty$. 
\subsection{Stationary Solutions}\label{stat sols}
In this subsection we discuss the underlying elliptic theory for the Adkins-Nappi wave maps equation,~\eqref{an eq}. That is, we study solutions $\fy$ to the ODE
\EQ{\label{ode}
&\fy_{rr} +\frac{2}{r}\fy_r = \frac{\sin(2 \fy)}{r^2} + \frac{( \fy - \sin \fy\cos \fy)(1-\cos2 \fy)}{r^4}.
}
As this paper concerns the case of degree zero maps, we restrict our attention to those solutions that vanish at $r = \infty$, i.e., 
\EQ{ \label{go to 0}
\fy(r) \to 0 \mas r \to \infty.
}
We will show that the zero solution $\fy \equiv 0$ is the only {\em finite energy} solution to \eqref{ode} that also satisfies~\eqref{go to 0}. However, in the analysis in later sections we will need to know a bit more about solutions to~\eqref{ode}   satisfying ~\eqref{go to 0}. In particular, we will make use of a $1$-parameter family of solutions given by the following elementary result. 
\begin{prop}\label{ode prop}
For every $\al \in \R$ there exists a solution, $\fy_{\al}$ to~\eqref{ode} defined on the interval $r \in (0, \infty)$ such that 
\EQ{
\fy_{\al}(r)  = \al r^{-2} + O(r^{-6}) \mas r \to \infty. \label{fy at inf}
}
The $O( \cdot)$ vanishes for $\al = 0$. Moreover, of the family of solutions $\fy_{\al}$, the solution $\fy_0 \equiv 0$ corresponding to $\al =0$ is the only such solution to~\eqref{ode} which also vanishes at $r=0$, making it the unique solution in this family which also lies in the energy class~$\E_0$. 
\end{prop}

\begin{proof}
We find it convenient to first perform the following change of variables. Set $s = \log r$ and $\phi(s) := \fy(r)$. Then~\eqref{ode} can be written as an equation for $\phi$, viz., 
\EQ{\label{phi eq}
\phi_{ss} + \phi_s - 2 \phi = (\sin(2 \phi)- 2 \phi) +e^{-2s}( \phi - \sin\phi \cos \phi)(1- \cos 2\phi),
}
where above we have moved the linear term in $\sin2 \phi$ over to the left-hand-side. In this formulation it is apparent that all solutions $\phi(s)$ to \eqref{phi eq} are defined globally for $s \in (- \infty, \infty)$ due to the linear in $\phi$ growth of the nonlinearity above. 
The existence and uniqueness statements will be proved via a standard iteration argument. We sketch the details below. In order to remove the $\phi_s$ term we do another change of variables. We define 
\EQ{\label{g def}
g(s):= e^{\frac{s}{2}} \phi(s).
}
The equation~\eqref{phi eq} reduces to an equation for $g$, namely, 
\EQ{
&g''-\frac{9}{4}g= N_1(s, g) + N_2(s, g),\\
&N_1(s, g):= e^{\frac{s}{2}}(\sin(2 e^{-\frac{s}{2}}g)- 2 e^{-\frac{s}{2}}g),\\
&N_2(s, g):= e^{-\frac{3}{2}s}( e^{-\frac{s}{2}}g- \frac{1}{2}\sin 2 e^{-\frac{s}{2}}g)(1- \cos 2e^{-\frac{s}{2}}g).
}
The nonlinearities, $N_1(s, g), N_2(s, g)$ satisfy the estimates 
\EQ{\label{N1N2}
&\abs{N_1(s, g)} \lesssim e^{-s} \abs{g}^3,\\
&\abs{N_2(s, g)} \lesssim e^{-4s}\abs{g}^5.
}
Next, consider the fundamental system, 
\EQ{
f_1(s) = e^{-\frac{3}{2}s},  \, \, \, 
f_2(s) = e^{\frac{3}{2}s},
}
to the underlying linear equation 
\EQ{
f'' - \frac{9}{4}f = 0.
}
As the Wronskian $W(f_1, f_2):=f_1'f_2 - f_2'f_1 = -3$, for each $\al \in \R$ we seek a solution $g_{\al}$ to the integral equation 
\EQ{\label{int eq}
g_{\al}(s) = \al f_1(s) + \frac{1}{3} \int_s^{\infty}[f_1(s)f_2(t) - f_2(s)f_1(t)](N_1(t,g_{\al})+ N_2(t, g_{\al}))\, dt,
}
on the interval $s_0 \le s \le \infty$ for some $s_0>0$. This is achieved via successive approximations. For each $k \in \N$ we set 
\ant{
&g_{\al, 0}(s):= f_{1}(s),\\
&g_{\al, k}(s):= \al f_1(s) + \frac{1}{3} \int_s^{\infty}[f_1(s)f_2(t) - f_2(s)f_1(t)](N_1(t,g_{\al, k-1})+ N_2(t, g_{\al, k-1}))\, dt.
}
One can show via a standard iterative argument that there exists $s_0>0$ so that $g_{\al, k}(s)$ converges uniformly on the interval $[s_0, \infty)$ to a function $g_{\al}(s)$ which satisfies~\eqref{int eq}. Moreover, we have 
\EQ{
g_{\al}(s) = \al e^{-\frac{3}{2} s} + O(e^{-\frac{11}{2} s}) \mas s \to \infty.
}
By~\eqref{g def} this means that for each $\al \in \R$ there exists $\phi_{\al}$, solving~\eqref{phi eq} and satisfying 
\EQ{
\phi_{\al}(s) = \al e^{-2s} +O(e^{-6s}) \mas s \to \infty,
}
which is exactly~\eqref{fy at inf} after rephrasing the above in terms of $r = e^{s}$ and $\fy(r) = \phi(s)$. 

It remains to prove that $\fy_0(r) \equiv 0$ is the only solution to~\eqref{ode} that  satisfies~\eqref{fy at inf} {\em and} vanishes at $r = 0$. We phrase this as a lemma. 
\begin{lem}\label{0 lem}Let $\fy_{\al}$ solve~\eqref{ode} and satisfy~\eqref{fy at inf}. Suppose in addition that 
\EQ{\label{fy lims}
&\lim_{r \to 0} \fy_{\al}(r) = 0,\\
&\lim_{r \to 0} \fy'_{\al}(r) = \beta \in \R,
}
and we can thus define $\fy_\al(0) = 0$. Then $\fy_{\al} = \fy_0 \equiv 0$. 
\end{lem}
\begin{proof}[Proof of Lemma~\ref{0 lem}]
We proceed via contradiction. Let $\fy(r) = \fy_{\al}(r)$ solve~\eqref{ode} and satisfy~\eqref{fy at inf} and suppose in addition that that $\fy(r) \not\equiv 0$ and that $$\lim_{r \to 0} \fy(r) = 0.$$ 
Define the auxiliary function 
\EQ{ \label{P def}
P(r) = r^2 \fy_r^2  - 2 \sin^2 \fy - r^{-2}( \fy - \sin \fy \cos \fy)^2.
}
Using that $\fy$ solves~\eqref{ode} one can readily show that 
\ant{
P'(r) = -2r \fy_r^2 +2r^{-3}( \fy - \sin \fy \cos \fy)^2.
}
Thus 
\ant{
\p_r(r^2 P(r)) = r^2 P'(r) + 2r P(r) = -4r \sin^2 \fy.
}
We now set $\Phi(r):= r^2P(r)$. We have shown that 
\ant{
&\Phi'(r) = -4r \sin^2 \fy.
}
By~\eqref{fy lims} it is apparent that $\Phi(0) = 0$, and we thus obtain the expression 
\ant{
\Phi(r) = -4\int_0^r \rho \sin^2 \fy(\rho) \, d \rho <0 \mfor r >0.
}
Since $\fy(r) \not \equiv 0$ satisfies~\eqref{fy at inf} the limit
\EQ{\label{neg lim}
\lim_{r \to \infty} \Phi(r) = - 4\int_0^\infty \rho \sin^2 \fy(\rho) \, d \rho <0,
} 
exists. On the other hand, combining \eqref{P def} and~\eqref{fy at inf} we have 
\ant{
\lim_{r \to \I}\Phi(r) = \lim_{r \to \infty} (r^4 \fy_r^2  - 2 r^2\sin^2 \fy - ( \fy - \sin \fy \cos \fy)^2) = 0,
}
which contradicts~\eqref{neg lim}. 
\end{proof}
This also completes the proof of Proposition~\ref{ode prop}. 
\end{proof}

\section{Small data theory}
\subsection{$5d$ Reduction}\label{5d reduction}
In what follows we will require a version of Hardy's inequality for radial functions which we state now for convenience. For the proof, we refer the reader to \cite{STZ94}. 
\begin{lem}[Hardy's Inequality] \cite[Lemma $1.2$]{STZ94} \label{Hardy}Let $d \ge 4$. There exists a constant $C>0$ so that for all radial functions $v \in \dot W^{s, q}(\R^d)$, with $1 \le q \le s q <d$  and for all $p$ with $q \le p \le \infty$ we have 
\ant{
\| r^{\frac{d}{q} - \frac{d}{p} - s}v \|_{L^p( \R^d)} \le C \|v \|_{ \dot{W}^{s, q}(\R^d)}.
}
\end{lem}

For solutions $\psi(t) \in \E_0$, i.e., those that satisfy $\psi(t, 0) = 0$ and $\psi(t, \infty) = 0$ for all $t \in I$, it is convenient to pass to a $5d$ semi-linear equation via the standard substitution  $r u(t, r):= \psi(t, r)$. The motivation behind this substitution lies with the fact that there is a linear term present in the nonlinearity in~\eqref{an eq}. Indeed,
\ant{
\frac{ \sin 2\psi}{ r^2}  + \frac{  (\psi - \frac{1}{2}\sin2 \psi)( 2 \sin^2 \psi)}{r^4}&= \frac{2 \psi}{r^2} + \frac{ \sin 2 \psi - 2 \psi}{r^2}  + \frac{O( \psi^5)}{r^4}\\
&= \frac{2 \psi}{ r^2} + \frac{O(\psi^3)}{r^2} + \frac{ O(\psi^5)}{r^4}.
}
The presence of the strong repulsive potential term $\frac{2}{r^2}$ indicates that the linearized operator of~\eqref{an eq} has more dispersion than the $3$-dimensional wave equation. In fact, it has the same dispersion as the $5$-dimensional wave equation since the equation for $u(t, r)$ defined by  $\psi = r u$ is given by 
\EQ{ \label{u eq}
&u_{tt} - u_{rr} - \frac{4}{r} u_r + Z_1( \psi) u^3 + Z_2( \psi) u^5 =0,\\
& \vec u(0) = (u_0, u_1), \\
&Z_1( \rho):= \frac{\sin 2 \rho - 2 \rho}{ \rho^3}, \\
&Z_2( \rho):= \frac{(\rho - \sin \rho \cos \rho)(1- \cos2 \rho)}{\rho^5},
}
where $ru_0(r) = \psi_0(r)$ and $ru_1(r) = \psi_1(r)$. We claim that for degree zero data, i.e.,  $\vec \psi(0) \in \E_0$, it suffices to study the Cauchy problem in the $u$-formulation. In fact, for the remainder of the paper we will deal exclusively with $\vec u(t, r)$ in the $5$-dimensional formulation rather than with the equivariant azimuth angle $\psi(t, r)$.  Indeed using Hardy's inequality and the relations 
\ant{
&\psi_r = ru_r + u = r u_r + r^{-1} \psi,\\
&\psi_{rr} = r u_{rr} + 2u_r = r u_{rr} + 2r^{-1} \psi_r - 2 r^{-2} \psi,
}
 one can show that for, e.g., $s= 1, 2$,  the map 
\ant{
 \dot{H}^s_{\textrm{rad}}(\R^5)  \ni u \mapsto \psi = r u \in  \dot{H}^s_{\textrm{rad}}(\R^3)
} 
is an isomorphism and we have 
\EQ{
\|\vec u\|_{\dot{H}^s \times \dot{H}^{s-1}(\R^5)} \simeq \| \vec \psi \|_{\dot{H}^s \times \dot{H}^{s-1}(\R^3)}.
}
In fact the map, 
\ant{
\vec u = (u_0, u_1) \mapsto ( r u_0, r u_1) =: \vec \psi
}
is an isomorphism between the spaces $\HH(\R^5)$ and $\HH( \R^3)$, where 
\ant{ 
\|(u_0, u_1) \|_{\HH(\R^5)} := \|(u_0, u_1)\|_{(\dot{H}^2 \times \dot H^1) \cap (\dot{H}^1 \times L^2)( \R^5)}.
}
We also note that the energy of degree zero data $(\psi_0, \psi_1)$ is controlled by the $\HH(\R^5)$ norm of $(u_0,  u_1)$ defined by  $(ru_0(r),  ru_1(r)) :=( \psi_0, \psi_1)$. Indeed,
\EQ{\label{finite energy in HH}
\E( \vec \psi) & \lesssim\int_{0}^{\infty} \left(\psi_1^2 + (\p_r \psi_0)^2 + \frac{\psi_0^2}{r^2} + \frac{ \psi_0^6}{r^4} \right) \, r^2 \, dr\\
&\lesssim \left(u_1^2 + (\p_ru_0)^2 + \frac{u_0^2}{r^2} + u_0^6 \right) \, r^4 \, dr \lesssim \| (u_0, u_1)\|_{\HH}^2
}
where in the final inequality we have used Hardy's inequality as well as interpolation and Sobolev embedding in $\dot{H}^2 \cap \dot{H}^1(\R^5) \hookrightarrow \dot{H}^{\frac{5}{3}} \hookrightarrow L^{6}( \R^5)$. Therefore
\EQ{\label{HH to E0}
\vec u \in \HH \Rightarrow \vec \psi = (ru_0, ru_1) \in \E_0
}
In what follows we will also require the simple uniform estimates for the functions $Z_1, Z_2$ appearing in the nonlinearity of~\eqref{u eq} as observed in \cite{GNR}, namely, 
\EQ{ \label{Z bounds}
&\abs{Z_1( \rho)} \lesssim \ang{\rho}^{-2},\\
&\abs{ \p_{\rho}^{1+j} Z_1( \rho)} \lesssim \ang{ \rho}^{-3},\\
&\abs{\p_{\rho}^{j} Z_2( \rho)} \lesssim \ang{ \rho}^{-4},
}
which hold for all $j \ge 0$ and with $\ang{\rho} = \sqrt{1 + \rho^2}$. 

We can thus reformulate Theorem~\ref{main} in terms of the $5d$ $u$-formulation as the two results are equivalent.

\begin{thm}\label{u main} Let $( u_0, u_1) \in \HH( \R^5)$. Then, there is a unique solution $\vec u(t) \in \HH$ to~\eqref{u eq} with initial data $\vec u(0) = ( u_0, u_1)$, defined on its maximal interval of existence $0 \in I_{\max} = (T_-, T_+)$. Assume in addition that we have 
\EQ{
\sup_{t \in [0, T_+)} \| \vec u(t) \|_{\HH} < \infty. 
}
Then, in fact $T_+ = + \infty$, i.e., $\vec u(t)$ is globally defined for positive times. Moreover, $\vec u(t)$ scatters to zero in $\HH$ as $t \to  \infty$. The corresponding statement holds for negative times  as well
 \end{thm}
\begin{rem}
The statement ``$\vec u(t)$ scatters to zero in $\HH$ as $t \to \pm \infty$" means that there exists a  solution $\vec v_{L}^\pm(t)$ to the free wave equation: 
 \EQ{ \label{free 5d}
 v_{tt} - v_{rr}- \frac{4}{r} v_r  = 0
 }
so that 
\EQ{
\| \vec u(t) - \vec v_{L}^\pm(t)\|_{\HH} \to 0 \mas t \to \pm \infty
}
\end{rem}

\subsection{Small data -- global existence, scattering, and perturbative theory}
An essential ingredient to  the small data theory are Strichartz estimates for the inhomogeneous wave equation in $\R^{1+5}$, which reads
\EQ{\label{lin wave}
&v_{tt}- \Delta v  = F,\\
&\vec v(0) = (f, g). 
}
A free wave will mean a solution to \eqref{lin wave} with $F=0$, and will be denoted by $\vec v(t) = S(t) \vec v(0)$. In what follows we say that $(p, q, \gamma)$ is an admissible triple if 
\begin{align*}
p, q \ge 2 , \quad \frac{1}{p}+ \frac{5}{q} =\frac{5}{2} - \gamma, \quad
 \frac{1}{p} + \frac{2}{q} \le 1 
\end{align*}
%We also say that a triple $(p, q, \gamma)$ is radial-admissible if it satisfies \eqref{pq2}, \eqref{scaling} along with the less restrictive condition
%\EQ{
%\frac{1}{p} + \frac{4}{q} < 2
%}
We can now state the Strichartz estimates that we will need below. 
\begin{prop}\cite{Kee-Tao, LinS}\label{strich} Let $(p, q, \gamma)$ and $(r, s, \rho)$ be admissible triples. Then any solution $\vec{v}(t)$ to~\eqref{lin wave} satifies
\EQ{
\| \abs{\nabla}^{- \gamma} \nabla v\|_{L^p_tL^q_x} \lesssim \|(f, g) \|_{\dot{H}^1 \times L^2} + \| \abs{\nabla}^{\rho} F\|_{L^{r'}_t L^{s'}_x}.
}
\end{prop}

%Sterbenz and Rodnianski proved the following refined version of the above Strichartz estimates in the case where the initial data $(f, g)$ and hence the corresponding evolution $\vec v(t)$ are radially symmetric functions: 

%\begin{lem}
%Let $(p, q, \gamma)$ and $(r, s, \rho)$ be radial-admissible triples. Then any solution $\vec{v}(t)$ to the free wave equation ~\eqref{lin wave} with $h  =0$ and radial initial data $(f, g)$ satisfies
%\EQ{
%\| \abs{\nabla}^{- \gamma} \nabla v\|_{L^p_tL^q_x} \lesssim \|(f, g) \|_{\dot{H}^1 \times L^2} 
%}
%\end{lem}

We can now formulate the local well-posedness theory for~\eqref{u eq}. This was previously established in \cite{GNR}, however here we use a different Strichartz norm for our scattering norm to facilitate the concentration compactness arguments later. We thus revisit the standard argument here.  For a time interval $0 \ni I \subset \R$ , define the norms $S(I)$ and $N(I)$ by
\EQ{
&\|u\|_{S(I)} := \|u\|_{L^3_t(I; L^{\frac{30}{7}} \cap \dot{W}^{1, \frac{30}{7}}(\R^5))},\\
&\|F\|_{N(I)} := \|F\|_{L^{\frac{3}{2}}_t(I; L^{\frac{30}{17}} \cap \dot{W}^{1, \frac{30}{17}}(\R^5))}.
} 
To simplify notation we will often write $X(I):= L^{\infty}_t(I; \HH) \cap S(I)$. 
We also recall the definitions of the homogeneous and inhomogeneous Besov spaces. 
 \ant{
& \|f\|_{\dot{B}^s_{p, q}} := \left( \sum_{j \in \Z} 2^{qsj} \|P_j f\|_{L^p}^q \right)^{\frac{1}{q}} \mif 1\le q < \infty\\
 & \|f\|_{B^s_{p, q}} := \|P_{ \le 0} f\|_{L^p} +  \left( \sum_{j \ge 1} 2^{qsj} \|P_j f\|_{L^p}^q \right)^{\frac{1}{q}} \mif 1\le q < \infty
 }
  and for $q = \infty$, 
  \ant{
 & \|f\|_{\dot{B}^s_{p, \infty}} :=  \sup_{ j \in \Z} 2^{sj} \|P_j f\|_{L^p} \\
 & \|f\|_{B^s_{p, q}} := \|P_{ \le 0} f\|_{L^p} +  \sup_{ j \in \Z} 2^{sj} \|P_j f\|_{L^p}
}
above the $P_k$ denote the usual Littlewood-Paley projections onto frequencies of size $ \abs{ \xi} \simeq 2^k$ and $P_{ \le 0}$ denotes projection onto frequencies $ \abs{ \xi} \lesssim 1$. 

\begin{prop}[Small data theory]\cite[Theorem~$2.2$]{GNR} \label{small data}
Let $\vec u(0) = (u_0, u_1) \in \HH (\R^5)$. Then there is a unique, solution $\vec u(t) \in \HH$ defined on a maximal interval of existence $I_{\max}( \vec u )= (T_-(\vec u), T_+( \vec u))$.  Moreover, for any compact interval $J \subset I_{\max}$ we have 
\ant{ 
\| u\|_{S(J)} < \infty.
}
Moreover, a globally defined solution $\vec u(t)$ for $t\in [0, \infty)$ scatters as $ t \to \infty$ to a free wave, i.e., a solution $\vec u_L(t) \in \HH$ of 
\ant{
\Box u_L =0
}
if and only if  $ \|u \|_{S([0, \infty))}< \infty$. In particular, there exists a constant $\de>0$ so that  
\EQ{ \label{global small}
 \| \vec u(0) \|_{\HH} < \de \Rightarrow  \| u\|_{S(\R)} \lesssim \|\vec u(0) \|_{\HH} \lesssim \de
 }
and hence $\vec u(t)$  scatters to free waves as $t \to \pm \infty$. Finally, we have the standard finite time blow-up criterion: 
\EQ{ \label{ftbuc}
T_+( \vec u)< \infty \Longrightarrow \|u \|_{S([0, T_+( \vec u)))} = + \infty
}
A similar statement holds if $- \infty< T_-( \vec u)$. 
\end{prop}

\begin{proof}
The proof follows from a standard contraction-mapping argument based on the Strichartz estimates of  Proposition~\ref{strich}. For completeness we show how the a priori global estimate~\eqref{global small} is derived. Using Proposition~\ref{strich} on the nonlinear equation for $u$ we have for any time interval $I$, 
\EQ{\label{nonlin strich}
\|u\|_{ S(I)} + \| \vec u(t)\|_{L^{\infty}_t(I, \HH)} \lesssim \|\vec u(0) \|_{\HH} + \|Z_1( \psi) u^3 + Z_2 ( \psi) u^5\|_{N(I)}.
}
Via the usual continuity argument (expanding $I$) it will suffice to control the nonlinearity in terms of higher powers of the norms on the left-hand side. Indeed, using \eqref{Z bounds} where appropriate,  we have 
\ant{
\|Z_1( \psi) u^3\|_{L^{\frac{30}{17}}_x} \le \|u\|_{L^{10}_x}\|u\|_{L^{\frac{30}{7}}_x}^2 \lesssim \|u\|_{\dot{H}^2} \|u\|_{L^{\frac{30}{7}}_x}^2.
}
The second inequality following from the Sobolev embedding $\dot{H}^2( \R^5) \hookrightarrow L^{10}(\R^5)$. Therefore,  
\ant{
\|Z_1( \psi) u^3\|_{L^{\frac{3}{2}}_t(I; L^{\frac{30}{17}}_x)} \lesssim \|u\|_{L^{\infty}_t(I; \dot{H}^2)} \|u\|_{L^3_t(I; L^{\frac{30}{7}}_x)}^2.
}
Next, again using Sobolev embedding we have 
\ant{
\|Z_2( \psi) u^5\|_{L^{\frac{30}{17}}_x} &\lesssim \|u\|^3_{L^{10}_x} \|u\|_{L^{\frac{15}{2}}_x}^2 \lesssim \|u\|_{\dot{H}^2}^3 \|u\|_{\dot{W}^{\frac{1}{2}, \frac{30}{7}}_x}^2\\
& \lesssim   \|u\|_{\dot{H}^2}^3  \|u\|_{ L^{\frac{30}{7}}_x} \|u\|_{ \dot W^{1, \frac{30}{7}}_x}\\
& \lesssim \|u\|_{\dot{H}^2}^3  (\|u\|_{ L^{\frac{30}{7}}_x}^2  + \|u\|_{ \dot W^{1, \frac{30}{7}}_x}^2),
}
where the second to last line follows from interpolation, i.e., $$ \dot{W}^{\frac{1}{2}, \frac{30}{7}} = ( L^{\frac{30}{7}}, \dot{W}^{1, \frac{30}{7}})_{\frac{1}{2}}.$$ 
As before we then have 
\ant{
\|Z_2( \psi) u^5\|_{L^{\frac{3}{2}}_t(I; L^{\frac{30}{17}}_x)} \lesssim \| u\|_{L^{\infty}_t(I; \dot{H}^2)}^3 \|u\|_{S(I)}^2.
}
Next we control the $\dot{W}^{1, \frac{30}{17}}$ norm of the nonlinearity. We have 
\ant{
\|Z_1(ru)u^3 + Z_2(ru)u^5\|_{\dot{W}^{1, \frac{30}{17}}} \simeq \| \p_r (Z_1(ru)u^3 + Z_2(ru)u^5)\|_{L^{\frac{30}{17}}}
}
We begin with the quintic term. Note  that 
\ant{
\p_r(Z_2(ru)u^5) = u^4 u_r(5 Z_2( \psi) + \psi Z_2'(\psi)) + u^4  \frac{u}{r} \psi Z_2'( \psi)}
Using~\eqref{Z bounds}, Hardy's inequality, and the Sobolev Embedding $\dot{H}^2(\R^5) \hookrightarrow \dot W^{1, \frac{10}{3}}(\R^5)$, we obtain 
\ant{
\|Z_2( ru)u^5 \|_{\dot{W}^{1, \frac{30}{17}}} &\lesssim \|u^4 u_r\|_{L^{\frac{30}{17}}} + \|u^4 \frac{u}{r}\|_{L^{\frac{30}{17}}} \\
& \lesssim \|u^4\|_{L^{\frac{15}{4}}}( \|u_r\|_{L^{\frac{10}{3}}} + \|r^{-1} u\|_{L^{\frac{10}{3}}})\\
&\lesssim \|u\|^4_{L^{15}} \|u\|_{\dot{H}^2}
}
Next, we combine the embeddings $\dot B^{\frac{5}{6}}_{\frac{30}{7}, 1} \hookrightarrow L^{15}$ and $ \dot{H}^2\hookrightarrow \dot{W}^{1, \frac{10}{3}} \hookrightarrow \dot{B}^{\frac{2}{3}}_{\frac{30}{7}, \infty}$ with the interpolation $$\dot B^{\frac{5}{6}}_{\frac{30}{7}, 1}  = (\dot B^1_{\frac{30}{7}, \infty}, \dot{B}^{\frac{2}{3}}_{\frac{30}{7}, \infty})_{\frac{1}{2}, 1},$$ to obtain
\ant{
\|u\|_{L^{15}} &\lesssim \|u\|_{\dot B^{\frac{5}{6}}_{\frac{30}{7}, 1}} \lesssim \|u\|^{\frac{1}{2}}_{\dot B^1_{\frac{30}{7}, \infty}} \|u\|^{\frac{1}{2}}_{\dot{B}^{\frac{2}{3}}_{\frac{30}{7}, \infty}} \\
&\lesssim \|u\|^{\frac{1}{2}}_{\dot W^{1, \frac{30}{7}}} \|u\|^{\frac{1}{2}}_{\dot{H}^2}.
}
Therefore we can conclude that 
\ant{
\|Z_2( ru)u^5 \|_{L^{\frac{3}{2}}_t(I; \dot{W}^{1, \frac{30}{17}}_x)}  &\lesssim \|u\|_{L^{\infty}_t(I;  \dot{H}^2_x)}^3 \|u\|_{L^{3}_t(I; \dot{W}^{1, \frac{30}{7}}_x)}^2\\
&\lesssim \|u\|_{L^{\infty}_t(I; \HH)}^3 \|u\|_{S(I)}^2 . 
}
Lastly we estimate the subcritical cubic term $Z_1(ru) u^3$. Proceeding as above we have 
\ant{
\|Z_1(ru) u^3\|_{\dot{W}^{1, \frac{30}{17}}} &\lesssim \|u^2\|_{L^{3}}\|u_r\|_{L^{\frac{30}{7}}} \\
&\lesssim \|u\|_{L^{10}} \|u\|_{L^{\frac{30}{7}}} \|u\|_{\dot{W}^{1, \frac{30}{7}}} \\
& \lesssim \|u\|_{\dot{H}^2}( \|u\|_{L^{\frac{30}{7}}}^2 + \|u\|_{\dot{W}^{1, \frac{30}{7}}}^2),
}
which implies 
\ant{
\|Z_1( ru)u^3 \|_{L^{\frac{3}{2}}_t(I; \dot{W}^{1, \frac{30}{17}}_x)}  \lesssim \|u\|_{L^{\infty}_t(I; \HH)} \|u\|_{S(I)}^2 . 
}
Putting this all together we have
\ant{
\|Z_1(ru)u^3 + Z_2(ru)u^5\|_{N(I)} \lec \left(\| u\|_{L^{\infty}_t(I; \HH)} + \| u\|_{L^{\infty}_t(I; \HH)}^3\right) \|u\|_{S(I)}^2.
}
Defining the space $X(I):= L^{\infty}_t(I; \HH) \cap S(I)$, and plugging the above estimate into \eqref{nonlin strich} we obtain the estimate
\ant{
\|u\|_{X(I)} \lec \| \vec u(0)\|_{\HH} + \|u\|_{X(I)}^3 + \|u\|_{X(I)}^5
} 
which is enough to finish the proof after the usual application of a continuity argument.
\end{proof}

\begin{lem}[Perturbation Lemma] \label{perturbation}
There are continuous functions $$\eps_0,C_0:(0,\I)\to(0,\I)$$ such that the following holds:
Let $I\subset \R$ be an open interval (possibly unbounded), $\vec u, \vec v\in C(I; \HH)$  satisfying for some $A>0$
%\EQ{ \label{asm ebd}
 % \|\vec u\|_{L^\I_t(I;\HH)} + \|\vec v\|_{L^\I_t(I;\HH)}  <\I , \pq \|v\|_{L^3_t(I;L^6_x)} \le B,}
\EQ{\nn
%\|\vec u\|_{L^\infty(I;\HH)} + 
 \|\vec v\|_{L^\infty(I;\HH)} +   \|v\|_{S(I)} & \le A \\
 \|\glei(u)\|_{N(I)}
   + \|\glei(v)\|_{N(I)} + \|w_0\|_{S(I)} &\le \eps \le \eps_0(A),
   }
where $\glei(u):=\Box u+Z_1(ru)u^3+Z_2(ru)u^5$ in the sense of distributions, and $\vec w_0(t):=S(t-t_0)(\vec u-\vec v)(t_0)$ with $t_0\in I$ arbitrary but fixed.  Then
\EQ{ \nn
  \|\vec u-\vec v-\vec w_0\|_{L^\I_t(I;\HH)}+\|u-v\|_{S(I)} \le C_0(A)\eps.}
  In particular,  $\|u\|_{S(I)}<\I$.
\end{lem}
%%%%%%%%%%%%%%%%%%%%%%%%
%%%%%%%%%%%%%%%%%%%%%%%%%
We also will use the following easy reformulation of the perturbation lemma, which we state as a corollary.

\begin{cor}\label{pert cor}
There are continuous functions $\eps_0,C_0:(0,\I)\to(0,\I)$ such that the following holds:
Let $I\subset \R$ be an open interval (possibly unbounded), $\vec v\in L^{\infty}(I; \HH)$  satisfying for some $A>0$
%Assume that $I$ is an open interval and $\vec{v}\in L^{\infty}(I; \HH)$ satisfies %the approximate equation
%\EQ{
%\Box v =Z_{1}(rv) v^3+ Z_2(rv)v^5+E, \quad \mathrm{on}~I
%}
%Assume also that 
\begin{equation}\label{bounded-approximation-norms}
\|v\|_{S(I)}+\|\vec{v}\|_{L^{\infty}(I; \HH)} \le A< \infty
\end{equation}
%Then there are continuous functions $\eps_0,C_0:(0,\I)\to(0,\I)$ such that for all $\e< \tilde{\e},$ if $\vec u(t_0)=(u_0, u_1)\in\HH$, and the error $E$ satisfy
In addition, assume that  
\EQ{\label{close-data}
\| \vec u(t_0)-\vec v(t_0)\|_{\HH}+\|\glei(v)\|_{N( I)} \leq \e \le \e_0(A),
}
for some $t_0\in I,$ where $\glei(u):=\Box u + Z_1(ru)u^3+Z_2(ru)u^5$ in the sense of distributions. Then there is a unique solution $\vec u \in L^{\infty}(I;  \HH)$ of
\begin{align*}
&\Box u+ Z_{1}(ru) u^3+ Z_2(ru)u^5 =0 \\
 &\vec u(t_0)= (u_0, u_1)
\end{align*}
Moreover, $\vec u$ satisfies 
\EQ{ \label{new u bounds}
%\|u\|_{S( I)}\le B_1,\nonumber\\
\|u-v\|_{S( I)}\le C(A)\e.
}
In particular, $\|u\|_{S( I)}< \infty$.
\end{cor}

\begin{proof}[Proof of the Lemma~\ref{perturbation}]  Let $X(I):= L^{\infty}_t( I; \HH) \cap S(I)$ and set 
\EQ{\nn
 &w:=u-v, \\
 &e:=\Box (u-v) +Z_1(ru)u^3 + Z_2(ru)u^5-Z_1(rv)v^3-Z_2(rv)v^5 = \glei(u) - \glei(v).}
There is a partition of the right half of $I$ as follows, where $\delta_{0}>0$ is a small
absolute constant which will be determined below:
\EQ{\nn
 \pt t_0<t_1<\cdots<t_n\le \I,\pq I_j=(t_j,t_{j+1}),\pq I\cap(t_0,\I)=(t_0,t_n),
 \pr \|v\|_{X(I_j)} \le \de_{0} \pq(j=0,\dots,n-1), \pq n\le C(A,\de_{0}).}
We omit the estimate on $I\cap(-\I,t_0)$ since it is the same by symmetry.
Let $\vec w_j(t):=S(t-t_j)\vec w(t_j)$ for all $0\le j <n$.  Then Duhamel's formula gives
\begin{multline}\label{w form}
\vec w(t)- \vec w_{0}(t)= \int_{t_0}^{t} S(t-s) (0,(e- (v+w)^{3}Z_1(r(v+w))+\\ + v^3Z_1(rv) -(v+w)^5Z_2(r(v+w)) + v^{5}Z_2(rv))(s)\, ds
\end{multline}
which implies that, for some absolute constant $C_{1}\ge1$,
\EQ{ \label{eq:ww0}
 \| w-w_{0}\|_{X(I_0)} &\lec \\
& \quad  \|v^3Z_1(rv)-(v+ w)^3Z_1(r(v+w))+\\
 &\quad+ v^5Z_2(r(v+w))-(v+w)^5Z_2(r(z+w)) +e\|_{N(I_0)}
 \\ &\le C_{1} ( \de_0^2 +\de_{0}^{4}+ \|w\|_{X(I_0)}^2+\| w\|_{X(I_0)}^{4})\| w\|_{X(I_0)}+C_{1}\eps
}
To estimate the differences involving the functions $Z_1, Z_2$ we use its smoothness along with the fact that by radiallity, $ru$ and $rv$ are bounded point-wise by the $L^{\infty}(I;\HH)$ norms of $u$ and $v$, respectively, which in turn are bounded by $A$, by assumption. Note that $\|w\|_{X(I_{0})}<\infty$ as long as $I_{0}$ is a finite interval. If $I_{0}$ is half-infinite, then we first replace it with an interval of the form $[t_{0},N)$, and  let $N\to\I$ after performing the estimates which are
uniform in~$N$.  Now we assume that $C_{1}( \de_0^2+ \delta_{0}^{4})\le \frac14$ and fix $\delta_{0}$ in this fashion.
By way of the continuity method (which refers to using that the $S$-norm is continuous in  the upper endpoint of $I_{0}$),
\eqref{eq:ww0} implies that $\| w\|_{X(I_{0})}\le 8C_{1}\eps$.
Furthermore, Duhamel's formula implies  that
\begin{multline}\nn
\vec w_{1}(t)- \vec w_{0}(t)= \int_{t_{0}}^{t_{1}} S(t-s) (0,(e- (v+w)^{3}Z_1(r(v+w))+ \\+ v^3Z_1(rv) -(v+w)^5Z_2(r(v+w)) +v^{5}Z_2(rv))(s)\, ds
\end{multline}
whence also
\begin{multline}\label{eq:w1w0}
\| w_{1}-w_{0}\|_{X(\R)} \lec \\
 \| (e- (v+w)^{3}Z_1(r(v+w)) + v^3Z_1(rv) -(v+w)^5Z_2(r(v+w)) +v^{5}Z_2(rv))(s)\|_{N(I_1)}
\end{multline}
which is estimated as in~\eqref{eq:ww0}.  We conclude that $ \| w_{1}\|_{X(\R)}\le 8C_{1}\eps$.
In a similar fashion one verifies that for all $0\le j<n$
\EQ{ \label{est S'}
 &\| w- w_j\|_{X(I_j)} + \| w_{j+1}-w_j\|_{X(\R)}\\
& \lec  \|  e- (v+w)^{3}Z_1(r(v+w)) + v^3Z_1(rv) -(v+w)^5Z_2(r(v+w)) +v^{5}Z_2(rv)  \|_{N(I_j)}\\
 &\le C_{1} ( \de_0^2+ \de_{0}^{4}+\|w\|_{X(I_j)}^2 + \| w\|_{X(I_j)}^{4})\| w\|_{X(I_j)}+C_{1}\eps}
 where $C_{1}\ge1$ is as above.
Inducting in $j$ one obtains that
 \EQ{\nn
 \| w\|_{X(I_{j})} + \| w_{j}\|_{X(\R)} \le C(j)\, \eps\quad \forall \; 1\le j<n
 }
 This requires that  $\eps<\eps_{0}(n)$ which can be done provided $\eps_0(A)$ is chosen small enough.
%Repeating the estimate~\eqref{est S'} once more,  but with the energy piece $L^{\I}_{t}\HH$ included on the left-hand side,
%we can now bound the $S(I)$-norm on~$w$.
\end{proof}

%%%%%%%%%%%%%%%%%%%%%%%%%
%%%%%%%%%%%%%%%%%%%%%%%%%
%\begin{proof}
%This follows from letting $v$ in Lemma~\ref{perturbation} be $\tilde{u}$ and $u$ be the local solution with $u[t_0]=\phi.$ The bound on $w_0$ follows from Proposition~\ref{strich} with $(p, q, \gamma) = (5, 10, 1)$ and $(a, b, \sigma) = ( \infty, 2, 0)$ on the local interval of existence. Since a solution exists as long as its $S-$norm is finite, Lemma \ref{perturbation} implies that $u$ can be extended as a solution to all of $I,$ and satisfies the bounds \eqref{new u bounds}.
%\end{proof}

\subsection{Small scale approximation}\label{small scale}
Although solutions to \eqref{u eq} are not invariant under any rescaling, a crucial ingredient in the concentration compactness argument in the next section will be the fact that for data which concentrate at very small scales, solutions to \eqref{u eq} are well approximated by solutions to an equation which is scaling invariant, namely 
\EQ{\label{defoc v eq}
& v_{tt} -v_{rr} - \frac{4}{r} v_r  +\frac{4}{3} v^5=0\\
&\vec v(0) = (v_0, v_1) \in \dot{H}^2 \times  \dot{H}^1( \R^5)
} 
The above is a \textit{de-focusing} quintic wave equation in $\R^{1+5}$, and we will be considering initial data $\vec v(0) \in \dot{H}^2 \times \dot H^1$. Note that we have included the constant $\frac{4}{3}$ in front of the quintic nonlinearity due the the Taylor expansion 
\ant{
 Z_2( \rho) = \frac{4}{3} + O( \rho^2)
 } 
 where $Z_2( \rho)$ is as in~\eqref{u eq}. The reason for this will become apparent below. We observe that if $\vec v(t)$ solves \eqref{defoc v eq} then for each $\la>0$ so does $\vec v_{\la}(t)$, which is defined as 
\EQ{
\vec v_{\la}(t, r) := \left( \frac{1}{ (\la)^{\frac{1}{2}}} \, v(t/ \la, r/ \la) , \, \frac{1}{ ( \la)^{\frac{3}{2}}} \, v_t ( t/ \la, r/ \la) \right).
}
A natural space in which to consider Cauchy data for \eqref{defoc v eq} is the homogeneous Sobolev space $\dot{H}^2 \times \dot H^1(\R^5)$ as this norm is invariant under the same scaling as the equation: 
\ant{
\|\vec v_{\la} \|_{ \dot{H}^2 \times \dot H^1} = \|\vec v \|_{ \dot{H}^2 \times \dot H^1},
}
which is why $\dot{H}^2 \times \dot H^1(\R^5)$ is referred to as the \textit{critical} norm for this problem. The main result we will need concerning~\eqref{defoc v eq} is the following conditional global well-posedness and scattering result proved in~\cite{KM11b}. 
\begin{thm} \cite[Theorem~$4.1$]{KM11b}  \label{euc thm}
Let $\vec v(0) = (v_0, v_1) \in  \HH(\R^5)$. Then there exists $T_-, T_+>0$ and a unique solution $\vec v(t) \in \HH$ to~\eqref{defoc v eq} defined on its maximal interval of existence $I_{\max} = (-T_-, T_+)$. Moreover, if we assume that 
\EQ{ 
\| \vec v \|_{L^{\infty}_t(I_{\max}; \HH)} < \infty,
}
then in fact $\vec v(t)$ is globally defined, i.e., $I_{\max} = \R$ and $\vec v(t)$  scatters to zero as $t \to \pm \infty$. 
\end{thm} 
\begin{rem} In~\cite{KM11b} the above theorem is stated with the space $\dot{H}^2 \times \dot{H}^1$ instead of in $\HH(\R^5)$ as we have stated it above.  However, this makes no difference in the result as the defocusing energy guarantees that the $\dot{H}^1 \times L^2$ is uniformly bounded by the energy on any interval $I$ on which the solution exists. Hence the $\HH$ norm is bounded if and only if the solution has finite energy and the $\dot{H}^2 \times \dot{H}^1$ norm is bounded. 
\end{rem}

In the remainder of this subsection we prove the following approximation result, which shows that solutions to the  Euclidean equation,~\eqref{defoc v eq}, serve as a good approximations to solutions to~\eqref{u eq} with initial data concentrating at very small scales.

\begin{prop}\label{small scale prop} Let $(f, g) \in \HH$ and denote by $\vec v(t)$ the solution to~\eqref{defoc v eq} with initial data $(f, g)$ defined on an open interval  $J= (-T_1, T_2)$. Assume in addition that 
\EQ{
\|v\|_{S(J)}+ \| \vec v \|_{L^{\infty}_t(J;  \HH)}  < \infty.
}
Let $\{ \la_n\} \subset (0, \infty)$ be any sequence with $\la_n \to 0$ as $n \to \infty$.  Consider the rescaled initial data 
\EQ{
(f_{n}, g_{n}): =  \left ( \la_n^{-\frac{1}{2}} f( r/ \la_n), \la_n^{-\frac{3}{2}} g( r/ \la_n)\right)
}
and denote by $ \vec v_{n}(t, r)$ the corresponding rescaled solution to~\eqref{defoc v eq}. Then, there exists $N_0$ large enough so that for all $n \ge N_0$ there exists a unique solution $ \vec u_{n}(t, r)$ to~\eqref{u eq} with initial data $(f_{n}, g_{n})$ defined on the interval $J_{n} := (-\la_nT_1, \la_nT_2)$. Moreover, $\vec u_{n}$ satisfies 
\EQ{
 \|  u_{n} - v_{n}\|_{S(J_{n})}  \to 0 \mas n  \to \infty.
 } 
 In particular, $\| u_{n}\|_{S(J_{n})}< \infty$. 
\end{prop}

\begin{proof} The idea is to use the perturbation lemma on $\vec v_{n}$. Begin by observing that 
\ant{
\glei(v_{n}) &= \Box v_{n} + Z_1(r v_{n})v_{n}^{3} + Z_2(r v_{n})v_{n}^5\\
& =\left( Z_2(r v_{n})-\frac{4}{3}\right)v_{n}^5 + Z_1(r v_{n})v_{n}^{3}.
}
Thus by Corollary~\eqref{pert cor} it will suffice to show that 
\EQ{
\|( Z_2(r v_{n})-4/3)v_{n}^5 + Z_1(r v_{n})v_{n}^{3}\|_{N(J_{n})} \to 0 \mas n \to \infty.
}
First, note that we have 
\EQ{
&\abs{Z_1(r v_{n})} \le C.
}
Crucially, one can readily check via the Talyor expansion for $Z_2$ that we have 
\EQ{ \label{Z2 taylor}
& Z_2(r v_{n}) -\frac{4}{3} = O(r^2v_{n}^2).
}
Thus we have 
\ant{
\|( Z_2(r v_{n})-4/3)v_{n}^5 + Z_1(r v_{n})v_{n}^{3}\|_{L^{\frac{3}{2}}_t(J_n; L^{\frac{30}{17}})} &\lesssim   \|r^2  v_n^7\|_{L^{\frac{3}{2}}_t(J_n; L^{\frac{30}{17}})}
+\|v_n^3\|_{L^{\frac{3}{2}}_t(J_n; L^{\frac{30}{17}})} \\
&=  \la_n^2 \|r^2 v^2 v^5\|_{L^{\frac{3}{2}}_t(J; L^{\frac{30}{17}})}\\
& \quad + \la_n^2\|v^3\|_{L^{\frac{3}{2}}_t(J; L^{\frac{30}{17}})},
}
were the second line follows by  recalling that $v_n(t, r) = \la^{-1/2} v(t/ \la_n, r/ \la_n)$ and the change of  variables $\la_n t' = t$ and $ \la_n r' = r$. Using Hardy's inequality for radial functions, i.e., Lemma~\ref{Hardy} and interpolation we have  and  we have 
\ant{
\la_n^2 \|r^2 v^2 v^5\|_{L^{\frac{3}{2}}_t(J; L^{\frac{30}{17}})} &\lesssim  \la_n^2 \|r v\|_{L^{\infty}_t(J; L^{\infty})}^2 \|v^5\|_{L^{\frac{3}{2}}_t(J; L^{\frac{30}{17}})}\\
& \lesssim   \la^2_n \|v\|_{ L^{\infty}_t (J; \dot H^{\frac{3}{2}})}^2 \|v^5\|_{L^{\frac{3}{2}}_t(J; L^{\frac{30}{17}})} \\
& \lesssim  \la^2_n \|v\|_{ L^{\infty}_t (J; \dot H^{1})} \|v\|_{ L^{\infty}_t (J; \dot H^{2})}\|v^5\|_{L^{\frac{3}{2}}_t(J; L^{\frac{30}{17}})}\\
& \lesssim  \la^2_n  \|v\|_{ L^{\infty}_t (J;  \HH)}^2\|v^5\|_{L^{\frac{3}{2}}_t(J; L^{\frac{30}{17}})}.
}
Finally, by the same argument as in the proof of Proposition~\ref{small data} we can control the right hand side above by powers of the $S$ norm and the energy norm which allows us to deduce that 
\ant{
 \la^2_n  \|v\|_{ L^{\infty}_t (J;  \HH)}^2\|v^5\|_{L^{\frac{3}{2}}_t(J; L^{\frac{30}{17}})} \lesssim    \la^2_n  \|v\|_{ L^{\infty}_t (J;  \HH)}^5 \|v\|_{S(J)}^2  \to 0 \mas n \to \infty.
 }
Similarly, we can use again the argument in the proof of Proposition~\ref{small data} to show that 
\ant{
\la_n^2\|v^3\|_{L^{\frac{3}{2}}_t(J; L^{\frac{30}{17}})} \lesssim \la_n^2  \|v\|_{ L^{\infty}_t (J;  \HH)} \|v\|_{S(J)}^2  \to 0 \mas n \to \infty.
}

%\ant{
%\| Z_1(r v_n) v_n^3\|_{L^{\frac{3}{2}}_t(J_n; L^{\frac{30}{17}})} &\lesssim  \|v_n^3\|_{L^{\frac{3}{2}}_t(J_n; L^{\frac{30}{17}})} \\
%& \le \|v_n\|_{L^{\infty}( J_n; L^{10})} \| v_n^2 \|_{L^{\frac{3}{2}}_t(J_n; L^{\frac{15}{7}})}\\
%& \lesssim   \|v_n\|_{L^{\infty}( J_n;  \dot{H}^2)} \| v_n \|^2_{L^{3}_t(J_n; L^{\frac{30}{7}})}.
%}
%Since 
%\ant{
%\|v_n\|_{L^{\infty}( J_n;  \dot{H}^2)} = \|v\|_{L^{\infty}( J;  \dot{H}^2)},
%}
%it suffices to show that 
%\ant{
%\| v_n \|^2_{L^{3}_t(J_n; L^{\frac{30}{7}})} \to 0 \mas n \to \infty.
%}
%Indeed, recalling that $v_n(t, r) = \la^{-1/2} v(t/ \la_n, r/ \la_n)$ and changing variables we have 
%\ant{
%\| v_n \|^2_{L^{3}_t(J_n; L^{\frac{30}{7}})} = \la_n \|v \|_{L^3_t(J; L^{\frac{30}{7}})} \le \la_n \| v\|_{S(J)} \to 0 \mas n \to \infty.
%}

%Therefore, 
%\ant{
%\|(Z_2(r v_{n}) -\frac{4}{3}) v^5 \|_{L^{\frac{3}{2}}_t(J_n; L^{\frac{30}{17}})}  &\lesssim  \|r^2 v_n^2 v_n^5\|_{L^{\frac{3}{2}}_t(J_n; L^{\frac{30}{17}})}
%}

It remains to show that 
\ant{
\|( Z_2(r v_{n})-4/3)v_{n}^5 + Z_1(r v_{n})v_{n}^{3}\|_{L^{\frac{3}{2}}_t(J_n; \dot W^{\frac{30}{17}})} \to 0 \mas n \to \infty.
}
We first deal with the terms involving the quintic part of the nonlinearity. We have  
\ant{
\p_r(Z_2(r v_{n})-4/3)v_{n}^5  = Z_2'(r v_n)v_n^6 + Z_2'(r v_n) r v_n^5 \p_r v_n + 5( Z_2(r v_n) - 4/3) v_n^4 \p_r v_n.
}
Using~\eqref{Z2 taylor} and the fact that $\abs{Z_2'( \rho)}  \lesssim \abs{\rho}$  we see that 
\ant{
\abs{\p_r(Z_2(r v_{n})-4/3)v_{n}^5 } \lesssim r^2 v_n^6  r^{-1}\abs{v_n}  + r^2 v_n^6 \abs{\p_r v_n}.
}
With the above and Hardy's inequality followed by Sobolev embedding we obtain, 
\ant{
\|( Z_2(r v_{n})-4/3)v_{n}^5 \|_{L^{\frac{3}{2}}_t(J_n; \dot W^{\frac{30}{17}})} &\lesssim \|r^2  v_n^6\|_{L^{\frac{3}{2}}_t(J_n; L^{\frac{15}{4}})} \| r^{-1} v_n\|_{L^{\infty}_t(J_n; L^{\frac{10}{3}})}\\
& \quad  +  \|r^2  v_n^6\|_{L^{\frac{3}{2}}_t(J_n; L^{\frac{15}{4}})}\| \p_r v_n\|_{L^{\infty}_t( J_n; L^{\frac{10}{3}})}\\
&  \lesssim  \|r^2  v_n^6\|_{L^{\frac{3}{2}}_t(J_n; L^{\frac{15}{4}})}\|  v_n\|_{L^{\infty}_t( J_n; \dot H^2)}\\
}
Since $$\|  v_n\|_{L^{\infty}_t( J_n; \dot H^2)} = \|  v\|_{L^{\infty}_t( J; \dot H^2)} \le C,$$ it suffices to show that 
\ant{
\|r^2  v_n^6\|_{L^{\frac{3}{2}}_t(J_n; L^{\frac{15}{4}})} \to 0 \mas n \to \infty.
}
Indeed, changing variables and another application of Lemma~\ref{Hardy} (Hardy's inequality) we can deduce that  
\ant{
\|r^2  v_n^6\|_{L^{\frac{3}{2}}_t(J_n; L^{\frac{15}{4}})}  &=  \la_n  \|r^2  v^6\|_{L^{\frac{3}{2}}_t(J; L^{\frac{15}{4}})}  \\
& \le \la_n \|r v\|_{L^{\infty}_t (J; L^{\infty})}^2 \|v \|_{L^{\frac{3}{2}}(J; L^{15})}^4 \\
& \lesssim \la_n \|v\|_{L^{\infty}_t (J; \HH)}^2 \|v \|_{L^{\frac{3}{2}}(J; L^{15})}^4\\
& \lesssim \la_n \| v\|_{L^{\infty}_t (J; \HH)}^4  \|v \|_{S(J)}^2 \to 0 \mas n \to \infty,
}
where the last line follows again by the argument in the proof of Proposition~\ref{small data} and the boundedness of $v$ in  $X(J) = L^{\infty}_t( J; \HH) \cap S(J)$. We are left to deal with the remaining cubic term. Note that by~\ref{Z bounds}
\ant{
\abs{\p_r(Z_1(r v_n) v_n^3)} &= \abs{r v_n Z_1'(r v_n)( v_n^2 r^{-1}  + v_n^2 \p_r v_n) + 3Z_1(r v_n) v_n^2 \p_r v_n}\\
& \lesssim   v_n^2 r^{-1} \abs{v_n} + v_n^2 \abs{\p_r v_n}.
}
Hence, 
\ant{
\| Z_1(r v_{n})v_{n}^{3}\|_{L^{\frac{3}{2}}_t(J_n; \dot W^{\frac{30}{17}})} \lesssim \|v_n^3 r^{-1}\|_{L^{\frac{3}{2}}_t(J_n; \dot W^{\frac{30}{17}})} + \|v_n^2\p_r v_n\|_{L^{\frac{3}{2}}_t(J_n; \dot W^{\frac{30}{17}})}.
}
Changing variables gives 
\ant{
\|v_n^3 r^{-1}\|_{L^{\frac{3}{2}}_t(J_n; \dot W^{\frac{30}{17}})} + \|v_n^2\p_r v_n\|_{L^{\frac{3}{2}}_t(J_n; \dot W^{\frac{30}{17}})} &= \la_n\|v^3 r^{-1}\|_{L^{\frac{3}{2}}_t(J; \dot W^{\frac{30}{17}})} \\
& \quad+ \la_n\|v^2\p_r v\|_{L^{\frac{3}{2}}_t(J; \dot W^{\frac{30}{17}})}.
}
Finally, using again the arguments from the proof of Proposition~\ref{small data} we obtain
\ant{
\la_n\|v^3 r^{-1}\|_{L^{\frac{3}{2}}_t(J; \dot W^{\frac{30}{17}})}+ \la_n\|v^2\p_r v\|_{L^{\frac{3}{2}}_t(J; \dot W^{\frac{30}{17}})}& \lesssim  \la_n \|v\|_{L^{\infty}(J; \HH)} \|v \|_{S(J)}^2,
}
which tends to $0$ as $n \to \infty$ by the boundedness of the $X(J)$ norm of $v$. This completes the proof. 
\end{proof}

%We begin by noting that all solutions to \eqref{an ode} are global. 

%\vspace{3 in} 
 
%We would like to show that for each $n$ there is a unique solution $Q_n$ to \eqref{an ode} satisfying the boundary conditions \eqref{bc}. We would also need asymptotics., monotonicity etc. e.g.,  can we show 
%\ant{
% Q(r) = \pi - O(r^{-\beta}) \mas r \to \infty\\
% Q(r) = O( r^{\al}) \mas r \to 0
% }
% and what are $\al, \be$? 

%\begin{lem}
%Let $\phi$ be a solution to 
%\end{lem}

%One can rewrite this as an autonomous system by defining the vector $v = (x, y, z) = ( s, \fy(s), \fy'(s))$ we would then have 
%\EQ{
%\dot v = \pmat{1\\ z\\ -z + \sin 2y + e^{-2x}f(y)} =: F(v)
%}
%where $f(y) = (y-\sin y \cos y)(1-\cos 2y)$. 

%\subsection{Asymptotic Stability of the A-N Soliton} 

%\subsubsection{Strichartz Estimates for the linearized operator}

%%%%%%%%%%%%%%%%%%%%%%%%%%%%%%%%%%%%%%%%%%%%%%%%%%%%%%%%%%%%%%%%%%%%%%%%
%%%%%%%%%%%%%%%%%%%%%%%%%%%%%%%%%%%%%%%%%%%%%%%%%%%%%%%%%%%%%%%%%%%%%%%%
%%%%%%%%%%%%%%%%%%%%%%%%%%%%%%%%%%%%%%%%%%%%%%%%%%%%%%%%%%%%%%%%%%%%%%%%

\section{Concentration Compactness}
\subsection{Profile decomposition in $ \HH = (\dot{H}^2 \times \dot{H}^1) \cap (\dot{H}^1 \times L^2)$}
In this subsection we prove a nonlinear Bahouri-Gerard type decomposition that lies at the heart of the concentration compactness argument. 

\subsubsection{Linear Profile Decomposition}
We begin with the profile decomposition at the level of the underlying free waves. As we will be dealing with sequences of radial free waves in $\HH$, which is not a scaling invariant norm, the statement of the profile decomposition will be slightly different than in \cite{BG, Bu}. The proof of the following theorem is nearly identical to those found in \cite{BG, Bu}. For completeness  we outline a few minor differences in what follows and we refer the reader to \cite{BG, Bu} for the reminder of the argument. 

\begin{prop}\label{BG} \cite{BG, Bu}
Let $\{\vec u_n\}$ be a sequence of free radial waves bounded in $\HH=(\dot H^2 \times \dot H^1)\cap (\dot H^1 \times L^2)(\R^5)$. Then after replacing it by a subsequence, there exist a sequence of free waves $\vec V_L^j$ bounded in $\HH$, and sequences of times $\{t_{n, j}\}\subset\R$  and scales $\{\la_{n,j}\} \subset (0, \infty)$ such that for  $\ga_n^k$ defined by
\EQ{\label{eq:BGdecomp}
 &u_n(t) = \sum_{1\le j<k}  \frac{1}{ \la_{n,j}^{\frac{1}{2}}}V_L^j\left(\frac{t-t_{n,j}}{  \la_{n,j}}, \frac{ \cdot}{\la_{n,j}}\right) + \ga_n^k(t) \\
 & \p_t u_n(t) =  \sum_{1\le j<k}  \frac{1}{ \la_{n,j}^{\frac{3}{2}}}\p_tV_L^j\left(\frac{t-t_{n,j}}{  \la_{n,j}}, \frac{ \cdot}{\la_{n,j}}\right) + \p_t\ga_n^k(t) \\
}
we have for any $j<k$, 
\EQ{
( \la_{n, j}^{\frac{1}{2}}\ga_n^k( \la_{n,j} t_{n,j}, \la_{n,j} \cdot), \la_{n,j}^{\frac{3}{2}} \p_t\ga_n^k( \la_{n,j} t_{n,j}, \la_{n,j} \cdot)) \rightharpoonup 0 \quad \textrm{weakly in} \quad \HH \mas n\to\I,
}  
as well as 
\EQ{\label{parameters diverge}
 \lim_{n\to\I} \frac{ \la_{n,j}}{ \la_{n,k}} + \frac{ \la_{n,k}}{ \la_{n,j}} + \frac{ |t_{n,j}-t_{n,k}|}{ \la_{n,k}} + \frac{ |t_{n,j}-t_{n,k}|}{ \la_{n,j}} = \I     
}
and the errors $\ga_{n}^{k}$ vanish asymptotically
in the sense that
\EQ{ \label{gamma vanish}
 \pt \lim_{k\to \I} \limsup_{n\to\I} \|\ga_n^k\|_{L^\I_t(\R; W^{1, \frac{10}{3}}_x)\cap S(\R)}=0 \quad  
}
where $S(\R) :=L^{3}_t(\R; W^{1, \frac{30}{7}})$. Finally, one has orthogonality of the $\HH$ norm. 
\EQ{ \label{HH orth}
 \| \vec u_n \|_{\HH}^2 = \sum_{1\le j<k} \|\vec V_L^j\|_{\HH}^2 + \|\vec\ga_n^k\|_{\HH}^2 +o_n(1)
} 
as $n\to\I$. 
\end{prop}

\begin{rem} The subscript $L$ in $V^j_{L}$ stands for ``linear" and is used here to distinguish these profiles, which are free waves, from the nonlinear profiles, which will be introduced below. 
\end{rem}
 \begin{rem}\label{0 or 1 scales} Note that up to extracting a further subsequence we can assume that each scale $\la_{n,j} \to \la_{\infty, j} \in [0, +\infty)$. In addition, we can further assume that for each scale $\la_{n,j}$ we either have 
 \EQ{
 \la_{n,j} &= 1 \, \, \forall n \\
 &\mor\\
 \la_{n,j } &\to 0 \mas n \to \infty.
 }
 This is due to the fact that our sequence $\vec u_n$ is bounded in the inhomogeneous space $\HH$ and hence the failure of compactness cannot arise via a low frequency buildup, but rather only via time translation or frequencies escaping to $\infty$. To see this suppose that $w \in \dot H^2 \cap \dot H^1$. Let $\{\la_n\}$ be any sequence of positive numbers so that $\la_n \to \infty$ as $n \to \infty$ and consider $$w_n(r):= \frac{1}{ \sqrt{ \la_n}} w( r/ \la_n).$$ We then have 
 \ant{
 \|w_n\|_{ \dot{H}^2} = \|w\|_{\dot{H}^2}
 }
 but 
 \ant{
 \|w_n\|_{ \dot{H}^1} = \la_n \|w\|_{ \dot{H}^1} \to \infty \mas n \to \infty.
 }
 Hence the sequence $w_n$ is not bounded in $\dot{H}^2 \cap \dot H^1$. 
 
 Also, if for any $j$ we have $\la_{n,j} \to \mu>0$ we can assume, by rescaling the profile that $\la_{n,j} = 1$ for all $n$. This observation will be important when we consider nonlinear profiles in the following subsection. 
 
 \end{rem}
 
 \begin{rem} \label{times 0 or inf}
 After passing to a subsequence if necessary, we can, without loss of generality, assume that the sequences $t_{n, j}/ \la_{n, j} \to \mu \in [- \infty, \infty]$ are convergent.   By translating the profiles, we can the also ensure that for a triple $(V_L^j, t_{n, j}, \la_{n. j})$ we either have
 \ant{
 &t_{n, j} = 0  \quad \forall n\\
& \mor\\
 & t_{n, j}/ \la_{n,j} \to \pm \infty \mas n \to \infty.
}
 
 \end{rem}

 The proof of Proposition~\ref{BG} requires the following refinement of the Sobolev embedding theorem, proved in \cite{BG, Bu}.  
 
 \begin{lem}[Refined Sobolev Embedding]\label{RSE}\cite{BG, Bu} Given $f \in \HH$ we have 
 \ant{
 \|f\|_{W^{1, \frac{10}{3}}} \lesssim \| f\|_{ \dot{B}^1_{2, \infty} \cap \dot B^2_{2, \infty}}^{\frac{1}{3}} \|f \|_{\HH}^{\frac{2}{3}}
 }
 \end{lem} 

\begin{proof}[Reduction of Proof of Proposition~\ref{BG} to the argument in~\cite{BG}]
Instead of reproducing the entire argument, we simply outline the reduction of the Proposition to a situation where the argument in~\cite{BG} can then be applied verbatim. We note that without loss of generality we can restrict our attention to sequences $\vec u_n$ such that 
\EQ{
\vec u_n(0) \rightharpoonup 0  \quad \textrm{weakly in}  \quad \HH,
}
since if $\vec u_n(0) \rightharpoonup \vec v^1(0) \neq 0$ weakly in $\HH$,  we always take the weak limit $ \vec v^1(0)$  to be the first profile with times $t_{n,1} = 0$  and scales $\la_{n,1} = 1$ for every $n$, and then consider the new sequence $\vec u_n - \vec v^1$. %Of course this means that all the subsequent time sequences $t_n^j$ with $j >1$ will need to satisfy $\abs{t_n^j} \to  \infty$ as $n \to \infty$ and this will be a consequence of the following construction. 

First, we observe that we can reduce~\eqref{gamma vanish} to estimating the $L^{\infty}_tW^{1, \frac{10}{3}}_x$ norm alone. Indeed, by interpolation followed by an application of the Strichartz estimates we have 
\ant{
\| \gamma_n^k\|_{L^{3}_tW^{1, \frac{30}{7}}_x} &\le \| \gamma_n^k\|_{L^{\infty}_t W^{1, \frac{10}{3}}_x}^{\frac{1}{3}} \| \gamma_n^k\|_{L^{2}_tW^{1, 5}}^{\frac{2}{3}}\\
&\lesssim \| \gamma_n^k\|_{L^{\infty}_t W^{1, \frac{10}{3}}_x}^{\frac{1}{3}}  \| \vec \gamma_n^k\|_{\HH}\\
&\lesssim \| \gamma_n^k\|_{L^{\infty}_t W^{1, \frac{10}{3}}_x}^{\frac{1}{3}} 
}
where the uniform boundedness of the $\HH$ norms of the $\gamma_n^k$ follows from~\eqref{HH orth}, (and the latter not depending on~\eqref{gamma vanish} for the proof). 
 
 Next we remark that we can absorb the low frequencies of our sequence into the errors $\gamma_n^k$. Indeed, write 
\EQ{
u_n  &= P_{ \le 0} u_n  + P_{ \ge 1} u_n\\
&=: w_n + \ti u_n
}
where $P_{ \ge 1}$ denotes projection onto frequencies of size $\gec 1$. We claim that 
\EQ{
\| \vec w_n\|_{L^{\infty}_tW^{1, \frac{10}{3}}} \to 0 \mas n \to \infty
}
By Lemma~\ref{RSE} it suffices to show that 
\EQ{
\limsup_{n \to \infty}\| \vec w_n\|_{L^{\infty}_t ((\dot{B}^2_{2, \infty} \times \dot{B}^1_{2, \infty}) \cap (\dot{B}^1_{2, \infty} \times \dot{B}^0_{2, \infty}))}  =0
}
Since both $\vec w_n(t) = (w_n(t), \p_t w_n(t))$  and $\vec w_{n, r}(t) = ( \p_r w_n(t), \p_t \p_r w_n (t))$ are free waves, the above norm is in fact constant in time --  indeed the free wave operator commutes with Fourier multipliers. So we can further reduce matters to showing that  
\EQ{\label{B vanish}
\limsup_{n \to \infty}\| \vec w_n(0)\|_{(\dot{B}^2_{2, \infty} \times \dot{B}^1_{2, \infty}) \cap (\dot{B}^1_{2, \infty} \times \dot{B}^0_{2, \infty})}  =0.
}
And the above follows from the fact that $\vec w_n \rightharpoonup 0$ weakly in $\HH$ and the frequency localization, $$ \supp \hat w_n(0,  \cdot) \subset \{ \abs{ \xi} \le 1\}.$$ (Indeed, in the terminology of \cite{BG}, $w_n$ and its derivatives are $\underline{\varepsilon}$-singular for all scales $\underline{\e}:= \{\e_n\}_{n=1}^{\infty} \to 0$ and hence~\eqref{B vanish} follows from \cite[Lemma $3.1$]{BG}. )

Therefore, it suffices to extract profiles from the sequence 
\EQ{
 \ti u_n := P_{ \ge 1} u_n.
 }
 The key point here is that for functions with Fourier support bounded away from $\{\xi = 0\}$ the homogeneous and inhomgeneous Sobolev (and Besov) norms are equivalent. Hence our sequence of free waves satisfies the uniform bounds 
 \EQ{
\sup_{t, n} \|  \vec{\ti{u}}_n(t) \|_{H^2 \times H^1} \lesssim \sup_{t, n} \|  \vec{\ti{u}}_n(t) \|_{\HH} \le  C.
 }
From here, the proof proceeds in the exact same fashion as in Bahouri, Gerard \cite{BG} or Bulut \cite{Bu} and we refer the reader to either of these references for the remainder of the argument. 
\end{proof}

\subsubsection{Nonlinear Profiles} \label{nonlinear profiles}
We now turn to the nonlinear profiles that we can associate to a sequence of data bounded in $\HH$. Consider a bounded sequence $\vec u_n \in \HH$ and after extracting a suitable subsequence the associated profile decomposition 
\EQ{
 &u_{n, 0} = \sum_{1\le j<k}  \frac{1}{ \la_{n,j}^{\frac{1}{2}}}V_L^j\left(\frac{-t_{n,j}}{  \la_{n,j}}, \frac{ \cdot}{ \la_{n,j}}\right) + \ga_{n, 0}^k \\
 &  u_{n, 1} =  \sum_{1\le j<k}  \frac{1}{ \la_{n,j}^{\frac{3}{2}}}\p_tV_L^j\left(\frac{-t_{n,j}}{ \la_{n,j}}, \frac{ \cdot}{ \la_{n,j}}\right) + \ga_{n, 1}^k 
 }
as in Proposition~\ref{BG}. Recall by Remark~\ref{0 or 1 scales} that we can assume that for each $j$ either $\la_{n,j} \to 0 $ as $n \to \infty$ or $\la_{n,j} = 1$ for all $n$. With this distinction in the behavior of the scales we define two different types of profiles accordingly. From here on, we will refer to a triple $(V_L^j, t_{n, j}, \la_{n, j})$ as   {\em Euclidean} if the scale $\la_{n,j}  \to 0$ as $n \to \infty$ and we will use the notation $V^j_{e, L}$ for the profile.  On the other hand, if $\la_{n, j} =1 $ for all $n$ we will refer to the triple $(V_L^j, t_{n, j}, \la_{n, j}) $ as {\em Adkins-Nappi}  and we will use the notation $U^j_{a,L}$ for the profile.  As each triple is either Euclidean  or Adkins-Nappi we can rewrite our profile decomposition as follows: 
\EQ{ \label{AN BG}
 &u_{n, 0} = \sum_{1\le j<k}  \frac{1}{\la_{n,j}^{\frac{1}{2}}}V^j_{e, L}\left(\frac{-t_{n,j}}{  \la_{n,j}}, \frac{ \cdot}{ \la{_n,j}}\right) + \sum_{1\le \ell<k}  U^{\ell}_{a,L}\left(-t_{n,\ell}, \cdot \right) + \ga_{n, 0}^k, \\
 &  u_{n, 1} =  \sum_{1\le j<k}  \frac{1}{ \la_{n,j}^{\frac{3}{2}}}\p_tV^j_{e, L}\left(\frac{-t_{n,j}}{  \la_{n,j}}, \frac{ \cdot}{ \la_{n,j}}\right)+ \sum_{1\le \ell<k}  \p_tU^{\ell}_{a,L}\left(-t_{n,\ell}, \cdot \right)  + \ga_{n, 1}^k.
 } 

Now to each profile $V_{e, L}^j$, respectively  $U_{a,L}^{\ell}$, we can associate a unique nonlinear profile $V^j_{e}$, respectively $U^{\ell}_{a}$, which is defined as follows. 

For a Euclidean profile $V_{e, L}^j$ with the associated sequences $\{t_{n, j}\}$, $\{\la_{n, j}\}$ we define the nonlinear profile $V^j_{e}$ to be the unique solution to~\eqref{defoc v eq} so that for all $n$, $-t_{n, j}/ \la_{n, j} \in I_{\max}(\vec V^j_{e}) = (T_-(\vec V_{e}^j), T_+(\vec V_e^j))$ and 
\EQ{
\| \vec V_{e, L}^j(-t_{n, j}/ \la_{n,j}) - \vec V_{e}^j(-t_{n, j}/ \la_{n, j}) \|_{\HH} \to 0 \mas n \to \infty.
}
We note that a nonlinear profile always exists. In the case $t_{n, j} = 0$ for all  $n$ the existence of the nonlinear profile follows from the local well-posedness theory for~\eqref{defoc v eq} developed in \cite{KM11b}. In the case that $-t_{n, j}/ \la_{n, j} \to  \pm \infty$ the existence of the nonlinear profile is referred to as the existence of wave operators and follows by running the same argument as for the local well-posedness theory from  $t_0 = \pm \infty$. Note that if $-t_{n, j}/ \la_{n.j}  \to + \infty$ then it is immediate that $T_+( \vec V^j_{e}) = + \infty$ and  
\EQ{
s_0> T_-(\vec V_e^j) \Longrightarrow \,  \|V_e^j\|_{S([s_0, \infty))} < \infty.
}
The above, we recall is equivalent to $\vec V_e^j$ scattering at $t=+ \infty$ which occurs by construction. An analogous statement holds if $-t_{n, j}/ \la_{n, j}  \to - \infty$. 
To simplify notation we will write 
\EQ{\label{Vn}
V_{e, n}^j(t, r):= \frac{1}{ \la_{n, j}^{\frac{1}{2}}} V^j_{e}\left( \frac{ t- t_{n, j}}{ \la_{n, j}}, \frac{r}{ \la_{n, j}} \right).
}
 We remark that this construction is natural in light of Proposition~\ref{small scale prop} since $\la_{n,j} \to 0$ implies that in this situation, solutions to the Euclidean equation~\eqref{defoc v eq} serve as a very good approximations to solutions to~\eqref{u eq} as $n \to  \infty$. This fact will be reflected in the nonlinear profile decomposition proved below in Proposiion~\ref{nonlin profile}.

For an Adkins-Nappi profile $ \vec U_{a,L}^{\ell}$, with sequence $\{t_{n, \ell}\}$, we similarly define a nonlinear profile, $\vec U_{an,}^{\ell}$, which is the unique solution to~\eqref{u eq}  with $-t_{n, \ell} \in I_{\max}( \vec U_{a})$ for all $n$ and
\EQ{
\| \vec U_{a,L}^{\ell}(-t_{n, \ell}) - \vec U_{a}^{\ell}(-t_{n, \ell}) \|_{\HH} \to 0 \mas n \to \infty.
}
Such a nonlinear profile again always exists by the local well-posedness theory from Proposition~\ref{small data}. And again we note that if $-t_{n, \ell} \to + \infty$ we have $T_+( \vec U_{a}^{\ell}) = + \infty$ and 
\EQ{
s_0> T_-(\vec U_{a}^{\ell}) \Longrightarrow \,  \|U_{a}^{\ell}\|_{S([s_0, \infty))} < \infty.
}
Similarly, if  $-t_{n, \ell} \to - \infty$ then by construction, $T_(\vec U_a^{\ell}) = - \infty$ and  $\vec U_a^{\ell}$ scatters at $- \infty$.  
To simplify notation we will denote 
\EQ{ \label{Un}
U_{a, n}^{\ell}(t, r):=   U_a^{\ell}(t - t_{n, j}, r).
}
 
The key point here is that there despite the lack of a superposition principle at the nonlinear level, the orthogonality of the parameters allows for a nonlinear profile decomposition which serves as a good approximation for the nonlinear evolution of the original sequence of initial data. 

\begin{prop}[Nonlinear profile decomposition] \cite{DKM1, CKLS1}\label{nonlin profile} Let $(u_{n, 0}, u_{n, 1}) \in \HH$ be a bounded sequence along with the associated profile decomposition as in~\eqref{AN BG}. Let $ \vec V_{e}^j$, $ \vec U_{a}^{\ell}$ be the associated Euclidean and Adkins-Nappi nonlinear profiles as defined in this subsection. Let $\{s_n\}  \subset (0, \infty)$ be any sequence of times so that 
\EQ{
&\forall j \ge 1, \, \,   \forall n, \, \, \frac{s_n - t_{n, j}}{ \la_{n, j}} < T_+( \vec V^{j}_e) \mand   \limsup_{n \to \infty}  \|V_e^j\|_{S\left( \frac{-t_{n, j}}{ \la_{n, j}},  \frac{ s_n- t_{n, j}}{ \la_{n, j}} \right)} < \infty, \\
&\forall  \ell \ge 1, \, \,   \forall n, \, \, s_n - t_{n, j}  < T_+( \vec U^{\ell}_{a}) \mand   \limsup_{n \to \infty}  \|U_{a}^{\ell}\|_{S\left( -t_{n, j},  s_n- t_{n, j}\right)} < \infty. 
  }
  If $\vec u_n(t) \in \HH$ is the solution to~\eqref{u eq} with initial data $\vec u_n(0) = (u_{n, 0}, u_{n, 1})$ then $\vec u_n(t)$ is defined on $[0, s_n)$ and 
  \EQ{
    \limsup_{n \to \infty} \|u_n\|_{S([0, s_n))} < \infty.
    }
    Moreover the following nonlinear profile holds decomposition holds: For $\eta_n^k$ defined by 
   \EQ{ \label{BG nonlin}
 &u_{n}(t,r) = \sum_{1\le j<k}  V^j_{e, n}(t, r) + \sum_{1\le \ell<k}  U^{\ell}_{a,n}(t,r) + \ga_{n}^k(t) + \eta_{n}^k(t), 
 %\\
% &  \p_t u_n(t) =  \sum_{1\le j<k}  \frac{1}{ \la_{n,j}^{\frac{3}{2}}}\p_tV^j_{e}\left(\frac{t-t_{n,j}}{  \la_{n,j}}, \frac{ \cdot}{ \la_{n,j}}\right)+ \sum_{1\le \ell<k}  \p_tU^{\ell}_{a,L}\left(t-t_{n,\ell}, \cdot \right)  + \p_t\ga_{n,L }^k(t).
 } 
we have 
\EQ{
\lim_{k \to \infty} \limsup_{n \to \infty}  \left(\| \eta_n^k\|_{S([0,  s_n))} + \|  \vec \eta_n^k\|_{L^{\infty}_t([0, s_n); \HH)} \right) = 0.
}
Here $\ga_{n}^k(t) \in \HH$ is as in Proposition~\ref{BG} and  $V_{e, n}^j$ and $U_{a, n}^{\ell}$ are defined as in~\eqref{Vn} and~\eqref{Un} respectively. Also, we note that an analogous statement holds for $ s_n<0$. 

\begin{proof}[Proof of Proposition~\ref{nonlin profile}]
Set 
\EQ{
\vec v_n^k(t, r) := \sum_{1\le j<k}  V^j_{e, n}(t, r) + \sum_{1\le \ell<k}  U^{\ell}_{a,n}(t,r) %+ \ga_{n}^k(t)
}
The idea is to use the Lemma~\ref{perturbation} on $\vec u_n$ and $v_n^k$ for large $n$ and we need to check that the conditions of Lemma~\ref{perturbation} are satisfied for these choices. 
For each $n$ we let $I_n := [0, s_n)$. First note that $\glei(u_n) = 0$. We claim that $\|\textrm{eq}(v_n^k)\|_{N(I_n)}$ is small for large $n$. To see this, we define 
\ant{
\mathcal{N}( u) &:= Z_1(ru)u^3 + Z_2(ru)u^5 \\
&= \frac{\sin(2ru)- 2ru}{r^3} +  \frac{(ru - \frac{1}{2}\sin(2ru))2\sin^2(ru)}{r^5}
}
to be the nonlinearity in~\eqref{u eq}. Then we have 
\EQ{ \label{eqvnk}
\glei( v_n^k) &=  %N\left(\sum_{1\le j<k}  V^j_{e, n}(t, r) + \sum_{1\le \ell<k}  U^{\ell}_{a,n}(t,r)\right) \\
%& \quad - \sum_{1\le j<k}  \frac{4}{3}(V^j_{e, n}(t, r))^5 - \sum_{1\le \ell<k} N( U^{\ell}_{a,n}(t,r))  \\
%&=
 \NN\left(\sum_{1\le j<k}  V^j_{e, n}(t, r) + \sum_{1\le \ell<k}  U^{\ell}_{a,n}(t,r)\right) \\
& \quad - \sum_{1\le j<k} \NN(V^j_{e, n}(t, r)) - \sum_{1\le \ell<k} \NN( U^{\ell}_{a,n}(t,r))  \\
& \quad +  \sum_{1\le j<k} \left[ \NN(V^j_{e, n}(t, r)) - \frac{4}{3} (V^j_{e, n}(t, r))^5\right]
}
The same argument used to prove Proposition~\ref{small scale prop} gives that for each fixed $k$ we have 
\ant{
\left\|\sum_{1\le j<k} \left[ \NN(V^j_{e, n}(t, r)) - \frac{4}{3} (V^j_{e, n}(t, r))^5\right] \right\|_{N(I_n)} \longrightarrow 0 \mas n \to \infty
}
To control the first two lines of the right hand side of~\eqref{eqvnk} we make use of the following simple trigonometric inequalities. 
\begin{multline} \label{trig 1}
\abs{\frac{\sin(2ru) + \sin(2rv) - \sin(2r(u+v))}{2r^3}} \\= \abs{\frac{2 \sin(2ru) \sin^2(rv) + 2 \sin(2rv) \sin^2 (ru)}{2r^3}}  \lesssim u^2\abs{v} + v^2 \abs{u}
\end{multline} 
and 
\begin{align}  \notag
&\Bigg\vert\frac{(ru + rv -  \frac{1}{2} \sin(2ru+2rv))2\sin^2(ru + rv)}{r^3}\\ \label{trig 2}
& \quad- \frac{(ru- \frac{1}{2} \sin(2ru))2\sin^2(ru) - (rv- \frac{1}{2} \sin(2rv)) 2\sin^2 (rv)}{r^3} \Bigg| \\
&\lesssim \abs{u}^4 \abs{v}+ \abs{u}^3\abs{v}^2 + \abs{u}^2 \abs{v}^3 + \abs{u} \abs{v}^4 .\notag
\end{align}
Then, using~\eqref{trig 1}, \eqref{trig 2} together with the pseudo-orthogonality of the parameters, i.e.,~\eqref{parameters diverge}, one can show that 
\EQ{\label{vnk vanish}
&\Bigg\|\NN\left(\sum_{1\le j<k}  V^j_{e, n}(t, r) + \sum_{1\le \ell<k}  U^{\ell}_{a,n}(t,r)\right)  \\
& \quad - \sum_{1\le j<k} \NN(V^j_{e, n}(t, r)) - \sum_{1\le \ell<k} \NN( U^{\ell}_{a,n}(t,r)) \Bigg\|_{N(I_n)} \to 0 \mas n \to \infty  .
 }
 The proof of~\eqref{vnk vanish} using~\eqref{parameters diverge} is standard and we refer the reader to~\cite[Lemma~$2.18$, Lemma~$2.16$]{CKLS1} for the details.
 
Next it is essential that 
\begin{align} \label{unif in k}
 \limsup_{n \to \infty} \left\|\sum_{1\le j<k}  V^j_{e, n} + \sum_{1\le \ell<k}  U^{\ell}_{a,n}\right\|_{S(I_n)} \le C < \infty,
 \end{align} 
 {\em uniformly in $k$}, which will follow from the small data theory together with \eqref{HH orth}. The point here is that the sum can be split into one over $1 \le j \le j_0$ and another over $j_0 \le j \le k$. The splitting is performed in terms of the $\HH$ norm, with $j_0$ being chosen so that 
 \begin{align*} 
  \limsup_{n \to \infty} \left( \sum_{j_0 <j \le k} \|V^j_{e,L}(-t_{n, j}/ \la_{n,j})\|_{\HH}^2  + \sum_{j_0 < \ell \le k} \|U^j_{a,L}(-t_{n, j})\|_{\HH}^2 \right)< \de_0^2,
  \end{align*}
  where $\de_0$ is chosen so that the small data theory for both ~\eqref{u eq} and~\eqref{defoc v eq} applies. Using again \eqref{parameters diverge} as well as the small data scattering theory one now obtains 
  \begin{align*} 
  \limsup_{n \to \infty}  &\left\|\sum_{j_0< j\le k}  V^j_{e, n} + \sum_{j_0< \ell<k}  U^{\ell}_{a,n}\right\|^3_{S(I_n)} = \sum_{j_0 < j \le k}\|V^j_{e}\|_{L^{3}_t\left((-\frac{t_{n, j}}{\la_{n, j}}, \frac{s_n- t_{n, j}}{ \la_{n, j}}); \dot W^{1, \frac{30}{7}}\right)}^3 \\
  &\hspace{2in} +  \sum_{j_0< \ell \le k}  \| U_{a}^{\ell}\|_{S(-t_{n, j}, s_n-t_{n, j})}^3 \\
  & \le C \limsup_{n \to \infty}\left( \sum_{j_0 <j \le k} \|V^j_{e,L}(-t_{n, j}/ \la_{n,j})\|_{\HH}^2  + \sum_{j_0 < \ell \le k} \|U^j_{a,L}(-t_{n, j})\|_{\HH}^2 \right)^{\frac{3}{2}} 
  \end{align*}
 with an absolute constant $C$. This implies \eqref{unif in k}. 
 
 Finally, we note that by fixing $k_0$ large enough we can use~\eqref{gamma vanish} to ensure that the $S(I_n)$ norm of $S(t)(\vec u_n(0)- \vec v_n^k(0)) = \vec \ga_n^k(t)$ is small enough to apply Lemma~\ref{perturbation} for all $k>k_0$ and for large $n$ enough. The desired result follows directly from Lemma~\ref{perturbation}.
 \end{proof}

\end{prop}

\subsection{Construction of a Critical Element}\label{CE}

We now turn to the proof of Theorem~\ref{u main}, which  we recall is the reformulation of Theorem~\ref{main} in terms of $u$, where $ru = \psi$ and $\psi$ is the equivariant azimuth angle. We follow the concentration compactness/rigidity methodology of Kenig and Merle introduced in \cite{KM06, KM08}. 

In this subsection we assume that the theorem fails and we construct a minimal non-scattering  solution, referred to as a critical element. We rely heavily on concentration compactness techniques. The refined construction presented here has its roots in the work of Kenig and Merle, \cite{KM10, KM11a, KM11b, DKM5}. Also relevant here is  the work of Ionescu, Pausader, Staffilani \cite{IPS} on the Schr\"odinger equation on $\mathbb{H}^3$ where the issue of two types of profiles also arises. 

We begin with some notation, following \cite{KM10} for convenience. For initial data $(u_0, u_1) \in \HH$ we denote by $\vec u(t) \in \HH$ the unique solution to~\eqref{u eq} with initial data $\vec u(0) = (u_0, u_1)$ defined on its maximal interval of existence \\$I_{\max}(\vec u) := (T_-( \vec u), T_+( \vec u))$.  For $A>0$ define 
\EQ{
B(A):= \{ (u_0, u_1) \in \HH  \, \, : \, \,      \|\vec u(t)\|_{L^{\infty}_t([0, T_+( \vec u); \HH)} \le A\}.
}
\begin{defn} We say that $\mathcal{SC}(A)$ holds if for all $ \vec u = (u_0, u_1) \in B(A)$ we have $T_+( \vec u) = + \infty$ and $ \|u \|_{S([0, \infty))} < \infty$. We also say that $\mathcal{SC}(A;  \vec u)$ holds if $\vec u \in B(A)$, $T_+( \vec u) = + \infty$ and $ \|u \|_{S([0, \infty))} < \infty$. 
\end{defn} 
\begin{rem}
Recall from Proposition~\ref{small data} that $\| u \|_{S([0, + \infty))}< \infty$ if and only if $\vec u$ scatters to a free wave as $t \to \infty$. Therefore Theorem~\ref{u main} (and hence Theorem~\ref{main}) is equivalent to the statement that $\mathcal{SC}(A)$ holds for all $A>0$. 
\end{rem} 

Now suppose that Theorem~\ref{u main} {\em fails}. By the small data theory, i.e., Proposition~\ref{small data}, there is an $A_0>0$ small enough so that $\mathcal{SC}(A_0)$ holds.  Give that we are assuming that Theorem~\ref{u main} fails we can find a critical value $A_C$ so that for $A<A_C$, $\mathcal{SC}(A)$ holds, and for $A>A_C$, $\mathcal{SC}(A)$ fails. Note that $A_C >A_0>0$. The remainder of this section is devoted proving the following proposition and its corollary below.

\begin{prop}\label{crit elem} Suppose that Theorem~\ref{u main} fails and let $A_C$ be defined as above. Then, there exists a solution $\vec u^*(t) \in \HH$ with defined on  $ [0, + \infty)$ so that $\mathcal{SC}(A_C, \vec u^*)$ fails. Moreover, the set 
\EQ{ \label{K+}
\K_+= \{ \vec u^*(t) \mid t \in [0, \infty)\} \subset \HH
}
is pre-compact in $\HH$. 
\end{prop}

\begin{cor}\label{global crit} Suppose that Theorem~\ref{u main} fails. Let $\vec u^*(t) \in \HH$ be the critical element constructed in Proposition~\ref{crit elem}. Then there is (a possibly different) nonzero critical element $\vec u_{\infty}(t) \in \HH$ which is global in both time directions. Moreover 
\EQ{
\| u_{\infty} \|_{S( [0, + \infty))} = \| u_{\infty}\|_{S(( - \infty, 0])} = \infty
}
and the set 
\EQ{\label{K}
\K= \{  \vec u_{\infty}(t) \mid t \in \R \} \subset \HH
}
is pre-compact in $\HH$. 
\end{cor}

\begin{proof} The proof generally follows the framework established in \cite[Proof of Proposition~$3.3$]{KM10}, with the main novelty here being the use of two different types of profiles. We give a fairly detailed sketch of the argument, but at times we refer the reader to \cite{KM10} where the now standard arguments apply, and rather focus our attention on how the argument must be adapted to account for the presence of both Adkins-Nappi and Euclidean profiles in our profile decompositions.   

By the definition of $A_C$ we can find $A_n \searrow A_C$ together with initial data $\vec u_n(0) = (u_{n, 0}, u_{n, 1}) \in \HH$ with corresponding solutions $\vec u_n(t) \in \HH$ to \eqref{u eq} so that 
\EQ{
  \sup_{t \in [0, T_+( \vec u_n))} \| \vec u_n \|_{\HH} \le A_n,  \mand
\| u_n \|_{S([0, T_+(\vec u_n))} = + \infty
}
By Lemma~\ref{BG}, and \eqref{AN BG} we have the following profile decomposition for the sequence  $\vec u_n(0)$, 
\EQ{ \label{BG ce}
 &u_{n, 0} = \sum_{1\le j<K}  \frac{1}{\la_{n,j}^{\frac{1}{2}}}V^j_{e, L}\left(\frac{-t_{n,j}}{  \la_{n,j}}, \frac{ \cdot}{ \la{_n,j}}\right) + \sum_{1\le \ell<K}  U^{\ell}_{a,L}\left(-t_{n,\ell}, \cdot \right) + \ga_{n, 0}^K, \\
 &  u_{n, 1} =  \sum_{1\le j<K}  \frac{1}{ \la_{n,j}^{\frac{3}{2}}}\p_tV^j_{e, L}\left(\frac{-t_{n,j}}{  \la_{n,j}}, \frac{ \cdot}{ \la_{n,j}}\right)+ \sum_{1\le \ell<K}  \p_tU^{\ell}_{a,L}\left(-t_{n,\ell}, \cdot \right)  + \ga_{n, 1}^K.
 } 
 with $\vec \gamma_n^K(t) = S(t) ( \ga_{n, 0}^K, \ga_{n, 1}^K)$,  $t_{n, j}$, $ \la_{n, j}$ and $t_{n, \ell}$ as in Lemma~\ref{BG}. We also have the corresponding nonlinear profiles as in Section~\ref{nonlinear profiles}, and we will continue to use the notation 
 \ant{
 V_{e, n}^j(t, r):= \frac{1}{ \la_{n, j}^{\frac{1}{2}}} V^j_{e}\left( \frac{ t- t_{n, j}}{ \la_{n, j}}, \frac{r}{ \la_{n, j}} \right)
}
for the Euclidean nonlinear profiles, and 
\ant{
U_{a, n}^{\ell}(t, r):=   U_a^{\ell}(t - t_{n, j}, r)
}
for the Adkins-Nappi nonlinear profiles. 

The goal will be to show that, in fact, there is only one nonzero  Adkins-Nappi profile in~\eqref{BG ce} and no nonzero Euclidean profiles. 

To begin we quickly note that only finitely many nonlinear profiles can be non-scattering. Indeed,  by the Pythagorean decomposition of the $\HH$ norm in~\eqref{HH orth}, we can apply Proposition~\ref{small data}, as well as the small data theory for~\eqref{defoc v eq}, and find $K_0>0$ so that for all $j, \ell > K_0$ we have $$I_{\max}( \vec V_e^j) = I_{\max}(\vec U_a^{\ell})  = (- \infty, + \infty)$$ and there exists $C>0$ so that  
\ant{
 \|U_a^{\ell} \|_{X(\R)} \le C\|\vec U_{a, L}(0)\|_{\HH}, \mand \|V_e^j\|_{X(\R)} \le C \|\vec V_{e, L}(0) \|_{\HH}
 }
where again $X(I) := L^{\infty}_t(I; \HH) \cap S(I)$. 

Next, we show that at least one of the nonlinear profiles must be non-scattering. We accomplish this in the following claim. 

\begin{claim}\label{bad profiles} It is impossible that for all $1 \le j \le K_0$ and $1 \le \ell \le K_0$ we have 
\EQ{ \label{all scat profiles}
 \| V_{e}^{j} \|_{S \left(-\frac{t_{n, j}}{ \la_{n, j}}, T_+( \vec V^j_e)\right)}< \infty, \mand  \|U_{a}^{\ell}\|_{S(-t_{n, \ell}, T_+( \vec U_a^{\ell}))} < \infty.
 }
  \end{claim}
  To prove the claim, we note that if~\eqref{all scat profiles} holds, then by~\eqref{ftbuc} we have for all $j, \ell$,  $T_+( \vec V_e^j) = T_+( \vec U_a^{\ell}) = + \infty$. We then apply Proposition~\ref{nonlinear profiles} with the time sequence  $s_n  = \infty$ for all $n$, which then says that for $n$ large enough, $u_n$ is global and 
  \ant{
    \limsup_{n \to \infty} \| u_n\|_{S([0, \infty))} \le C < \infty,
   }
   which contradicts the definition of the $\vec u_n$, establishing Claim~\ref{bad profiles}.
   
   By Claim~\ref{bad profiles}, after rearranging the profiles, we can then find $1 \le J_1   \le K_0$ and $1 \le L_1 \le K_0$ so that  
   for all $1 \le j \le  J_1 $ we have 
   \ant{
   \|V_e^j\|_{S \left( - \frac{t_{n, j}}{ \la_{n, j}}, T_+( \vec V_e^j) \right)} =  \infty
   }
   and for all $1 \le \ell \le L_1$ we have 
   \ant{
   \|U_a^{\ell}\|_{S \left( - t_{n,  \ell}, T_+( \vec U_a^{\ell}) \right)} =  \infty
}
Moreover, for all $j > J_1$, respectively $ \ell> L_1$, we have $T_+(\vec V_e^j) = + \infty$, respectively, $T_+( \vec U_a^{\ell}) = +\infty$, and 
\ant{
 &\|V_e^j\|_{S \left( - \frac{t_{n, j}}{ \la_{n, j}}, T_+( \vec V_e^j) \right)} < \infty,  \, \quad \textrm{respectively},\\
 & \|U_a^{\ell}\|_{S \left( - t_{n,  \ell}, T_+( \vec U_a^{\ell}) \right)} < \infty.
}
Next, we define sequences of times: 
\EQ{
&T^+_{j, k} := \begin{cases} T_+( \vec V_{e}^j) -  \frac{1}{k} &\mif T_+(\vec V_e^j)< \infty \\
k &\mif  T_+(\vec V_e^j)= \infty 
\end{cases}
 \\
&T^+_{\ell, k} := \begin{cases} T_+( \vec U_{a}^{\ell}) -  \frac{1}{k} &\mif T_+(\vec U_a^{\ell})< \infty \\
k &\mif  T_+(\vec U_a^{\ell})= \infty 
\end{cases}
 } 
Then we define $s_{j, k}^n$, resp. $s_{\ell, k}^n$, by 
\EQ{
&\frac{ s_{n, j}^k - t_{n, j}}{\la_{n, j}} = T^+_{j, k}\\
& s_{n, \ell}^k - t_{n, \ell}  = T^+_{\ell, k}
}
Finally, we set 
\EQ{
s_n^k := \min_{1 \le j \le J_1, \, \, 1 \le \ell \le L_1} \{ s_{n, j}^k, \, s_{n, \ell}^k\}
}
The reason for this definition is that by construction, for all $n$ large, for all $j \ge 1$ and for all $ \ell \ge 1$ we have 

\ant{
&\frac{ s_{n}^k - t_{n, j}}{\la_{n, j}} < T_{+}(\vec V_e^j) \, \mand   \|V_e^j\|_{S \left( - \frac{t_{n, j}}{ \la_{n, j}},\, \frac{ s_{n}^k - t_{n, j}}{\la_{n, j}}  \right)} < \infty,\\
& s_{ n}^k - t_{n, \ell}  <T_+( \vec U_a^{\ell}), \mand \|U_a^{\ell}\|_{S \left( - t_{n,  \ell}, \,  \, s_{n}^k - t_{n, \ell} \right)} < \infty.
}
Therefore the sequence $s^k_n$ satisfies the conditions of Proposition~\ref{nonlin profile}. In fact, %we can find subsequences $n = n(m, k)$, $K = K(m , k)$ so that  the
the following nonlinear profile decomposition holds on the interval  
$ t \in [0, s_n^k)$, 
\EQ{ 
 u_{n}(t,r) &= \sum_{1\le j<K}  V^j_{e, n}(t, r) + \sum_{1\le \ell<K}  U^{\ell}_{a,n}(t,r)  + \ga_{n}^{K}(t) + \eta_{n}^{K}(t), 
 %\\
% &  \p_t u_n(t) =  \sum_{1\le j<k}  \frac{1}{ \la_{n,j}^{\frac{3}{2}}}\p_tV^j_{e}\left(\frac{t-t_{n,j}}{  \la_{n,j}}, \frac{ \cdot}{ \la_{n,j}}\right)+ \sum_{1\le \ell<k}  \p_tU^{\ell}_{a,L}\left(t-t_{n,\ell}, \cdot \right)  + \p_t\ga_{n,L }^k(t).
 } 
where for each $k \in \N$, 
\ant{
&\lim_{K \to \infty} \limsup_{n \to \infty} \| \eta_n^K\|_{S([0,  s^{k}_n))} + \|  \vec \eta_n^{K}\|_{L^{\infty}_t([0, s_{n}^k); \HH)}  = 0, \\
&\lim_{K \to \infty} \limsup_{n \to \infty} \| \ga_{n}^K  \|_{S([0, s^{k}_n))} =0.
}
Next, we claim that there exists a subsequence $k_{\al}  \to \infty$ so that {\em either} there exist $\ell_0$ with $1 \le \ell_0 \le L_1$ with 
\EQ{\label{l case}
 s_n^{k_{\al}} = s_{n, \ell_0}^{k_{\al}} \quad \textrm{for all} \, \, \al,
 }
 {\em  and/or} there exists $j_0$ with $1 \le j_0 \le J_1$ so that 
 \EQ{\label{j case}
 s_n^{k_{\al}} = s_{n, j_0}^{k_{\al}} \quad \textrm{for all} \, \, \al.
 }
This follows directly from the pigeonhole principle since both $L_1$ and $J_1$ are finite so either~\eqref{l case} holds for infinitely many $k$'s and/or~\eqref{j case} holds for  infinitely many $k$'s. Thus we have potentially two cases to study. 
\begin{flushleft} \textbf{Case $1$:} Assume \eqref{l case} holds so that $\ell_0$ corresponds to an Adkins-Nappi profile $ \vec U_a^{\ell_0}$. We note that for $n$ large enough
\end{flushleft}
\EQ{
\|U_a^{\ell_0}\|_{S(-t_{n, \ell_0}, T_+( \vec U_{a}^{\ell_0}))} =  \infty.
}
Also, recall that by construction, either $t_{n, \ell_0} = 0$ for all $n$ or, $-t_{n, \ell_0} \to - \infty$. In either case, $0   \ge  -t_{n, \ell_0}$ for $n$ large enough so that we have 
\ant{
\|U_a^{\ell_0}\|_{S([0, T_+( \vec U_{a}^{\ell_0}))} =  \infty.
}
By the definition of $A_C$ and since $ \vec U_{a}^{\ell_0}$ is an Adkins-Nappi nonlinear profile, i.e., solution to~\eqref{u eq}, we have 
\EQ{\label{A big}
A^2 := \sup_{t \in [0, T_+( \vec U^{\ell_0}_a))} \|  \vec U^{\ell_0}_a(t) \|_{\HH}^2 \ge A_C.
}
Observe that by the definition of the $T^+_{\ell_0, k}$, 
\ant{
A_k^2 := \sup_{t \in [0, T^+_{\ell_0, k})} \| \vec U_{a}^{\ell_0}(t) \|_{\HH}^2
} 
satisfies $\lim_{k \to \infty} A_k^2 = A^2$. The point is that for each $k$,  we can find $T_{\ell_0, k} \in [0, T_{\ell_0, k}^+]$ so that 
\ant{
A_k^2 = \| \vec U_{a}^{\ell_0}(T_{\ell_0, k}) \|_{\HH}^2.
}
We now define $\tau_{n, \ell_0}^k$ so that
\ant{
\tau_{n, \ell_0}^k -t_{n, \ell_0}  = T_{\ell_0, k},
}
and we remark that for $n$ large enough, $\tau_{n, \ell_0}^k \ge 0$, and $\tau_{n, \ell_0}^k \le s_{n, \ell_0}^k$.  Since $s_{n, \ell_0}^{k_{\al}} = s_{n}^{k_{\al}}$ for all $\al$, the nonlinear profiles $\vec U_{a,n}^{\ell}(\tau_{n, \ell_0}^{k_{\al}})$, $ \vec V_{e, n}^{j}( \tau_{n, \ell_0}^{k_{\al}})$ are defined for all $ \ell \ge 1$ and $j \ge 1$.  

We make the following observation: For every $\al$ with $k_{\al}$ fixed and  fixed $K$ large enough we have the following almost orthogonality
\EQ{ \label{ortho at tau}
\|\vec u_n(\tau_{n, \ell_0}^{k_{\al}})\|_{\HH}^2 &= \sum_{1 \le j < K} \|\vec V_{e, n}^{j}( \tau_{n, \ell_0}^{k_{\al}}) \|_{\HH}^2 + \sum_{1 \le \ell < K} \|\vec U_{a, n}^{\ell}( \tau_{n, \ell_0}^{k_{\al}}) \|_{\HH}^2 \\
& \quad + \|\vec \gamma_n^K( \tau_{n, \ell_0}^{k_{\al}}) \|_{\HH}^2 +o_n(1) \mas n \to \infty.
}
The claim follows directly from the orthogonality of the parameters in~\eqref{parameters diverge} and for a detailed argument we refer the reader to \cite[Proof of $(3.22)$]{KM10}. Using~\eqref{ortho at tau} we can deduce that 
\ant{
A_{n}^2 \ge A_{k_{\al}}^2 + o_n(1) \mas n \to \infty.
}
Letting $n \to \infty$ we see that $A_C^2 \ge A_{k_{\al}}^2$. Then we let $k_{\al} \to \infty$ to obtain $A_{C}^2 \ge A^2$. Combining this \eqref{A big} we see that 
\ant{
A = A_C.
}
In fact, this also shows that all of the other profiles must be $ \equiv 0$ and that $\vec \ga_n^K  \to 0$ in $\HH$ as $n \to \infty$. To see this, suppose that for some $m \neq \ell_0$ we have a nonzero profile. To not have to distinguish between Euclidean and Adkins-Nappi, we will simply denote this nonzero profile by $\vec U^{m}$. For any $ \eps>0$ we can choose $k_{\al}$ so large that $\abs{A_k^2 - A_C^2}< \eps$. Then by~\eqref{ortho at tau} we have 
\ant{
A_n^2 \ge A_{C}^2  - \eps + \|\vec U^m_n(\tau_{n, \ell_0}^{k_{\al}}) \|_{\HH}^2 + o_n(1).
}
Taking $n$ arbitrarily large we can make $\|\vec U^m_n(\tau_{n, \ell_0}^{k_{\al}}) \|_{\HH}^2$ arbitrarily small. By the small data theory, i.e. Proposition~\ref{small data} or the corresponding small data result if the profile is Euclidean, we can then make $\sup_{t \in \R} \|\vec U^m(t) \|_{\HH}$ arbitrarily small, which means that the corresponding  linear profile $\vec U^m_{L}$ must be identically zero, as desired. A similar argument shows the vanishing of the error $\vec \ga^K_n$ as $n \to \infty$ in $\HH$. 

Thus, there can only be one nonzero profile $\vec U_{a}^{\ell_0}$. 

We next show that this is in fact, the only possibility, i.e., that~\eqref{l case} must hold and we cannot have~\eqref{j case} holding but~\eqref{l case} not holding. 
\begin{flushleft} \textbf{Case $2$:} Assume \eqref{j case} holds so that $j_0$ corresponds to a Euclidean profile $ \vec V_e^{j_0}$. We show that this is in fact impossible. To see this note that for $n$ large enough
\end{flushleft}
\EQ{
\|V_e^{j_0}\|_{S(-\frac{t_{n, j_0}}{ \la_{n, j_0}}, T_+( \vec V_{e}^{j}))} =  \infty.
}
By Theorem~\ref{euc thm} we must then have 
\EQ{
\sup_{t \in \left(-\frac{t_{n, j_0}}{ \la_{n, j_0}}, T_+( \vec V_{e}^{j})\right)} \| \vec V_e^{j_0}(t) \|_{\HH} = +\infty.
}
We define $A_k$, $T_{j_0, k}$, and $\tau_{n, j_0}^k$ exactly as in Case $1$ so that % and we find $\tau_{n, j_0}^k$ so that 
%\ant{
%\frac{ \tau_{n, j_0}^k - t_{n, j_0}}{ \la_{n, j_0}} = T_{j_0, k}
%}
%Hence we have 
\ant{
& \|\vec V_e^{j_0}(T_{j_0, k})\|_{\HH}^2 = A_k^2 \to \infty,\\
&\frac{ \tau_{n, j_0}^k - t_{n, j_0}}{ \la_{n, j_0}} = T_{j_0, k}.
}
%By the same construction as in Case $1$  we can find  and we can find $T_m \nearrow T_+( \vec V_e^{j_0})$ so that 
%\ant{
 %\| \vec V_e^{j_0}( T_m) \|_{\HH} \to  \infty \mas m \to \infty
 %}
 We again have the orthogonality as in~\eqref{ortho at tau}, only now at time $\tau_{n, j_0}^k$. But this yields  
 \ant{
 A_n^2 \ge A_{k_{\al}}^2 + o_n(1) .
 }
 Letting $n  \to \infty$ yields $A_C \ge A_{k_{\al}}$ which is a contradiction since $A_C$ is assumed to be finite, but $A_{k_{\al}}^2 \to \infty$  as $k_{\al} \to \infty$. Therefore~\eqref{j case} cannot happen and~\eqref{l case} must happen. Which means that there is  one nonzero profile, namely $U_{a}^{\ell_0}$ from Case~$1$.
 
 We  let $\vec u^*(t) :=  \vec U_a^{\ell_0}(t)$ be the critical element. We still need to show the compactness property~\eqref{K+}. 

Set 
\EQ{\label{K+ 1}
\K_{+}:= \{ \vec u^*(t) \mid t \in[0, T_+(\vec u^*)\} \subset \HH.
}
We first show that the trajectory $\K_+$ is pre-compact in $\HH$ . If the pre-compactness of~\eqref{K+ 1} fails, then there exists a $\de>0$ so that for some sequence of times $t_n \to T_+(\vec u^*)$ we have 
\EQ{ \label{noncompact} 
\| \vec u^*(t_n) - \vec u^*(t_m)\|_{\HH} > \de, \quad \forall m>n.
}
We can apply Lemma~\ref{BG} to the sequence $\vec u^*(t_n)$ and conclude by the same arguments as before that there can only be one nonzero profile and it must be Adkins-Nappi. Thus we have 
\EQ{ \label{one prof} 
 \vec u^*(t_n) = \vec U_{L}(\tau_n) + \vec \ga_{n}(0)
 }
 where $\vec U_L(t)$ and $\vec  \ga_n(t)$ are free waves, $\tau_n$ is some sequence of times and $\| \vec \ga_n \|_{\HH}  \to 0$ in $\HH$ as $n \to \infty$.  
 If  $\tau_n \to \tau_{\infty} \in \R$ then~\eqref{noncompact} and~\eqref{one prof} lead to an immediate contradiction. If $\tau_n \to + \infty$ then 
 \ant{
 \|U_{L}( \cdot + \tau_n) \|_{S([0, \infty))} \to 0 \mas n \to \infty,
 } 
 which implies, via~\eqref{one prof} and the local well-posedness theory that $$\| u^*( \cdot + t_n)\|_{S([0, \infty))} \le C < \infty$$ for large $n$, which is a contradiction. Finally, assume $\tau_n \to - \infty$. But then 
 \ant{
 \|U_{L}( \cdot + \tau_n) \|_{S([- \infty, 0])} \to 0 \mas n \to \infty,
 }
 which implies that 
 \ant{
 \| u^*( \cdot + t_n)\|_{S([- \infty, 0])} \le C < \infty
 }
 for all large $n$ for some fixed constant $C>0$. However, this is a  contradiction as well since $t_n \to T^+( \vec u^*)$.  Therefore~\eqref{noncompact} must be false and therefore~\eqref{K+ 1} is pre-compact as desired. 
 
 Next we claim that in fact $T_+( \vec u^*) = + \infty$. To see this, suppose $T_+( \vec u^*)< \infty$. Then let $t_n \to T_+( \vec u^*)$ be any sequence. Since $\K_+$ is pre-compact we can find a subsequence, still denoted by $t_n$  so that 
 \ant{
  \vec u^*( t_n) \to \vec u_{\infty} \, \,  \textrm{in} \, \,  \HH
  }
 Denote by $\vec u_{\infty}(t)$ the solution to~\eqref{u eq} with initial data  $\vec u_{\infty}$ at the initial time $T_+(  \vec u^*)$. By the local well-posedness theory, we can choose $\de>0$ so that $\vec u_{\infty}(t)$ exists on the interval $I_{\de} = (T_+( \vec u^*) - \de, T_+( \vec u^*) + \de)$ and satisfies 
 \ant{
 \| u_\infty \|_{S(I_{\de})} \le C < \infty
 }
 Now, define $\vec h_n(t)$ to be the solution to~\eqref{u eq} with initial data at time $T_+( \vec u^*)$ which is equal to  $\vec u^*(t_n)$. But then, by the Perturbation Lemma, i.e., Lemma~\ref{perturbation}, we can conclude that for $n$ large enough, $\vec h_n(t)$ exists on the interval $$I_{\de/2}= (T_+( \vec u^*)- \frac{\de}{2}, T_+( \vec u^*) + \frac{\de}{2})$$ and satisfies 
 \EQ{ \label{bounded h}
 \limsup_{n \to \infty} \| h_n\|_{S(I_{\de/2})} < \infty.
 }
 However, note that by construction $h_n(t)=\vec u^*(t_n + t- T_+( \vec u^*))$ and so~\eqref{bounded h} contradicts~\eqref{ftbuc} as we are assuming $T_+( \vec u^*)< \infty$.  Thus $T_+( \vec u^*) = + \infty$. 
 This completes the proof. 
\end{proof}

 Lastly we would like to pass to a new nonzero critical element that is global in both time directions and has a pre-compact  trajectory as both $t \to  \pm \infty$ as well.  This is the content of Corollary~\ref{global crit}
 \begin{proof}[Proof of Corollary~\ref{global crit}] 
 Let $\vec u^*(t)$ be the critical element in Proposition~\ref{crit elem}. We extract the new global critical element from the sequence 
 \ant{
\{ \vec u^*(n)\}_{n = 1}^{\infty} \subset \K_+ \subset \HH.
}
Since we know that $\|u^*\|_{S([0, \infty))} =  \infty$ we can find $A_0>0$ so that  $\| \vec u^*(n) \|_{\HH} \ge A_0$. If we couldn't do this we could use the small data theory to prove that $\vec u^*$ scatters at $t=+ \infty$, which would be a contradiction.  

By the pre-compactness of $\K_+$ we can find $\vec u_{\infty}$ so that 
\ant{
 \vec u^*(n) \to \vec u_{\infty} \in \HH.
 }
  By construction the $S$-norm of the evolution $\vec u_{\infty}(t)$ is infinite in both time directions, and the same argument as in the proof of Proposition~\ref{crit elem} shows the pre-compactness of $\mathcal{K}$.
  \end{proof}

%%%%%%%%%%%%%%%%%%%%%%%%%%%%%%%%%%%%%%%%%%%%%%%%%%%%%%%%%%%%%%%%%%%%%%%%
%%%%%%%%%%%%%%%%%%%%%%%%%%%%%%%%%%%%%%%%%%%%%%%%%%%%%%%%%%%%%%%%%%%%%%%%
%%%%%%%%%%%%%%%%%%%%%%%%%%%%%%%%%%%%%%%%%%%%%%%%%%%%%%%%%%%%%%%%%%%%%%%%

\section{Rigidity Argument}\label{rigidity}
In this section we complete the proof of Theorem~\ref{u  main} and the equivalent result Theorem~\ref{main} by showing that the critical element constructed in Corollary~\ref{global crit} cannot exist. This is done by way of a rigidity result, which excludes the possibility of nonzero, pre-compact trajectories as in~\eqref{K}. 
\begin{prop}\label{rigid}Let $\vec u(t) \in \HH(\R^5)$ be a global solution to~\eqref{u eq} and suppose that the trajectory 
\EQ{
K:=\{\vec u(t) \mid t \in \R\}
}
is pre-compact in $\HH(\R^5)$. Then $\vec u(t) = (0, 0)$. 
\end{prop}
We will make use of the following simple consequence of the above, namely that the compactness of $K$ implies that the $\dot{H}^1 \times L^2$ norm of $\vec u(t)$ vanishes on any exterior cone $\{ r \ge R + \abs{t}\}$ as $\abs{t} \to \infty$. 
\begin{cor}\label{vanish tail} Let $\vec u(t)$ and $K$ be as in Proposition~\ref{rigid}. Then for any $R \ge0$ we have 
\EQ{\label{compact ext en}
\| \vec u(t) \|_{\dot H^1 \times L^2( r \ge R+\abs{t})} \to 0 \mas t \to \pm \infty
}
where 
\ant{
\| \vec u(t) \|_{\dot H^1 \times L^2( r \ge R+\abs{t})}^2 := \int_{R + \abs{t}}^{\infty}( u_t^2(t, r) + u_r^2(t, r) )\, r^4 \, dr
}
\end{cor}

The proof will proceed in several steps. The argument is reminiscent of the one presented in~\cite{KLS} and finds its inspiration in the so-called ``channels of energy" method pioneered in the seminal works \cite{DKM4, DKM5}, albeit  in a different context. A crucial ingredient in the proof are the exterior energy estimates for the underlying free radial wave equation in $5$ dimensions proved in~\cite{KLS}. 

\begin{prop}\label{linear prop}\cite[Proposition~$4.1$]{KLS}
Let $\Box w=0$ in $\R^{1+5}_{t,x}$ with radial data $(f,g)\in \dot H^1\times L^2(\R^5)$.
Then with some absolute constant $c>0$ one has for every $a>0$
\EQ{
\label{R5ext}
\max_{\pm}\;\limsup_{t\to\pm\I} \int_{r>a+|t|}^\I(w_t^2 + w_r^2)(t,r) r^4\, dr \ge c \| \pi_a^\perp (f,g)\|_{\dot H^1\times L^2(r>a)}^2
}
where $\pi_a=\Id-\pi_a^\perp$ is the orthogonal projection onto the plane $$P(a):=\{( c_1 r^{-3}, c_2 r^{-3})\:|\: c_1, c_2\in\R\}$$
in the space $\dot H^1\times L^2(r>a)$.  The left-hand side of~\eqref{R5ext} vanishes for all data in this plane. 
\end{prop}
\begin{rem}One should note that the appearance of the projections $\pi^{\perp}_a$ on the right-hand-side of~\eqref{R5ext} is due to the fact that $r^{-3}$ is the Newton potential in $\R^5$. To be precise, consider initial data $(f, 0) \in \dot H^1 \times L^2 (r \ge R)$ which satisfies $(f, 0) = (r^{-3}, 0)$ on $r \ge R>0$, with $f(r)$ vanishing on $r \le R/2$. Then the corresponding free evolution $w(t, r)$ is given by $w(t, r) = r^{-3}$ on the region $r \ge R+\abs{t}$ by finite speed of propagation. Needless to say, the left-hand-side of~\eqref{R5ext} vanishes for this solution as $t \to \pm \infty$, and thus precludes an estimate without the projection. The other family of enemies is generated by taking data $(0, g) = (0, r^{-3})$ on the exterior region $r\ge R>0$ which has solution $w(t, r) = tr^{-3}$ on $r \ge R+ \abs{t}$. 
\end{rem}
\begin{rem}
The orthogonal projections $\pi_a$, $\pi_a^{\perp}$ are given by 
\ant{
&\pi_a(f, 0) = a^3r^{-3} f(a), \quad 
\pi_a(0, g) = a r^{-3} \int_a^{\infty} g( \rho)  \rho \, d \rho,\\
&\pi_a^{\perp}(f, 0) = f(r) - a^3r^{-3} f(a), \quad \pi_a^{\perp}(0, g) = g(r) - a r^{-3} \int_a^{\infty} g( \rho)  \rho \, d \rho,
}
and thus we have 
\ant{
\|\pi_a(f,g)\|_{\dot{H}^1 \times L^2(r>a)}^2 &= 3a^3 f^2(a) + a\left(\int_a^{\infty} rg(r) \, dr\right)^2 \\
\|\pi_a^{\perp}(f, g)\|_{\dot{H}^1 \times L^2(r>a)}^2&= \int^{\infty}_a f_r^2(r) \, r^4 \, dr - 3a^3f^2(a)\\
& \quad + \int_a^{\infty} g^2(r) \, r^4 \, dr - a\left(\int_a^{\infty} rg(r) \, dr\right)^2.
}
\end{rem}

The general idea is that the exterior energy decay~\eqref{compact ext en} can be combined with the exterior energy estimates for the underlying free equation in Proposition~\ref{linear prop}, to obtain precise {\em spacial} asymptotics for $u_0(r)= u(0, r)$  and $u_1(r) = u_t(0, r)$ as $r \to \infty$, namely, 
\EQ{\label{space asym}
&r^3 u_0(r)  \to \ell_0 \mas r \to \infty\\
&r \int_r^{\infty} u_1( \rho) \rho \, d \rho  \to 0 \mas r \to \infty
}
We then argue by contraction to show that $\vec u(t) =(0, 0)$ is the only solution with both a pre-compact trajectory and initial data with the above asymptotics.  We proceed with the proof of Proposition~\ref{rigid}
\subsection{Step 1} In this first step, we will use the exterior energy estimates for the free radial wave equation in Proposition~\ref{linear prop} along with Corollary~\ref{compact ext en} to derive the following inequality for our pre-compact trajectory $\vec u(t)$. We note that by compactness this result will hold uniformly in time.  
\begin{lem} \label{project ineq}There exists $R_0>0$ so that for all $R>R_0$ and for all $t \in \R$ we have 
\EQ{ \label{key ineq}
\| \pi^{\perp}_R \vec u(t) \|_{\dot{H}^1 \times L^2(r \ge R)} \lesssim R^{-1} \| \pi_R \vec u(t)\|_{\dot{H}^1 \times L^2(r \ge R)}^3 %+ R^{-4} \| \pi_R \vec u(t)\|_{\dot{H}^1 \times L^2(r \ge R)}^5
}
where $P(R):= \{(c_1r^{-3}, c_2r^{-3}) \mid, c_1, c_2 \in \R\}$, $\pi_R$ denotes orthogonal projection onto $P(R)$ and $\pi_R^{\perp}$ denotes orthogonal projection onto the orthogonal complement of the plane $P(R)$ in the space $\dot{H}^1 \times L^2(r>R)(\R^5)$. We remark that the constant in~\eqref{key ineq} is uniform in $t \in \R$. 
\end{lem}

In order to prove Proposition~\ref{key ineq} we first need a preliminary result concerning a modified Cauchy problem that has a small data theory in the energy space and is designed to capture the dynamics of our critical element on the exterior cones $\CC_R:=\{(t, r) \mid r \ge R + \abs{t}\}$. In order to have a meaningful comparison in the energy space between a nonlinear wave and the underlying free evolution with the same initial data  we need to first put ourselves in a small data setting, where the Duhamel formula and Strichartz estimates give us the tools we need  to apply Proposition~\ref{linear prop}. The fact that we are only considering the evolution on the exterior cone $\CC_R$ allows us to truncate the initial data and the nonlinearity in a way that will render the initial value problem subcritical relative to the energy, while still preserving the flow on $\CC_R$. 

With this in mind we fix a smooth function $\chi \in C^{\infty}([0, \infty))$ where $\chi(r) = 1$ for $r \ge 1$ and $\chi(r) = 0$ on $r \le 1/2$. Then set $\chi_R(r):= \chi(r/R)$ and for each $R>0$ we consider the modified Cauchy problem: 
\EQ{\label{h eq}
&h_{tt} - h_{rr} - \frac{4}{r} h_r = \NN_R(r, h)\\
& \NN_R(r, h) = -\chi_RZ_1(rh)h^3 - \chi_R Z_2(rh)h^5\\
&\vec h(0) = (h_0, h_1)
}
where $Z_1$ and $Z_2$ are defined as in~\eqref{u eq}. The benefit of this modification is that forcing the nonlinearity to have support outside the ball of radius $R$ removed the super-critical nature of the problem and allows for a small-data theory in $\dot{H}^1 \times L^2$ via Strichatz estimates and the usual contraction mapping based argument. In order to formulate the small data theory for~\eqref{h eq} we define the norm $Z(I)$ where $0 \in I \subset \R$ is a time interval by 
\EQ{ \label{Z norm def}
\|\vec h\|_{Z(I)} = \|h\|_{L^{\frac{7}{3}}_t(I; L^{\frac{14}{3}}_x(\R^5))} + \|\vec h(t) \|_{L^{\infty}_t(I; \dot H^1 \times L^2)}
}

\begin{lem} \label{h small data}
There exists a $\de_0$ small enough, so that for all  $R>0$ and all initial data $\vec h(0)=(h_0, h_1) \in \dot{H}^1 \times L^2 (\R^5)$ with 
\ant{
\| \vec h(0) \|_{ \dot H^1 \times L^2(\R^5)} < \de_0
}
there exists a unique global solution $\vec h(t) \in \dot{H}^1 \times L^2$ to \eqref{h eq}. In addition $\vec h(t)$ satisfies 
\EQ{\label{h small}
\|\vec h\|_{Z(\R)} \lesssim \|\vec h(0)\|_{\dot{H}^1 \times L^2} \lesssim  \de_0
}
Moreover, if we denote the free evolution of the same data by $\vec h_L(t):= S(t) \vec h(0)$, then we have 
\EQ{\label{h-hL}
\sup_{t \in \R} \|  \vec h(t) - \vec h_L(t) \|_{ \dot{H}^1 \times L^2 } \lesssim R^{-1} \| \vec h(0) \|_{\dot H^1 \times L^2}^3 + R^{-4}\| \vec h(0) \|_{\dot H^1 \times L^2}^5 
}
\end{lem}
Besides Strichartz estimates and the Duhamel formula, the key ingredient to the proof of Lemma~\eqref{h small data} is the Strauss' lemma for radial functions, namely if $f \in \dot{H}^1(\R^5)$ then for each $r>0$ we have the estimate 
\EQ{\label{Strauss}
\abs{f (r)} \le Cr^{-3/2} \|f\|_{\dot{H}^1}
}
%The proof of~\eqref{Strauss} follows from the Fundamenal Theorem of Calculus and H\"older's inequality. 
\begin{proof}[Sketch of Proof of Lemma~\ref{h small data}] The small data global well-posedness theory, including the estimate~\eqref{h small} follows from the usual contraction and continuity arguments based on the Strichartz estimates in Proposition~\eqref{strich}. In particular, we will make use of the estimates 
\EQ{
\| v\|_{L^{\frac{7}{3}}_t L^{\frac{14}{3}}_x} + \|\vec v\|_{L^{\infty}_t(\dot{H}^1 \times L^2)}  \lesssim \| \vec v(0)\|_{\dot{H}^1 \times L^2} + \|F\|_{L^1_tL^2_x}
}
for solutions $\vec v $ to the in-homogenous $5d$ wave equation, \eqref{lin wave}.  
Rather than review this standard argument, we just give details  for the estimate~\eqref{h-hL} which has the same flavor. By the Duhamel formula, Strichartz estimates,  and the uniform estimates~\eqref{Z bounds} for $Z_1, Z_2$ we have 
\ant{
\|\vec h(t)- \vec h_L(t)\|_{\dot{H}^1 \times L^2}  &\lesssim \|\NN_R( \cdot, h)\|_{L^1_t L^2_x}\\
& \lesssim \|\chi_RZ_1(r h) h^3\|_{L^1_tL^2_x} + \| \chi_R Z_2(r h) h^5\|_{L^1_t L^2_x}\\
& \lesssim \|\chi_R h^3\|_{L^1_tL^2_x} + \| \chi_R  h^5\|_{L^1_t L^2_x}\\
}
Next we estimate the right-hand-side above. We see that for each $t\in \R$ 
\ant{
\|\chi_R h^3\|_{L^2_x(\R^5)} &= \left(\int_0^{\infty} \chi_R^2(r) h^6(t, r)\, r^4 \, dr \right)^{\frac{1}{2}}\\
&\lesssim \left(\int_R^\I  h^6(t, r)\, r^4 \, dr \right)^{\frac{1}{2}}\\
&\lesssim (\sup_{r  \ge R} \abs{h(t, r)})^{\frac{2}{3}} \| h(t) \|_{L^{\frac{14}{3}}}^{\frac{7}{3}}\\
& \lesssim R^{-1} \|h(t)\|_{\dot{H}^1}^{\frac{2}{3}}  \| h(t) \|_{L^{\frac{14}{3}}}^{\frac{7}{3}}
}
where on the last line we used the Strauss estimate~\eqref{Strauss}. Thus, 
\ant{
\|\chi_R h^3\|_{L^1_tL^2_x} &\lesssim R^{-1} \|\vec h\|_{L^{\infty}_t(\dot{H}^1 \times L^2)}^{\frac{2}{3}}  \| h \|_{L^{\frac{7}{3}}_tL^{\frac{14}{3}}_x}^{\frac{7}{3}}\\
& \lesssim R^{-1}\| \vec h\|_{Z(\R)}^3 \lesssim R^{-1} \| \vec h(0) \|_{\dot{H}^1 \times L^2}^3
}
with the last inequality following from~\eqref{h small}. Similarly, 
\ant{
\|\chi_R h^5\|_{L^2_x(\R^5)} &= \left(\int_0^{\infty} \chi_R^2(r) h^{10}(t, r)\, r^4 \, dr \right)^{\frac{1}{2}}\\
&\lesssim (\sup_{r  \ge R} \abs{h(t, r)})^{\frac{8}{3}} \| h(t) \|_{L^{\frac{14}{3}}}^{\frac{7}{3}}\\
& \lesssim R^{-4} \|h(t)\|_{\dot{H}^1}^{\frac{8}{3}}  \| h(t) \|_{L^{\frac{14}{3}}}^{\frac{7}{3}}
}
and thus, 
\ant{
\|\chi_R h^5\|_{L^1_tL^2_x} &\lesssim R^{-4}\| \vec h\|_{Z(\R)}^5 \lesssim R^{-4} \| \vec h(0) \|_{\dot{H}^1 \times L^2}^5
}
which completes the proof of~\eqref{h-hL}. 
\end{proof}
\begin{rem}\label{u=uR}
We remark that for every $t \in \R$ the nonlinearity $\NN_R$ in~\eqref{h eq} satisfies
\ant{
\NN_R(r, u) = \NN(r, u):=- Z_1(ru)u^3 - Z_2(ru)u^5, \quad \forall r \ge R + \abs{t}.
}
Therefore, by finite speed of propagation we can  conclude that solutions to~\eqref{h eq} and~\eqref{u eq} agree on the exterior cone $\CC_R:=\{(t, r) \mid r \ge R + \abs{t}\}$.
\end{rem}
With the above remark in mind, we can now prove Lemma~\ref{project ineq}. 
\begin{proof}[Proof of Lemma~\ref{project ineq}] As always in this section, $\vec u(t)$ is our pre-compact trajectory. We will prove the Lemma first for time $t=0$. The proof for all times $t \in \R$ with   $R>R_0$ independent of $t$ will follow immediately from the pre-compactness of the set $K$. We begin by defining truncated initial data, $\vec u_R(0) = (u_{0, R}, u_{1, R})$ by
\EQ{
&u_{0, R}(r):= \begin{cases} u_0(r) \mfor r \ge R\\ u_0(R) \mfor 0 \le r \le R \end{cases}\\
&u_{1, R}(r):= \begin{cases} u_1(r) \mfor r \ge R\\ 0 \mfor 0 \le r \le R \end{cases}
}
This new, truncated initial data now has small $\dot H^1 \times L^2$ norm for large $R$. Indeed, 
\EQ{
\| \vec u_R(0)\|_{\dot{H}^1 \times L^2} = \|\vec u(0) \|_{\dot{H}^1 \times L^2(r \ge R)}
}
and so we can choose $R_0>0$ large enough so that for all $R \ge R_0$ we have 
\EQ{
\| \vec u_R(0)\|_{\dot{H}^1 \times L^2} \le \de \le \min(\de_0, 1)
}
where $\de_0$ is chosen as in Lemma~\eqref{h small data}. By Lemma~\eqref{h small data} we can associate to $\vec u_R(0)$ the solution $\vec u_R(t)$ to~\eqref{h eq}, which satisfies~\eqref{h small} and \eqref{h-hL}. We note that by Remark~\ref{u=uR} we have 
\EQ{\label{u=uR eq}
\vec u_R(t) = \vec u(t), \quad \forall (t, r) \in \CC_R.
}
Denote by $\vec u_{R, L}(t)$ the free evolution of the data $\vec u_{R}(0)$, i.e., $\vec u_{R, L}(t):= S(t) \vec u_{R}(0)$. By the triangle inequality and~\eqref{u=uR eq}, we have 
\ant{
\| \vec u(t) \|_{\dot{H}^1\times L^2(r \ge R+\abs{t})} &= \| \vec u_R(t) \|_{\dot{H}^1\times L^2(r \ge R+\abs{t})} \\
&\ge \|u_{R, L}(t) \|_{\dot{H}^1\times L^2(r \ge R+\abs{t})} - \|\vec u_R(t)- u_{R, L}(t) \|_{\dot{H}^1\times L^2(r \ge R+\abs{t})} 
}
Applying~\eqref{h-hL} to $\vec u_{R}(t)$ and taking $R>R_0$ large enough,  we can conclude 
\ant{
\|\vec u_R(t)- u_{R, L}(t) \|_{\dot{H}^1\times L^2(r \ge R+\abs{t})}  &\le  \|\vec u_R(t)- u_{R, L}(t) \|_{\dot{H}^1\times L^2} \\
&\lesssim R^{-1} \| \vec u_{R}(0)\|_{\dot{H}^1 \times L^2}^3 +R^{-4} \| \vec u_{R}(0)\|_{\dot{H}^1 \times L^2}^5\\
& = R^{-1} \| \vec u(0)\|_{\dot{H}^1 \times L^2(r \ge R)}^3 +R^{-4} \| \vec u(0)\|_{\dot{H}^1 \times L^2(r \ge R)}^5\\
& \lesssim R^{-1} \| \vec u(0)\|_{\dot{H}^1 \times L^2(r \ge R)}^3.
}
Combining the above two inequalities gives
\ant{
\| \vec u(t) \|_{\dot{H}^1\times L^2(r \ge R+\abs{t})} \ge  \|u_{R, L}(t) \|_{\dot{H}^1\times L^2(r \ge R+\abs{t})} - C_0R^{-1}\| \vec u(0)\|_{\dot{H}^1 \times L^2(r \ge R)}^3.
}
Letting $t \to \pm \infty$ -- the choice determined by Proposition~\ref{linear prop} -- we can  use Proposition~\ref{linear prop} to give a lower bound for the right-hand side and use Corollary~\ref{vanish tail} to see that the left-hand-side above goes to zero as $\abs{t} \to \infty$ and deduce the estimate
\ant{
 \| \pi^{\perp}_R  \vec u_{R}(0) \|_{ \dot{H}^1 \times L^2(r \ge R)} \lesssim R^{-1}\| \vec u(0)\|_{\dot{H}^1 \times L^2(r \ge R)}^3.
 }
 Since $\vec u_R(0)= \vec u(0)$ on $\{r \ge R\}$  we in fact have that 
 \ant{
 \| \pi^{\perp}_R  \vec u(0) \|_{ \dot{H}^1 \times L^2(r \ge R)} \lesssim R^{-1}\| \vec u(0)\|_{\dot{H}^1 \times L^2(r \ge R)}^3.
 }
 Finally, we can use the orthogonality of the projection $\pi_{R}$ to expand the right-hand side  
 \ant{
 \| \pi^{\perp}_R  \vec u(0) \|_{ \dot{H}^1 \times L^2(r \ge R)} \lesssim R^{-1}\left(\| \pi_R\vec u(0)\|_{\dot{H}^1 \times L^2(r \ge R)}^2  + \| \pi_R^{\perp}\vec u(0)\|_{\dot{H}^1 \times L^2(r \ge R)}^2 \right)^{\frac{3}{2}}.
 }
 To conclude, we simply choose $R_0$ large enough so that we can absorb the $\pi_R^{\perp}$ term on the right-hand-side into the left-hand-side and deduce that 
  \ant{
  \| \pi^{\perp}_R  \vec u(0) \|_{ \dot{H}^1 \times L^2(r \ge R)} \lesssim R^{-1}\| \pi_R\vec u(0)\|_{\dot{H}^1 \times L^2(r \ge R)}^3,
  } 
 which proves the lemma for $t=0$. To see that the inequality holds for all $t \in \R$ we note that by the pre-compactness of $K$ we can choose $R_0 = R_0( \de_0)$ so that  for all $R \ge R_0$ we have 
 \ant{
 \| \vec u(t) \|_{ \dot{H}^1 \times L^2( r \ge R)} \le \min(\de_0, 1)
 } 
 uniformly in $t \in \R$. Now we just repeat the entire argument above with truncated initial data at time $t=t_0$ and $R \ge R_0$ given by 
 \ant{
&u_{0, R, t_0}(r):= \begin{cases} u(t_0, r) \mfor r \ge R\\ u_0(t_0, R) \mfor 0 \le r \le R \end{cases}\\
&u_{1, R, t_0}(r):= \begin{cases} u_t(t_0, r) \mfor r \ge R\\ 0 \mfor 0 \le r \le R \end{cases}
}
This finishes the proof. 
\end{proof}

\subsection{Step 2}
Next, we will use the estimates in Lemma~\ref{project ineq} to deduce the asymptotic behavior of $\vec u(0, r)$ as $r \to \infty$ that was described in~\eqref{space asym}. We prove the following lemma. 

\begin{lem}\label{spacial decay}
Let $\vec u(t)$ be as in Proposition~\ref{rigid}. Then there exists $\ell_0 \in \R$ so that 
\begin{align}
r^3 u_0(r) \to \ell_0 \mas r \to \infty \label{u0 decay}\\
r \int_r^\I u_1( \rho) \rho \, d \rho \to 0 \mas r \to \infty \label{u1 decay}
\end{align}
Moreover the above convergence occurs at the rates 
\begin{align}
\label{u0 rate}\abs{r^3 u_0(r) - \ell_0}  = O(r^{-4}) \mas r \to \infty\\
\label{u1 rate}\abs{r \int_r^\I u_1( \rho) \rho \, d \rho } = O(r^{-2}) \mas r \to \infty
\end{align}
\end{lem}
For the proof of Lemma~\ref{spacial decay} it will be convenient to make the following substitutions. Define 
\EQ{\label{v def}
&v_0(t, r):= r^3 u(t, r), \\
&v_1(t, r):= r \int_r^\I u_1( \rho) \rho \, d \rho.
}
To simplify notation further we will often write $v_0(r) :=v_0(0, r)$ and $v_1(r):=v_1(0, r)$. With these definitions we can compute 
\EQ{\label{v u project} 
&\| \pi_R \vec u(t)\|^2_{\dot{H}^1 \times L^2(r \ge R)} = 3R^{-3} v_0^2(t, R) + R^{-1}v_1^2(t, R)\\
&\| \pi_R^{\perp} \vec u(t)\|^2_{\dot{H}^1 \times L^2(r \ge R)} = \int_R^\I \left( \frac{1}{r} \p_r v_0(t, r) \right)^2 \, dr + \int_R^\I \left( \p_r v_1(t ,r) \right)^2 \, dr
}
To begin, we rewrite the conclusions of Lemma~\ref{space asym} in terms of $(v_0, v_1)$. 
\begin{lem}\label{v project}
Let $(v_0, v_1)$ be defined as in~\eqref{v def}. Then there exists $R_0>0$ such that for all $R \ge R_0$ we have 
\ant{
\left( \int_R^\I \left( \frac{1}{r} \p_r v_0(t, r) \right)^2 + \left( \p_r v_1(t ,r) \right)^2 \, dr \right)^{\frac{1}{2}} \lesssim R^{-1} \left(3R^{-3} v_0^2(t, R) + R^{-1}v_1^2(t, R)\right)^{\frac{3}{2}}
}
with a constant that is uniform in $t \in \R$. 
\end{lem}

We will use Lemma~\ref{v project} to prove difference estimates. We let $\de_1>0$ be a small number to be determined below so that $\de_1 \le \de_0^2$ where $\de_0$ is as in the small data theory in Lemma~\ref{h small data}. We also define $R_1 = R_1(\de _1)$ large enough so that for all $R \ge R_1$ we have 
\EQ{\label{R1 de1}
&\| \vec u(t) \|_{\dot{H}^1 \times L^2(r \ge R)}^2 \le  \de_1 \le \de_0^2\\
&R_1^{-1} \le  \de_1
}
Such an $R_1$ exists by the pre-compactness of $K$. 
\begin{cor}\label{diff esti}
Let $R_1$ be as above. Then for all $R_1 \le r \le r' \le 2r$ and for all $t \in \R$  we have 
\EQ{\label{v0 diff1}
\abs{ v_0(t, r) - v_0(t, r') } \lesssim r^{-4} \abs{ v_0(t, r)}^3 + r^{-1} \abs{ v_1(t, r)}^3
} 
and
\EQ{\label{v1 diff1}
\abs{ v_1(t, r) - v_1(t, r') } \lesssim r^{-5} \abs{ v_0(t, r)}^3 + r^{-2} \abs{ v_1(t, r)}^3
} 
with the above estimates holding uniformly in $t \in \R$. 
\end{cor}
We will also use a rewording of Corollary~\ref{diff esti} which is an immediate consequence of~\eqref{v u project} and the definitions of $\de_1$ and $R_1 = R_1( \de_1)$ in~\eqref{R1 de1}.
\begin{cor} \label{diff esti 2}
Let $R_1, \de_1$ be as in~\eqref{R1 de1}. Then for all $r, r'$ with $R_1 \le r \le r' \le 2r$ and for all $t \in \R$ we have 
\EQ{\label{v0 diff2}
\abs{ v_0(t, r) - v_0(t, r') } \lesssim r^{-1} \de_1 \abs{ v_0(t, r)}+ \de_1 \abs{ v_1(t, r)}
} 
and
\EQ{\label{v1 diff2}
\abs{ v_1(t, r) - v_1(t, r') } \lesssim r^{-2} \de_1 \abs{ v_0(t, r)} + r^{-1} \de_1 \abs{ v_1(t, r)}
} 
with the above estimates holding uniformly in $t \in \R$. 
\end{cor}

\begin{proof}[Proof of Corollary~\ref{diff esti}]
This follows rather directly from Lemma~\ref{v project}. For $r \ge R_1$ and $r \le r' \le 2r$ we see that 
\ant{
\abs{ v_0(t, r) - v_0(t, r') } &\le \left( \int_r^{r'} \abs{\p_r v_0(t,  \rho) } \, d \rho\right)\\
& \le  \left( \int_r^{r'} \abs{\frac{1}{\rho}\p_r v_0(t,  \rho) }^2 \, d \rho\right)^{\frac{1}{2}} \left( \int_r^{r'} \rho^2 \, d \rho\right)^{\frac{1}{2}}\\
&\lesssim r^{\frac{3}{2}} \left[ r^{-1} \left(3r^{-3} v_0^2(t, r) + r^{-1}v_1^2(t, r)\right)^{\frac{3}{2}}\right]\\
&\lesssim r^{-4} \abs{ v_0(t, r)}^3 + r^{-1} \abs{ v_1(t, r)}^3
}
where in the second to last inequality above we used Lemma~\ref{v project}. Similarly, 
\ant{
\abs{ v_1(t, r) - v_1(t, r') } &\le \left( \int_r^{r'} \abs{\p_r v_1(t,  \rho) } \, d \rho\right)\\
& \le  \left( \int_r^{r'} \abs{\p_r v_0(t,  \rho) }^2 \, d \rho\right)^{\frac{1}{2}} \left( \int_r^{r'}  \, d \rho\right)^{\frac{1}{2}}\\
&\lesssim r^{\frac{1}{2}} \left[ r^{-1} \left(3r^{-3} v_0^2(t, r) + r^{-1}v_1^2(t, r)\right)^{\frac{3}{2}}\right]\\
&\lesssim r^{-5} \abs{ v_0(t, r)}^3 + r^{-2} \abs{ v_1(t, r)}^3
}
as desired.
\end{proof}
Next, we can use the difference estimates to provide an upper bound on the growth rates of $v_0(t, r)$ and $v_1(t, r)$ as $r \to \infty$. 
\begin{claim}
Let $v_0(t, r)$ and $v_1(t, r)$ be as in~\eqref{v def}. Then 
\begin{align}
\label{v0 grow} &\abs{v_0(t, r)} \lesssim r^{\frac{1}{6}} \\
\label{v1 grow} &\abs{v_1(t, r)} \lesssim r^{\frac{1}{6}}
\end{align}
where again the constants above are uniform in $t \in \R$. 
\end{claim}
\begin{proof}
We begin by noting that it suffices to prove the claim for $t=0$ since the argument relies solely on estimates in this section that hold uniformly in $t \in\R$. 

Fix $r_0 \ge R_1$ and observe that by setting $r= 2^nr_0$, $r'=2^{n+1}r_0$ in the difference estimates~\eqref{v0 diff2},~\eqref{v1 diff2} we have for each $n \in \N$, 
\begin{align}
&\label{v0 tri}\abs{v_0(2^{n+1} r_0)} \le (1+ C_1(2^n r_0)^{-1} \de_1) \abs{v_0(2^nr_0)} + C_1 \de_1 \abs{v_1(2^nr_0)}\\
&\label{v1 tri}\abs{v_1(2^{n+1} r_0)} \le (1+ C_1(2^n r_0)^{-1} \de_1) \abs{v_1(2^nr_0)} + C_1 \de_1(2^n r_0)^{-2} \abs{v_0(2^nr_0)}
\end{align}
For clarity of the exposition, we introduce the notation 
\begin{align}
a_n:= \abs{v_0(2^nr_0)}\\
b_n:=\abs{v_1(2^nr_0)}
\end{align}
Adding~\eqref{v0 tri} with \eqref{v1 tri} yields
\ant{
a_{n+1} + b_{n+1} &\le (1+ C_1\de_1((2^n r_0)^{-1} +(2^n r_0)^{-2}))a_n + (1+C_1 \de_1(1+(2^n r_0)^{-1}))b_n\\
&\le (1+ 2C_1 \de_1)(a_n + b_n)
}
Arguing inductively we see that for each $n$, 
\ant{
(a_n+b_n) \le (1+2C_1 \de_1)^{n}(a_0+b_0)
}
we now choose $\de_1$ small enough so that $(1+2C_1 \de_1) \le 2^{\frac{1}{6}}$. This allows us to conclude that 
\EQ{\label{an bn}
a_n \le C(2^n r_0)^{\frac{1}{6}}\\
b_n \le C(2^n r_0)^{\frac{1}{6}}
}
where the constant $C=C(r_0)$ but this does not matter for our purposes since $r_0$ is fixed. Note that~\eqref{an bn} proves~\eqref{v0 grow} and~\eqref{v1 grow} for $r=2^nr_0$. The general estimates ~\eqref{v0 grow} and~\eqref{v1 grow} now follow easily by combining~\eqref{an bn} with the difference estimates~\eqref{v0 diff1},~\eqref{v1 diff1}. 
\end{proof}

With the bounds on the growth of $(v_0, v_1)$ proved in the previous claim, we can now extract a limit. We begin with $v_1(t, r)$ as we will need to first show that this tends to $0$ in  order to get the correct rate for $v_0(t, r)$. We will show this in several steps. We begin with the following. 
\begin{claim} \label{ell1 claim}For each $t \in \R$ there exists $\ell_1(t) \in \R$ so that 
\EQ{
\abs{ v_1(t, r) -  \ell_1(t)}  = O(r^{-2}) \mas r \to \infty
}
where the $O( \cdot)$ above is uniform in $t \in \R$. 
\end{claim}
\begin{proof}
As usual, it suffices to show this for $t=0$. Let $r_0 \ge R_1$ where $R_1 $ is as in~\eqref{R1 de1}. Plugging~\eqref{v0 grow},~\eqref{v1 grow} into the difference estimates~\eqref{v1 diff1} gives 
\EQ{\label{v1 conv}
\abs{ v_1(2^{n+1}r_0) - v_1(2^n r_0) } &\lesssim (2^nr_0)^{-5} (2^nr_0)^{\frac{1}{2}} + (2^nr_0)^{-2} (2^nr_0)^{\frac{1}{2}}\\
& \lesssim  (2^nr_0)^{-\frac{3}{2}}
} 
Therefore the series 
\ant{
 \sum_n \abs{ v_1(2^{n+1}r_0) - v_1(2^n r_0) } < \infty,
 }
 which in turn implies the existence of an $\ell_1 \in \R$ so that 
 \ant{
  \lim_{n \to \infty} v_1(2^nr_0) = \ell_1.
  }
Using again the difference estimates~\eqref{v0 diff1}, ~\eqref{v1 diff1} as well as the growth estimates allows us to conclude that in fact, 
 \ant{
  \lim_{r \to \infty} v_1(r) = \ell_1.
  }
To obtain the rate of convergence, we note that the above also implies that $\abs{v_1(r)}$ is bounded, and hence  the same logic that provided~\eqref{v1 conv} can be upgraded to  give 
\ant{
\abs{ v_1(2^{n+1}r) - v_1(2^n r) } &\lesssim(2^nr)^{-2}
}
for every $r$ large enough.
Hence, 
\ant{
\abs{v_1(r) - \ell_1} = \abs{ \sum_{n \ge 0}(v_1(2^{n+1} r) - v_1(2^n r)) } \lesssim r^{-2} \sum_{n \ge 0} 2^{-2n} \lesssim r^{-2}
}
which finishes the proof of the claim.
\end{proof}
Next, we prove that the limit $\ell_1(t)$ is, in fact, independent of $t$. 
\begin{claim}\label{ell1 constant} The function $\ell_1(t)$ in Claim~\ref{ell1 claim} is independent of $t$, i.e., $\ell_1(t) =  \ell_1$ for all $t \in \R$.

\end{claim}

\begin{proof} Recall that, by definition 
\ant{
v_1(t, r) = r\int_r^{\I} u_t(t, \rho) \, \rho \, d\rho
}
By Claim~\eqref{ell1 claim}, we then have 
\ant{
\ell_1(t) =  r\int_r^{\I} u_t(t, \rho) \, \rho \, d\rho + O(r^{-2}) \mas r \to \infty
}
Now, let $t_1, t_2 \in \R$ with  $t_1 \neq t_2$. We will show that 
\ant{
\ell_1(t_2)- \ell_1(t_1) =0
}
To see this, we begin by noting that averaging in $R \ge R_1$ gives 
\ant{
\ell_1(t_2) - \ell_1(t_1) 
&= \frac{1}{R} \int_R^{2R}\left(s \int_s^{\infty} (u_t(t_2, r) - u_t(t_1, r)) r \, dr \right) \, ds + O(R^{-2})\\
&=\frac{1}{R} \int_R^{2R}\left(s \int_s^{\infty}  \int_{t_1}^{t_2}u_{tt}(t, r) \, dt\,  r \, dr \right) \, ds + O(R^{-2})
}
Using the fact that $\vec u(t)$ is a solution to \eqref{u eq}, and writing the nonlinearity in~\eqref{u eq} as 
\ant{ 
\NN(r, u) = - Z_1(ru)u^3 - Z_2(ru)u^5
}
 we can rewrite the above integral as
\ant{
&=  \int_{t_1}^{t_2}\frac{1}{R} \int_R^{2R}\left(s \int_s^{\infty}  (ru_{rr}(t, r) + 4u_r(t, r) ) \, dr \right) \, ds\, dt + \\
&\, +  \int_{t_1}^{t_2} \frac{1}{R} \int_R^{2R}\left(s \int_s^{\infty}\NN(r, u(t,r))   \, dr \right) \,ds\, dt  + O(R^{-2})\\
& = A + B +  O(R^{-2})
}
To estimate $A$ we integrate by parts twice: 
\EQ{ \label{I}
A&= \int_{t_1}^{t_2}\frac{1}{R} \int_R^{2R}\left(s \int_s^{\infty}  \frac{1}{r^3} \p_r(r^4 u_r(t, r)) \, dr \right) \, ds\, dt\\
& = 3\int_{t_1}^{t_2}\frac{1}{R} \int_R^{2R}\left(s \int_s^{\infty}  u_r(t, r) \, dr \right) \, ds\, dt -\int_{t_1}^{t_2}\frac{1}{R} \int_R^{2R}s^2  u_r(t, s) \, ds\, dt \\
& = -3\int_{t_1}^{t_2}\frac{1}{R} \int_R^{2R}r \, u(t, r) \, dr\, dt -\int_{t_1}^{t_2}\frac{1}{R} \int_R^{2R}r^2\, u_r(t, r) \, dr\, dt \\
&= -\int_{t_1}^{t_2}\frac{1}{R} \int_R^{2R}r \, u(t, r) \, dr\, dt + \int_{t_1}^{t_2} (Ru(t,R)-2Ru(t, 2R))\, dt
}
To control the last line above we recall that by the definition of $v_0(t, r)$ and ~\eqref{v0 grow} we have 
\EQ{\label{u control}
r^3\abs{u(t, r)} := \abs{v_0(t, r))} \lesssim r^{\frac{1}{6}}
}
uniformly in $t \in \R$, and hence $\abs{u(t, r)} \lesssim r^{-\frac{17}{6}}$ uniformly in $t \in\R$. Plugging this into the last line in the estimate for $A$ yields, 
\ant{
A = \abs{t_2-t_1}O(R^{-\frac{11}{6}})
}
To estimate $B$ we use~\eqref{u control} and the rough estimates for the size of $\abs{\NN(u)}$ implied by~\eqref{Z bounds} to see that for $r>R_1$ we have 
\ant{
\abs{\NN(r, u(t, r))} \lesssim \abs{u(t, r)}^3 + \abs{u(t, r)}^5 \lesssim r^{-\frac{17}{2}}  \lesssim r^{-8}
}
Therefore, 
\ant{
B \lesssim \int_{t_1}^{t_2} \frac{1}{R} \int_R^{2R}\left(s \int_s^{\infty}r^{-8}   \, dr \right) \,ds\, dt = \abs{t_2- t_1} O(r^{-6})
}
Putting this all together gives
\ant{
\abs{\ell_1(t_2)- \ell_1(t_1)}  = O(R^{-\frac{11}{6}}) \mas R \to \infty
}
which implies that $\ell_1(t_1) = \ell_1(t_2)$ as desired.
\end{proof}
The next step is to show that $\ell_1 = 0$. 
\begin{claim}\label{ell1=0}  $\ell_1 =0$. 
\end{claim} 
\begin{proof} Suppose $\ell_1 \neq 0$. 
We know that for all $R \ge R_1$  and for all $t \in \R$ we have 
\ant{
R\int_R^{\infty}u_t(t, r) \, r \, dr = \ell_1 + O(R^{-2}),
}
where $O(\cdot)$ is uniform in~$t$. Therefore, by taking $R$ large, the left-hand side above has the same sign as $\ell_1$, for all $t$.
Thus we can choose $R \ge R_1$ large enough so that for all $t \in \R$, 
\ant{
\abs{ R\int_{R}^{\infty}u_t(t, r) \, r \, dr }\ge \frac{ |\ell_1|}{2}.
}
Integrating from $t=0$ to $t=T$ gives
\ant{
\Big| \int_0^TR\int_{R}^{\infty}u_t(t, r) \, r \, dr\, dt\Big|  \ge T\frac{|\ell_1|}{2} .
}
However, we integrate in $t$ on the left-hand side and use \eqref{u control} to obtain
\ant{
\abs{R\int_{R}^{\infty}\int_0^Tu_t(t, r) \, r \, dt \, dr }&= \abs{ R\int_R^{\infty}(u(T, r) - u(0,r) )\, r \, dr}\\
& \lesssim R\int_R^{\infty} r^{-\frac{11}{6}}\, dr
\lesssim R^{\frac{1}{6}}.
}
Therefore for fixed large $R$  we have 
\ant{
T\frac{|\ell_1|}{2}  \lesssim R^{\frac{1}{6}},
}
which gives a contradiction by taking $T$ large enough. 
\end{proof}

We can now finish the proof of Lemma~\ref{spacial decay}. 
\begin{proof}[Proof of Lemma~\ref{spacial decay}]
We note that combining the results of Claims~\ref{ell1 claim},~\ref{ell1 constant},~\ref{ell1=0}, we have established~\eqref{u1 decay} and~\eqref{u1 rate}, namely that 
\EQ{\label{v1 to 0}
\abs{v_1(r)} = O(r^{-2}) \mas r \to 0.
}
 It therefore remains to show that there exists $\ell_0 \in \R$ so that 
\EQ{
\abs{v_0(r) - \ell_0} = O(r^{-4}) \mas r \to \infty.
}
To prove this, we plug~\eqref{v1 to 0} along with the growth estimates~\eqref{v0 grow} into the difference estimate~\eqref{v0 diff1}. We see that for fixed $r_0 \ge R_1$ and all $n \in \N$ we have 
\ant{
\abs{v_0(2^{n+1}r_0) - v_0(2^nr_0)} &\lesssim (2^nr_0)^{-4} (2^nr_0)^{\frac{1}{2}} + (2^nr_0)^{-1} (2^nr_0)^{-6}\\
&\lesssim (2^nr_0)^{-\frac{7}{2}}.
}
Therefore the series 
\ant{
 \sum_n \abs{ v_0(2^{n+1}r_0) - v_0(2^n r_0) } < \infty,
 }
 which in turn implies the existence of an $\ell_0 \in \R$ so that 
 \ant{
  \lim_{n \to \infty} v_0(2^nr_0) = \ell_0.
  }
Using again the difference estimates~\eqref{v0 diff1} and the fact that $\abs{v_1(r)} \to 0$ allows us to conclude that in fact, 
 \ant{
  \lim_{r \to \infty} v_0(r) = \ell_0.
  }
To obtain the rate of convergence, we note that the above also implies that $\abs{v_0(r)}$ is bounded, and hence the difference estimate can be upgraded to  
\ant{
\abs{ v_0(2^{n+1}r) - v_0(2^n r) } &\lesssim(2^nr)^{-4}
}
which holds for every $r \ge R_1$.
Hence, 
\ant{
\abs{v_0(r) - \ell_0} = \abs{ \sum_{n \ge 0}(v_1(2^{n+1} r) - v_1(2^n r)) } \lesssim r^{-4} \sum_{n \ge 0} 2^{-4n} \lesssim r^{-2},
}
which finishes the proof.
\end{proof}

\subsection{Step $3$} We complete the proof of Proposition~\ref{rigid} by showing that $\vec u(t, r)  \equiv (0, 0)$. We separate the argument into two cases depending on whether the number $\ell_0$ in Lemma~\ref{spacial decay} is zero or nonzero.

\begin{flushleft} \textbf{Case 1: $\ell_0=0$ implies $\vec u(t)\equiv(0, 0)$:} 
\end{flushleft}

We formulate this step as a lemma: 
\begin{lem} \label{ell=0 lem}Let $\vec u(t)$ be as in Proposition~\ref{rigid} and let $\ell_0\in \R$ be as in Lemma~\ref{spacial decay}.  If $\ell_0 = 0$ then $\vec u(t)  \equiv (0, 0)$. 
\end{lem}
To prove the lemma, we will first establish the following claim, which says that if $\ell_0 = 0$ then $(u_0, u_1)$ must be compactly supported. We then show that the only solution with pre-compact trajectory and compactly supported initial data is necessarily $\vec u(t) = (0, 0)$. 
\begin{claim}\label{comp supp}Let $\ell_0$ be as in Lemma~\ref{spacial decay}. If $\ell_0=0$ then $(u_0, u_1)$ is compactly supported. 
\end{claim}
\begin{proof}
We being by noting that the assumption $\ell_0 = 0$ implies that for $r \ge R_1$, 
\EQ{ \label{ell0 means}
\abs{v_0(r)}  = O (r^{-4}) \mas r \to \infty\\
\abs{v_1(r)}  = O( r^{-2}) \mas r \to \infty
}
This means that for $r_0 \ge R_1$ we have 
\EQ{\label{low bound}
\abs{v_0(2^nr_0)} + \abs{v_1(2^nr_0)} \lesssim (2^nr_0)^{-4} +(2^nr_0)^{-2}
}
On the other hand, the difference estimates~\eqref{v0 diff1} and~\eqref{v1 diff1} together with~\eqref{ell0 means} give
\ant{
&\abs{v_0(2^{n+1} r_0)} \ge (1-C_1(2^nr_0)^{-12}) \abs{v_0(2^nr_0)} - C_1(2^nr_0)^{-5} \abs{v_1(2^nr_0)}\\
&\abs{v_1(2^{n+1} r_0)} \ge (1-C_1(2^nr_0)^{-6}) \abs{v_0(2^nr_0)} - C_1(2^nr_0)^{-13} \abs{v_1(2^nr_0)}
}
For large $r_0$ we can combine the above estimates to obtain 
\ant{
\abs{v_0(2^{n+1} r_0)} + \abs{v_1(2^{n+1} r_0)} \ge (1- 2C_1r_0^{-5})( \abs{v_0(2^{n} r_0)} + \abs{v_1(2^{n} r_0)})
}
Fixing $r_0$ large enough so that $2C_1r_0^{-5} < \frac{1}{4}$ and arguing inductively we conclude that 
\ant{
( \abs{v_0(2^{n} r_0)} + \abs{v_1(2^{n} r_0)}) \ge \left(\frac{3}{4}\right)^{n} ( \abs{v_0(r_0)} + \abs{v_1( r_0)})
}
Now, use~\eqref{low bound} to estimate the left-hand-side above to get 
\ant{
\left(\frac{3}{4}\right)^{n} ( \abs{v_0(r_0)} + \abs{v_1( r_0)}) \lesssim 2^{-2n} r_0^{-2}
}
which, in turn means that 
\ant{
3^n ( \abs{v_0(r_0)} + \abs{v_1( r_0)}) \lesssim 1, 
}
which is impossible unless $(v_0(r_0), v_1(r_0)) = (0, 0)$.
Hence, $$\vec v(r_0):= (v_0(r_0), v_1(r_0)) = (0, 0).$$ To turn this into a statement about $(u_0, u_1)$ we note that~\eqref{v u project} implies that 
\ant{
\| \pi_{r_0} \vec u(0)\|_{ \dot{H}^1 \times L^2(r \ge r_0)} = 0.
}
By Lemma~\ref{project ineq} we can also then deduce that 
\ant{
\| \pi_{r_0}^{\perp} \vec u(0)\|_{ \dot{H}^1 \times L^2(r \ge r_0)} = 0,
}
and hence 
\ant{
\|  \vec u(0)\|_{ \dot{H}^1 \times L^2(r \ge r_0)} = 0,
}
which concludes the proof since $\lim_{r \to \infty} u_0(r) = 0$. 
\end{proof}

\begin{proof}[Proof of Lemma~\ref{ell=0 lem}] Assume that $\ell_0 = 0$. Then, by Claim~\ref{comp supp}, $(u_0, u_1)$ is compactly supported. We assume that $(u_0, u_1) \not \equiv (0, 0)$ and argue by contradiction. In this case we define $\rho_0>0$ by 
\ant{
\rho_0:= \inf\{ \rho \, \, : \, \,  \| \vec u(0)\|_{\dot{H}^1 \times L^2(r \ge \rho)} = 0\}
}
Let $\e>0$ be a small number to be determined below. Choose $\rho_1 = \rho_1( \e)$ so that $\frac{ 1}{2} \rho_0 < \rho_1 < \rho_0$  so that 
\ant{
0< \| \vec u(0)\|_{\dot{H}^1 \times L^2(r \ge \rho_1)}^2  \le \e \le \de_1^2
}
where $\de_1$ is as in~\eqref{R1 de1}. With $(v_0, v_1)$ as in~\eqref{v def} we have 
\EQ{\label{all small}
3 \rho_1^{-3} v_0^2(\rho_1) + \rho_1^{-1}v_1^2(\rho_1)&+ \int_{\rho_1}^\I \left( \frac{1}{r} \p_r v_0( r) \right)^2 \, dr + \int_{\rho_1}^\I \left( \p_r v_1(r) \right)^2 \, dr  = \\
& = \| \pi_{\rho_1} \vec u(0)\|^2_{\dot{H}^1 \times L^2(r \ge \rho_1)}+\| \pi_{\rho_1}^{\perp} \vec u(0)\|^2_{\dot{H}^1 \times L^2(r \ge \rho_1)} \\
&= \|  \vec u(0)\|^2_{\dot{H}^1 \times L^2(r \ge \rho_1)} < \e
}
We also record the result of Lemma~\ref{v project} with $R= \rho_1$: 
\EQ{\label{v rho1}
\left( \int_{\rho_1}^\I \left[\left( \frac{1}{r} \p_r v_0(r) \right)^2 + \left( \p_r v_1( r) \right)^2 \right]\, dr\right)^{\frac{1}{2}}  \lesssim \rho_1^{-\frac{11}{2}} \abs{v_0( \rho_1)}^3 + \rho_1^{-\frac{5}{2}}\abs{v_1( \rho_1)}^3
}
Since $v_0( \rho_0) = v_1( \rho_0) = 0$ we can argue as in Corollary~\ref{diff esti} and Corollary~\ref{diff esti 2} to obtain
\ant{
&\abs{v_0( \rho_1)} = \abs{v_0(\rho_1) - v_0( \rho_0)} \le C_1\, \e\, (\abs{v_0( \rho_1)} + \abs{v_1( \rho_1)})\\
&\abs{v_1( \rho_1)} = \abs{v_1(\rho_1) - v_1( \rho_0)} \le C_2\, \e\, (\abs{v_0( \rho_1)} + \abs{v_1( \rho_1)})
}
where here we have used that $\frac{1}{2} \rho_0< \rho_1 < \rho_0$ to obtain constants $C_1, C_2$ which depend only $\rho_0$ which is fixed, and the uniform constant in~\eqref{v rho1}, but not on~$\e$. 
Putting the above estimates together yields
\EQ{
(\abs{v_0( \rho_1)} + \abs{v_1( \rho_1)}) \le C_3 \e (\abs{v_0( \rho_1)} + \abs{v_1( \rho_1)})
}
which implies that $\abs{v_0( \rho_1)} = \abs{v_1( \rho_1)} = 0$ by taking $\e>0$ small enough. By~\eqref{v rho1} and the equalities in~\eqref{all small} we can deduce that 
\ant{
\| \vec u(0)\|_{\dot{H}^1 \times L^2(r \ge \rho_1)} = 0
}
which contradicts the definition of $\rho_0$ since $\rho_1< \rho_0$.
\end{proof}
This completes the proof in the case that $\ell_0 = 0$. We now proceed to examine the case when $\ell_0 \neq0$. 
\vspace{\baselineskip}

\begin{flushleft} \textbf{Case 2: $\ell_0 \neq 0$ is impossible:} \end{flushleft}
In this final part of the rigidity argument we show that the case $\ell_0 \neq 0$ is impossible. In fact, we prove that if $\ell_0 \neq0$ then the solution to equation~\eqref{an eq} given by $\psi(t, r)  = r u(t, r)$ is equal to a nonzero stationary solution, $\fy_{\ell_0}$, of~\eqref{ode}. However, we know by Proposition~\ref{ode prop} that $(\fy_{\ell_0}, 0) \not \in \E_0$ for any $\ell_0 \neq 0$, which is a contradiction since by construction our critical element $\vec u(t) \in \HH$ and hence $\vec \psi(t) \in \E_0$ by~\eqref{finite energy in HH}.

The main idea will be to linearize about the elliptic solution $\fy_{\ell_0}$ given by Proposition~\ref{ode prop} with the same spacial asymptotic behavior as our critical data $\vec \psi(0) = r \vec u(0)$.  In the previous steps we have shown that 
\ant{
r^3u_0(r) = \ell_0 + O(r^{-4}) \mas r \to \infty.
}
Since $ \psi_0(r):=ru_0(r)$ we then have 
\ant{
\psi_0(r) = \ell_0r^{-2} + O(r^{-6}) \mas r \to \infty.
}
Now, let $\fy_{\ell_0}(r)$ be given as in Proposition~\ref{ode prop}, which means that $\fy_{\ell_0}$ solves the elliptic equation~\eqref{ode} and satisfies
\ant{
\fy_{\ell_0}(r)= \ell_0r^{-2} + O(r^{-6}) \mas r \to \infty.
}
We then define $ \vec u_{\ell_0}(0) = ( u_{\ell_0, 0}, u_{\ell_0, 1})$ by 
\EQ{\label{ul def}
&u_{\ell_0, 0}(r) := \frac{1}{r}( \psi_0(r) - \fy_{\ell_0}(r))\\
&u_{\ell_0, 1}(r):= \frac{1}{r} \psi_1(r)
}
For each $t \in \R$ we also define 
\EQ{\label{ul t def}
u_{\ell_0}(t, r):= \frac{1}{r}( \psi(t, r) - \fy_{\ell_0}(r))
}
where of course $\psi(t, r) = ru(t, r)$. We can record several properties of $\vec u_{\ell_0} = ( u_{\ell_0}, \p_t u_{\ell_0})$. By construction we have 
\EQ{\label{ul limits}
&v_{\ell_0, 0}(r):= r^3 u_{\ell_0, 0}(r) = O(r^{-4}) \mas r \to \infty\\
&v_{\ell_0, 1}(r):= r \int_r^{\infty} u_{\ell_0, 1}(\rho) \rho, d\rho = O(r^{-2}) \mas r \to \infty
}
Also, using that $\vec \psi(t)$ solves~\eqref{an eq} and $\fy_{\ell_0}$ solves \eqref{ode} we can deduce that $\vec u_{\ell_0}(t)$  satisfies the equation 
\EQ{\label{ul eq}
\p_{tt}u_{\ell_0} - \p_{rr} u_{\ell_0} - \frac{4}{r} \p_r u_{\ell_0} &= -V_{1}(r, \fy_{\ell_0}) u - V_{2}(r, \fy_{\ell_0}) u \\
& \quad + \NN_{1}(r, \fy_{\ell_0}, u_{\ell_0}) +\NN_{2}(r, \fy_{\ell_0},u_{\ell_0})
}
where we have the estimates
\EQ{\label{nonlin size}
\abs{V_{1}(r, \fy_{\ell_0})}& \lesssim r^{-2} \abs{\fy_{\ell_0}}^2 \\
 \abs{V_{2}(r, \fy_{\ell_0})} &\lesssim r^{-4} \abs{\fy_{\ell_0}}^4 \\
 \abs{\NN_{1}(r, \fy_{\ell_0}, u_{\ell_0}) }& \lesssim r^{-3} \abs{\fy_{\ell_0}}\abs{ru_{\ell_0}}^2 +  \abs{u_{\ell_0}}^3\\
 \abs{\NN_{2}(r, \fy_{\ell_0},u_{\ell_0})} &\lesssim r^{-5} \abs{\fy_{\ell_0}}^3\abs{ru_{\ell_0}}^2 +r^{-5} \abs{\fy_{\ell_0}}^2\abs{ru_{\ell_0}}^3\\
 &\quad+r^{-5} \abs{\fy_{\ell_0}}\abs{ru_{\ell_0}}^4 + \abs{u_{\ell_0}}^5
 }
 The form of the right-hand-side of~\eqref{ul eq} and the above estimates can be proved by reducing the terms
 \ant{
 &\frac{\sin2(\fy_{\ell_0} +ru_{\ell_0}) -2(\fy_{\ell_0} +ru_{\ell_0}) - \sin2\fy_{\ell_0}+ 2\fy_{\ell_0}}{r^3}\\
 &\frac{ f(\fy_{\ell_0} +ru_{\ell_0}) - f(\fy_{\ell_0})}{r^5}
 }
 where $f(x):=(x- \cos x\sin x)(1-\cos(2x)$, using trigonometric identities. 
 
 The crucial point is that by construction $\vec u_{\ell_0}$ inherits the compactness property~\eqref{compact ext en} from $\vec \psi(t)$ since $\fy_{\ell_0}$ is stationary. By this we mean that the trajectory,
 \ant{
 \ti K:= \{ \vec u_{\ell_0}(t) \mid t \in \R\}
 }
 is pre-compact in $\dot{H}^1 \times L^2(\R^5)$ and hence for each $R > 0$ we have 
 \EQ{
 \|\vec u_{\ell_0}(t)\|_{\dot{H}^1 \times L^2(r \ge R+ \abs{t})} \to 0 \mas \abs{t} \to \infty
 }
With this set-up we are in position to prove, much as in the $\ell_0 = 0$ case that we must have $\vec u_{\ell_0} \equiv (0, 0)$. In particular, we prove the following lemma. 
\begin{lem}\label{ul 0}
Suppose $\ell_0 \neq 0$. Define  $\vec u_{\ell_0}$ as in~\eqref{ul def}, \eqref{ul t def}. Then, $\vec u_{\ell_0} \equiv (0, 0)$, i.e., 
$\vec \psi(0) = ( \fy_{\ell_0}, 0)$ where $\fy_{\ell_0}$ is given by Proposition~\ref{ode prop} and therefore $\vec \psi(0) \not \in \E_0$. 
\end{lem}
The argument that we will use to prove Lemma~\ref{ul 0} is very similar to the argument we presented in the previous steps which finished in the case $\ell_0 = 0$. Due to similar nature of the proof we will omit many details here. 

Although we already have established the asymptotic behavior of $\vec u_{\ell_0}$ given in~\eqref{ul limits}, we recall that the driving force behind the entire argument in the previous steps was the estimate~\eqref{key ineq}, which gave a quantitative restriction on how close $\vec u(0)$ had to be to the plane $P(R)$. We will prove a similar estimate here for $\vec u_{\ell_0}$ and for this we will need to modify the Cauchy problem~\eqref{ul eq} on the interior of the cone $\{r \le R + \abs{t}\}$ for large $R$. Similar to the argument in~\cite{KLS} we alter the potential terms and nonlinearity in~\eqref{ul eq}. In particular, for each $R>0$ we define $\fy_{\ell_0, R}$ by
\EQ{ 
\fy_{\ell_0,R}(t, r):= \begin{cases} \fy_{\ell_0}(R+ \abs{t}) \mfor 0 \le r \le R + \abs{t}\\ \fy_{\ell_0}(r) \mfor r \ge R+ \abs{t}\end{cases}
}
Next, we set  
\begin{align*} %\label{VR def}
&V_{1,R}(t, r):= \begin{cases}  V_{1}(R+\abs{t},\fy_{\ell_0, R})\mfor 0 \le r \le R+ \abs{t} \\  V_{1}(r, \fy_{ \ell_0, R}) \mfor r \ge R+ \abs{t} \end{cases}\\ 
&V_{2,R}(t, r):= \begin{cases}  V_{2}(R+\abs{t}, \fy_{\ell_0, R})\mfor 0 \le r \le R+ \abs{t} \\  V_{2}(r, \fy_{ \ell_0, R}) \mfor r \ge R+ \abs{t} \end{cases}
\end{align*}
as well as 
\begin{align*}
&\NN_{1,R}(t, r, w):= \begin{cases}  \NN_{1}(R+\abs{t}, \fy_{\ell_0, R}, w)\mfor 0 \le r \le R+ \abs{t} \\  \NN_{1}(r, \fy_{\ell_0, R}, w) \mfor r \ge R+ \abs{t} \end{cases}\\ 
&\NN_{2,R}(t, r, w):= \begin{cases}  \NN_{2}(R+ \abs{t}, \fy_{\ell_0, R}, w)\mfor 0 \le r \le R+ \abs{t} \\  \NN_{2}(r, \fy_{\ell_0, R}, w) \mfor r \ge R+ \abs{t} \end{cases}\\ 
\end{align*}
Note that for $R$ large enough we have, using \eqref{fy at inf} and \eqref{nonlin size}, that 
\begin{align} \label{V1R}
&\abs{V_{1, R}(t, r)} \lesssim \begin{cases} (R+\abs{t})^{-6} \mfor 0 \le r \le R + \abs{t}\\ r^{-6} \mfor r \ge R + \abs{t}\end{cases}\\
&\abs{V_{2, R}(t, r)} \lesssim \begin{cases} (R+\abs{t})^{-12} \mfor 0 \le r \le R + \abs{t}\\ r^{-12} \mfor r \ge R + \abs{t}\end{cases} \label{V2R}
\\ \label{N1R}
&\abs{\NN_{1, R}(t, r, w)} \lesssim \begin{cases} (R+\abs{t})^{-3}\abs{ w}^2 + \abs{w}^3\mfor 0 \le r \le R + \abs{t}\\ r^{-3}\abs{w}^2 + \abs{w}^3 \mfor r \ge R + \abs{t}\end{cases}
\\ \label{N2R}
&\abs{\NN_{2, R}(t, r, w)} \lesssim \begin{cases} (R+\abs{t})^{-9}\abs{ w}^2 + (R+\abs{t})^{-6}\abs{ w}^3\\+ (R+\abs{t})^{-3}\abs{ w}^4 + \abs{ w}^5 \mfor 0 \le r \le R + \abs{t}\\ r^{-9}\abs{w}^2 +r^{-6}\abs{w}^3+ r^{-3}\abs{w}^4+  \abs{w}^5 \\\mfor r \ge R + \abs{t}\end{cases}%\label{NN2R decay}
\end{align}
We consider a modified Cauchy problem in $\R^{1+5}$. As in the set-up for Lemma~\ref{h small data} we fix a smooth function $\chi \in C^{\infty}([0, \infty))$ where $\chi(r) = 1$ for $r \ge 1$ and $\chi(r) = 0$ on $r \le 1/2$. Then set $\chi_R(r):= \chi(r/R)$ and for each $R>0$ we consider the modified Cauchy problem: 
\begin{align} \label{w eq mod}
&w_{tt}- w_{rr} - \frac{4}{r} w_r =\M_R(t, r,w) \\ \notag
&\M_R(t, r, w) := - \chi_RV_{1,R}(t, r) w - \chi_RV_{2,R}(t, r) w + \chi_R\NN_{1, R}(t, r, w)+ \chi_R\NN_2(t, r, w)\\
& \vec w(0) = (w_0, w_1) \in \dot{H}^1 \times L^2(\R^5) \notag
\end{align}
The point here is that we have introduced additional decay and in particular, time integrability into the potential terms which will allow these to be treated perturbatively in the small data theory given in the following lemma. We have also introduced the cut-off $\chi_R$ which removes the supercritical nature of the nonlinearity. This will allow us to treat the right-hand-side perturbatively in the energy space. This is an analog of Lemma~\ref{h small data} where here we have linearized about a nontrivial elliptic solution $\fy_{\ell_0}$. First we recall the definition of the norm $Z(I)$ from~\eqref{Z norm def}: 
\ant{
\|\vec w\|_{Z(I)} = \|w\|_{L^{\frac{7}{3}}_t(I; L^{\frac{14}{3}}_x(\R^5))} + \|\vec w(t) \|_{L^{\infty}_t(I; \dot H^1 \times L^2)}
}
\begin{lem}\label{mod cp} There exists $R_2>0$ and there exists $\de_2>0$ small enough so that for all $R>R_2$ and all  initial data $\vec w(0) = (w_0, w_1) \in \dot{H}^1 \times L^2(\R^5)$ with 
\ant{
\|\vec w(0) \|_{\dot{H}^1 \times L^2(\R^5)}   < \de_2
} 
there exists a unique global solution $\vec w(t) \in \dot{H}^1 \times L^2$ to \eqref{w eq mod}. In addition $\vec w(t) $ satisfies 
\EQ{ \label{mod a priori}
\|w\|_{Z(\R))} \lesssim \|\vec w(0)\|_{\dot{H}^1 \times L^2(\R^5)} \lesssim \de_2
}
Moreover, if we denote the free evolution with the same data by $h_L(t):=S(t) \vec h(0)\in \dot{H}^1 \times L^2(\R^5)$, then we have 
\EQ{\label{w lin nl mod}
\sup_{t \in \R}\|\vec w(t) - \vec w_L(t)\|_{\dot{H}^1 \times L^2}& \lesssim R^{-4}\|\vec w(0)\|_{\dot{H}^1 \times L^2} + R^{-5/2}\|\vec w(0)\|_{\dot{H}^1 \times L^2}^2 \\
& \quad+R^{-1} \|\vec w(0)\|_{\dot{H}^1 \times L^2}^3 + R^{-11/2}\|\vec w(0)\|_{\dot{H}^1 \times L^2}^4 \\
&\quad+ R^{-4} \|\vec w(0)\|_{\dot{H}^1 \times L^2}^5
}
\end{lem}
\begin{proof} The proof is very similar to the proof of Lemma~\ref{h small data} and we give a sketch here. The small data well-posedness theory, including the estimate~\eqref{mod a priori} follows from the usual contraction mapping and continuity arguments based on Strichartz estimates. To prove~\eqref{w lin nl mod} we note that by Duhamel and Strichartz we have 
\ant{
\|\vec w(t) -  \vec w_L(t) \|_{\dot{H}^1 \times L^2} \lesssim \|\M_R( \cdot, \cdot, w)\|_{L^1_tL^2_x(\R^{1+5})}
}
Hence it suffices to control the right-hand-side above by the right-hand-side of~\eqref{w lin nl mod}. %For large enough $R$ we can combine the estimates in~\eqref{M decay} by grouping the terms with the same power of $\abs{w}$ to obtain the point-wise estimates 
%\ant{
%\abs{\M_R(t, r, w)} \lesssim   \begin{cases} (R+\abs{t})^{-6}\abs{ \chi_R w} + (R+\abs{t})^{-3}\abs{ \chi_R w}^2+ \abs{ \chi_R w}^3 \\+ (R+\abs{t})^{-3}\abs{ \chi_R w}^4 + \abs{ w}^5\mfor 0 \le r \le R + \abs{t}\\ r^{-6}\abs{ \chi_R w} +r^{-3}\abs{ \chi_Rw}^2+ \abs{ \chi_R w}^2+  r^{-3}\abs{ \chi_R w}^4+  \abs{w}^5 \\
%\mfor r \ge R + \abs{t}\end{cases}
%}
For the linear-in-$w$ terms we have 
\ant{
\| (V_{1,R} + V_{2,R}) \chi_R w\|_{L^{1}_tL^2_x} \lesssim \|(V_{1,R} + V_{2,R})\|_{L^{\frac{7}{4}}L^{\frac{7}{2}}} \|  \chi_R w\|_{L^{\frac{7}{3}}L^{\frac{14}{3}}}
}
Next, combining~\eqref{V1R} and  \eqref{V2R} we can estimate
\ant{
\| (V_{1,R} + V_{2,R}) \chi_R w\|_{L^{\frac{7}{2}}_x}  &\lesssim  \left( \int_0^{R+ \abs{t}}(R+\abs{t})^{-21} \, r^4\, dr  + \int_{R+ \abs{t}}^{\infty} r^{-21} \, r^4\, dr \right)^{\frac{2}{7}}\\
& \lesssim (R+ \abs{t})^{-\frac{32}{7}}
}
Hence, 
\ant{
\|(V_{1,R} + V_{2,R})\|_{L^{\frac{7}{4}}L^{\frac{7}{2}}} &\lesssim \left( \int_{\R}(R+ \abs{t})^{-8} \, dt \right)^{\frac{4}{7}} \lesssim R^{-4}
}
The quadratic-in-$w$ terms are handled similarly. Finally, for the cubic, quartic, and quintic terms in the nonlinearity, one also has to use the Strauss estimate~\eqref{Strauss} together with the point wise estimates~\eqref{N1R} and~\eqref{N2R} and argue as the proof of Lemma~\ref{h small data}. For the cubic and quintic terms the proof is identical to the argument in Lemma~\ref{h small data} and we see the identical $R^{-1}$ and $R^{-4}$ decay in~\eqref{w lin nl mod} as we obtained~\eqref{h-hL}.  To estimate the quartic term in the nonlinearity we use~\eqref{N2R} to argue as follows: 
\ant{
\left\|((R+\abs{t})^{-3}1_{\{r \le R+\abs{t}\}} + r^{-3}1_{\{r \ge R+\abs{t}\}}) \chi_R \abs{w}^4\right\|_{L^1_tL^2_x} & \lesssim R^{-3} \|\chi_R \abs{w}^4\|_{L^1_tL^2_x}.
}
Now, we again argue as in the proof of Lemma~\ref{h small data}, using the Strauss estimate,~\eqref{Strauss},  to obtain
\ant{
R^{-3}\| \chi_R \abs{w}^4\|_{L^2}& \lesssim \sup_{r \ge R}\abs{w(t, r)}^{\frac{5}{3}} \|w(t)\|_{L^{\frac{14}{3}}}^{\frac{7}{3}} \\
& \lesssim R^{-11/2} \|w(t)\|_{\dot{H}^1}^{\frac{5}{3}} \|w(t)\|_{L^{\frac{14}{3}}}^{\frac{7}{3}}.
}
Integrating in time we can then conclude 
\ant{
\|R^{-3}\| \chi_R \abs{w}^4\|_{L^2_x}\|_{L^1_t} \lesssim R^{-11/2} \|w\|_{Z(\R)}^4 \lesssim R^{-11/2} \|\vec w(0)\|_{\dot{H}^1\times L^2}^4,
}
where the last estimate follows from~\eqref{mod a priori}. 
%One argues similarly to handle the quartic term but we gain additional decay from the $(R+ \abs{t})^{-3}$ on $\{r \le R+\abs{t}\}$ and $r^{-3}$ on $\{r \ge R + \abs{t}\}$ decay that  appears with the quartic terms in~\eqref{N2R}. 
\end{proof}
With Lemma~\ref{mod cp} we can now argue exactly as in the proof of Lemma~\ref{project ineq} to establish: 
\begin{lem} \label{project ineq mod}There exists $R_2>0$ so that for all $R>R_2$ and for all $t \in \R$ we have 
\ant{ 
\| \pi^{\perp}_R \vec u_{\ell_0}(t) \|_{\dot{H}^1 \times L^2(r \ge R)} &\lesssim R^{-4} \| \pi_R \vec u_{\ell_0}(t)\|_{\dot{H}^1 \times L^2(r \ge R)} +R^{-\frac{5}{2}} \| \pi_R \vec u_{\ell_0}(t)\|_{\dot{H}^1 \times L^2(r \ge R)}^2\\
&+  R^{-1} \| \pi_R \vec u_{\ell_0}(t)\|_{\dot{H}^1 \times L^2(r \ge R)}^3 +  R^{-11/2} \| \pi_R \vec u_{\ell_0}(t)\|_{\dot{H}^1 \times L^2(r \ge R)}^4\\
& +  R^{-4} \| \pi_R \vec u_{\ell_0}(t)\|_{\dot{H}^1 \times L^2(r \ge R)}^5  
}
where $P(R):= \{(c_1r^{-3}, c_2r^{-3}) \mid, c_1, c_2 \in \R\}$, $\pi_R$ denotes orthogonal projection onto $P(R)$ and $\pi_R^{\perp}$ denotes orthogonal projection onto the orthogonal complement of the plane $P(R)$ in the space $\dot{H}^1 \times L^2(r>R)(\R^5)$. We remark that the constant above is uniform in $t \in \R$. 
\end{lem}

The next step in the proof of Lemma~\ref{ul 0} is to prove that $(\p_r  u_{\ell_0, 0}, u_{\ell_0, 1})$ is compactly supported. 
\begin{claim}\label{comp supp claim} Let $\vec u_{\ell_0}$ be as in~\eqref{ul def}. Then $(\p_r  u_{\ell_0, 0}, u_{\ell_0, 1})$ must be compactly supported.
\end{claim}
\begin{proof}[Proof of Claim~\ref{comp supp claim}]
To prove the  claim, we pass to the $\vec v_{\ell_0}$ formulation.  With $(v_{\ell_0, 0}, v_{\ell_0, 1})$ defined as in~\eqref{ul limits} we can  conclude that for all $R>R_2$ large enough we have 
 \ali{\label{vl ineq}
\int_R^{\infty} &\left(\frac{1}{r} \p_r v_{\ell_0, 0}(r)\right)^2 \, dr + \int_R^{\I} (\p_r v_{\ell_0,1}( r))^2 \, dr   
\lesssim  R^{-19}v_{\ell_0,0}^2(R) + R^{-11} v_{\ell_0,0}^4(R)\\
&\quad+ R^{-11}v_{\ell_0,0}^6(R) + R^{-23} v_{\ell, 0}^8(R) +R^{-31} v_{\ell, 0}^{10}(R) 
+R^{-17}v_{\ell_0,1}^2(R)\\
& \quad + R^{-7} v_{\ell_0,1}^4(R) + R^{-5}v_{\ell_0,1}^6( R) +R^{-15}v_{\ell_0, 1}^8(R) + R^{-21} v_{\ell_0, 1}^{10}(R)\\
&\lesssim  R^{-11}(v_{\ell_0, 0}^2(R) + v_{\ell_0, 1}^2(R))
}
where the first inequality follows by rewriting the conclusion of Lemma~\ref{project ineq mod} in terms of $ \vec v_{\ell_0} = (v_{\ell_0, 0}, v_{\ell_0, 1})$ by using~\eqref{v u project},  and the last line following from the known decay estimates in \eqref{ul limits}.  

Next, arguing as in the proof of Corollary~\ref{diff esti}, we can establish difference estimates. Indeed, for all $R_2 \le r \le r' \le 2r$ we obtain
\ali{
&\abs{v_{\ell_0, 0}(r) - v_{\ell_0, 0}(r')} \lesssim r^{-4}(v_{\ell_0, 0}^2(r) + v_{\ell_0, 1}^2(r))^{\frac{1}{2}},\\
&\abs{v_{\ell_0, 1}(r) - v_{\ell_0, 1}(r')} \lesssim r^{-5}(v_{\ell_0, 0}^2(r) + v_{\ell_0, 1}^2(r))^{\frac{1}{2}}.
}
In terms of the vector $\vec v_{\ell_0}  = (v_{\ell_0, 0},  v_{\ell_0, 1})$, 
where
\ant{ 
\abs{\vec v_{\ell_0}(r)}:= (v_{\ell_0, 0}^2(r) + v_{\ell_0, 1}^2(r))^{\frac{1}{2}}
}
we then have 
\ali{
&\abs{\vec v_{\ell_0}(r) - \vec v_{\ell_0}(r')} \lesssim r^{-4}\abs{\vec v_{\ell_0}(r)}.
}
Using the triangle inequality we can then deduce that for fixed $r_0 \ge R_2$ large enough we have 
\ant{
\abs{\vec v_{\ell_0}(2^{n+1}r_0)} \ge \frac{3}{4} \abs{\vec v_{\ell_0}(2^{n}r_0)}.
}
Iterating this, we see that for each $n$, 
\ant{
\abs{\vec v_{\ell_0}(2^{n}r_0)} \ge \left(\frac{3}{4}\right)^{n} \abs{\vec v_{\ell_0}(r_0)}.
}
On the other hand, by~\eqref{ul limits} we have 
\ant{
\abs{ \vec v_{\ell_0}(2^n r_0)} \lesssim (2^n r_0)^{-2}.
}
Putting together the last two lines we get 
\ant{
3^n \abs{ \vec v_{\ell_0}(r_0)} \lesssim 1,
}
which implies that $\vec v_{\ell_0}(r_0) = (0, 0)$. Inserting this into~\eqref{vl ineq} gives
\ant{
\int_{r_0}^{\infty} \left(\frac{1}{r} \p_r v_{\ell_0, 0}(r)\right)^2 \, dr + \int_{r_0}^{\I} (\p_r v_{\ell_0,1}( r))^2 \, dr    = 0.
}
Therefore, 
\begin{multline*}
\| \vec u_{\ell_0}\|_{\dot{H}^1 \times L^2(r \ge r_0)}^2= \\= \int_{r_0}^{\infty} \left(\frac{1}{r} \p_r v_{\ell_0, 0}(r)\right)^2 \, dr + \int_{r_0}^{\I} (\p_r v_{\ell_0,1}( r))^2 \, dr + 3r_0^{-3}v_{\ell_0, 0}^2(r_0) + r_0^{-1} v_{\ell_0, 1}^2(r_0) = 0
\end{multline*}
which means that $(\p_r u_{\ell_0, 0}, u_{\ell_0, 1})$ is compactly supported. 
\end{proof}
We can now finish the proof of Lemma~\ref{ul 0} by showing that $\vec u_{\ell_0} = (0, 0)$.
\begin{proof}[Proof of Lemma~\ref{ul 0}]The proof is nearly identical to the proof of Lemma~\ref{ell=0 lem}. Suppose $$(\p_ru_{\ell_0,0}, u_{\ell_0,1}) \neq (0, 0)$$ and we argue by contradiction.  
By Claim~\ref{comp supp claim},   $(\p_ru_{\ell_0,0}, u_{\ell_0,1})$ is compactly supported. Then we can define
 $\rho_0>0$ by 
 \ant{
 \rho_0 := \inf \{ \rho \, : \,  \|\vec u_{\ell_0}\|_{\dot{H}^1 \times L^2(r \ge \rho)} =  0\}
 }
Let $\e>0$ be a small number to be determined below. Choose $\rho_1 = \rho_1( \e)$ so that 
\begin{align} \label{rho1 def}
&\frac{ 1}{2} \rho_0 < \rho_1 < \rho_0, \mand  \rho_0 - \rho_1 < \e, \\
&0< \| \vec u_{\ell_0}(0)\|_{\dot{H}^1 \times L^2(r \ge \rho_1)}^2  < \de_2^2 \label{rho1 def2}
\end{align}
where $\de_2$ is as in Lemma~\ref{mod cp}. With $(v_0, v_1)$ as in~\eqref{v def} we have 
\EQ{\label{all small l}
3 \rho_1^{-3} v_{\ell_0,0}^2(\rho_1) + \rho_1^{-1}v_{\ell_0, 1}^2(\rho_1)&+ \int_{\rho_1}^\I \left( \frac{1}{r} \p_r v_{\ell_0,0}( r) \right)^2 \, dr + \int_{\rho_1}^\I \left( \p_r v_{\ell_0,1}( r) \right)^2 \, dr  = \\
& = \| \pi_{\rho_1} \vec u_{\ell_0}(0)\|^2_{\dot{H}^1 \times L^2(r \ge \rho_1)}+\| \pi_{\rho_1}^{\perp} \vec u_{\ell_0}(0)\|^2_{\dot{H}^1 \times L^2(r \ge \rho_1)} \\
&= \|  \vec u_{\ell_0}(0)\|^2_{\dot{H}^1 \times L^2(r \ge \rho_1)} < \de^2_2
}
Setting  $R= \rho_1$ in~\eqref{vl ineq} we obtain  
\EQ{\label{vl rho1}
\left( \int_{\rho_1}^\I \left[\left( \frac{1}{r} \p_r v_{\ell_0,0}(r) \right)^2 + \left( \p_r v_{\ell_0,1}( r) \right)^2 \right]\, dr\right)^{\frac{1}{2}} & \lesssim \rho_1^{-\frac{11}{2}}\left( \abs{v_{\ell_0,0}( \rho_1)}^2 + \abs{v_{\ell_0,1}( \rho_1)}^2\right)^{\frac{1}{2}} \\
&\lesssim \rho_0^{-\frac{11}{2}}\left( \abs{v_{\ell_0,0}( \rho_1)}^2 + \abs{v_{\ell_0,1}( \rho_1)}^2\right)^{\frac{1}{2}} 
}
where we have used our assumption that $\frac{1}{2} \rho_0< \rho_1< \rho_0$ in the last line above. Since $v_{\ell_0, 0}( \rho_0) = v_{\ell_0,1}( \rho_0) = 0$ we can argue as in Corollary~\ref{diff esti} and Corollary~\ref{diff esti 2} to obtain
\ant{
&\abs{v_{\ell_0,0}( \rho_1)}^2 = \abs{v_{\ell_0,0}(\rho_1) - v_{\ell_0,0}( \rho_0)}^2 \le C_1\, \e^{3}\, (\abs{v_{\ell_0,0}( \rho_1)}^2 + \abs{v_{\ell_0,1}( \rho_1)}^2)\\
&\abs{v_{\ell_0,1}( \rho_1)}^2 = \abs{v_{\ell_0,1}(\rho_1) - v_{\ell_0,1}( \rho_0)}^2 \le C_2\, \e\, (\abs{v_{\ell_0,0}( \rho_1)}^2 + \abs{v_{\ell_0,1}( \rho_1)}^2)
}
where %here we have used that $\frac{1}{2} \rho_0< \rho_1 < \rho_0$ to obtain 
the constants $C_1, C_2$ depend only $\rho_0$ which is fixed, and the uniform constant in~\eqref{vl rho1}, but not on~$\e$. 
Putting the above estimates together yields
\EQ{
(\abs{v_{\ell_0,0}( \rho_1)}^2 + \abs{v_{\ell_0,1}( \rho_1)}^2) \le C_3 \e (\abs{v_{\ell_0,0}( \rho_1)}^2 + \abs{v_{\ell_0,1}( \rho_1)}^2),
}
which implies that $\abs{v_{\ell_0,0}( \rho_1)} = \abs{v_{\ell_0,1}( \rho_1)} = 0$ by taking $\e>0$ small enough. By~\eqref{vl rho1} and the equalities in~\eqref{all small l} we can deduce that 
\ant{
\| \vec u_{\ell_0}(0)\|_{\dot{H}^1 \times L^2(r \ge \rho_1)} = 0,
}
which contradicts the definition of $\rho_0$ since $\rho_1< \rho_0$. Therefore, $(\p_ru_{\ell_0,0}, u_{\ell_0,1}) =(0, 0)$ Since $u_{\ell_0}(r) \to 0$ as $r \to \infty$ we can also conclude that $(u_{\ell_0,0}, u_{\ell_0,1}) = (0, 0)$.
\end{proof}

\subsection{Proof of Proposition~\ref{rigid} and Proof of Theorem~\ref{main}}
 For clarity, we summarize what we have done in the proof of Proposition~\ref{rigid}. 
\begin{proof}[Proof of Proposition~\ref{rigid}]
Let $\vec u(t)$ be a solution to~\eqref{u eq} and suppose that the trajectory 
\ant{
K=\{ \vec u(t) \mid t \in \R\}
}
is pre-compact in $ \HH$. We recall that 
\ant{
r\vec  u(t, r) = \vec \psi(t, r) 
}
where $\vec \psi(t) \in \E_0$ is a solution to,~\eqref{an eq}. 
By Lemma~\ref{spacial decay} there exists $\ell_0 \in \R$ so that 
\begin{align} 
&\abs{r^3 u_0(r) - \ell_0 } = O(r^{-4}) \mas r \to \infty\\
&\abs{r \int_r^{\infty} u_1( \rho) \rho \, d\rho} = O(r^{-2}) \mas r \to \infty 
\end{align}
If $\ell_0 \neq 0$ then by Lemma~\ref{ul 0}, $\vec \psi(0) = (\psi_0, \psi_1) = (\fy_{\ell_0}(r), 0)$ where $\fy_{\ell_0}$ is the nonzero solution to~\eqref{ode} given by Proposition~\ref{ode prop} with $\al = \ell_0$. However, this is impossible since $\fy_{\ell_0} \not \in \E_0$, while on the other hand we know that  $\vec u(0) \in \HH$ implies $\vec \psi \in \E_0$ by~\eqref{finite energy in HH} and~\eqref{HH to E0}. 

Hence, $\ell_0 = 0$.  Then by Lemma~\ref{ell=0 lem} we can conclude that $\vec u(0)= (0, 0)$, which proves Proposition~\ref{rigidity}. 
\end{proof}

The proof of Theorem~\ref{main} is now complete. We conclude by summarizing the argument. 
\begin{proof}[Proof of Theorem~\ref{u main} and hence of Theorem~\ref{main}]
Suppose that Theorem~\ref{u main} fails. Then by Corollary~\ref{global crit} there exists a critical element, that is, a nonzero solution $\vec u_{\infty}(t) \in \HH$ to~\eqref{u eq} such that the trajectory $K= \{ \vec u_{\infty}(t) \mid t \in \R\}$ is pre-compact in $\HH$. However, Proposition~\ref{rigid} implies that any such solution is necessarily identically equal to $(0, 0)$, which contradicts the fact that the critical element $\vec u_{\infty}(t)$ is by construction, nonzero. 
\end{proof}

\bibliographystyle{plain}
\bibliography{researchbib}

\begin{thebibliography}{10}

\bibitem{AN}
G.~S. Adkins and C.~R. Nappi.
\newblock Stabilization of chiral solitons via vector mesons.
\newblock {\em Physics Letters B}, 137(3-4):251 -- 256, 1984.

\bibitem{BG}
P.~Bahouri and G\'erard.
\newblock High frequency approximation of solutions to critical nonlinear wave
  equations.
\newblock {\em Amer. J. Math.}, 121:131--175, 1999.

\bibitem{Biz}
P.~Bizo{{\'n}}, T.~Chmaj, and M.~Maliborski.
\newblock Equivariant wave maps exterior to a ball.
\newblock {\em Nonlinearity}, 25(5):1299--1309, 2012.

\bibitem{BCR}
P.~Bizo{{\'n}}, T.~Chmaj, and A.~Rostworowski.
\newblock Asymptotic stability of the skyrmion.
\newblock {\em Phys. Rev. D}, 75(12):121702, 5, 2007.

\bibitem{Bu}
A.~Bulut.
\newblock Maximizers for the {S}trichartz inequalities for the wave equation.
\newblock {\em Differential Integral Equations}, 23(11-12):1035--1072, 2010.

\bibitem{Bul12a}
A.~Bulut.
\newblock The defocusing cubic nonlinear wave equation in the
  energy-supercritical regime.
\newblock In {\em Recent advances in harmonic analysis and partial differential
  equations}, volume 581 of {\em Contemp. Math.}, pages 1--11. Amer. Math.
  Soc., Providence, RI, 2012.

\bibitem{Bul12b}
A.~Bulut.
\newblock Global well-posedness and scattering for the defocusing
  energy-supercritical cubic nonlinear wave equation.
\newblock {\em J. Funct. Anal.}, 263(6):1609--1660, 2012.

\bibitem{CKLS1}
R.~C\^{o}te, C.~Kenig, A.~Lawrie, and W.~Schlag.
\newblock Characterization of large energy solutions of the equivariant wave
  map problem: I.
\newblock {\em To appear in Amer. J. Math}, Preprint 2012.

\bibitem{DKM2}
T.~Duyckaerts, C.~Kenig, and F.~Merle.
\newblock Universality of the blow-up profile for small type $\textrm{II}$
  blow-up solutions of the energy critical wave equation.
\newblock {\em To appear in J. Eur. math. soc.}, 2010.

\bibitem{DKM1}
T.~Duyckaerts, C.~Kenig, and F.~Merle.
\newblock Universality of the blow-up profile for small radial type
  $\textrm{II}$ blow-up solutions of the energy critical wave equation.
\newblock {\em J. Eur math. Soc. (JEMS)}, 13(3):533--599, 2011.

\bibitem{DKM4}
T.~Duyckaerts, C.~Kenig, and F.~Merle.
\newblock Classification of radial solutions of the focusing, energy critical
  wave equation.
\newblock {\em Preprint}, 2012.

\bibitem{DKM3}
T.~Duyckaerts, C.~Kenig, and F.~Merle.
\newblock Profiles of bounded radial solutions of the focusing, energy-critical
  wave equation.
\newblock {\em Geom. Funct. Anal.}, 22(3):639--698, 2012.

\bibitem{DKM5}
T.~Duyckaerts, C.~Kenig, and F.~Merle.
\newblock Scattering for radial, bounded solutions of focusing supercritical
  wave equations.
\newblock {\em To appear in I.M.R.N}, Preprint, 2012.

\bibitem{GNR}
D-A. Geba, K.~Nakanishi, and S.~G. Rajeev.
\newblock Global well-posedness and scattering for {S}kyrme wave maps.
\newblock {\em Commun. Pure Appl. Anal.}, 11(5):1923--1933, 2012.

\bibitem{GR10a}
D-A. Geba and S.~G. Rajeev.
\newblock A continuity argument for a semilinear {S}kyrme model.
\newblock {\em Electron. J. Differential Equations}, pages No. 86, 9, 2010.

\bibitem{GR10b}
D-A Geba and S.~G. Rajeev.
\newblock Nonconcentration of energy for a semilinear {S}kyrme model.
\newblock {\em Ann. Physics}, 325(12):2697--2706, 2010.

\bibitem{IPS}
A.~D. Ionescu, B.~Pausader, and G.~Staffilani.
\newblock On the global well-posedness of energy-critical {S}chr\"odinger
  equations in curved spaces.
\newblock {\em Preprint}, 2010.

\bibitem{Kee-Tao}
Markus Keel and Terence Tao.
\newblock Endpoint {S}trichartz estimates.
\newblock {\em Amer. J. Math.}, 120(5):955--980, 1998.

\bibitem{KLS}
C.~Kenig, A.~Lawrie, and W.~Schlag.
\newblock Relaxation of wave maps exterior to a ball to harmonic maps for all
  data.
\newblock {\em Preprint}, 2013.

\bibitem{KM06}
C.~Kenig and F.~Merle.
\newblock Global well-posedness, scattering and blow-up for the
  energy-critical, focusing, non-linear {S}chr{\"o}dinger equation in the
  radial case.
\newblock {\em Invent. Math.}, 166(3):645--675, 2006.

\bibitem{KM08}
C.~Kenig and F.~Merle.
\newblock Global well-posedness, scattering and blow-up for the energy-critical
  focusing non-linear wave equation.
\newblock {\em Acta Math.}, 201(2):147--212, 2008.

\bibitem{KM11b}
C.~Kenig and F.~Merle.
\newblock Radial solutions to energy supercritical wave equations in odd
  dimensions.
\newblock {\em Discrete Contin. Dyn. Syst.}, 31(4):1365--1381, 2011.

\bibitem{KM10}
C.~E. Kenig and F.~Merle.
\newblock Scattering for {$\dot H^{1/2}$} bounded solutions to the cubic,
  defocusing {NLS} in 3 dimensions.
\newblock {\em Trans. Amer. Math. Soc.}, 362(4):1937--1962, 2010.

\bibitem{KM11a}
C.~E. Kenig and F.~Merle.
\newblock Nondispersive radial solutions to energy supercritical non-linear
  wave equations, with applications.
\newblock {\em Amer. J. Math.}, 133(4):1029--1065, 2011.

\bibitem{KV11b}
R.~Killip and M.~Visan.
\newblock The defocusing energy-supercritical nonlinear wave equation in three
  space dimensions.
\newblock {\em Trans. Amer. Math. Soc.}, 363(7):3893--3934, 2011.

\bibitem{KV11a}
R.~Killip and M.~Visan.
\newblock The radial defocusing energy-supercritical nonlinear wave equation in
  all space dimensions.
\newblock {\em Proc. Amer. Math. Soc.}, 139(5):1805--1817, 2011.

\bibitem{LS}
A.~Lawrie and W.~Schlag.
\newblock Scattering for wave maps exterior to a ball.
\newblock {\em Advances in Mathematics}, 232(1):57--97, 2013.

\bibitem{Li}
Dong Li.
\newblock Global global well-posedness of hedgehog solutions for the (3+1)
  {S}kyrme model.
\newblock {\em Preprint}, 2012.

\bibitem{LinS}
Hans Lindblad and Christopher~D. Sogge.
\newblock On existence and scattering with minimal regularity for semilinear
  wave equations.
\newblock {\em J. Funct. Anal.}, 130(2):357--426, 1995.

\bibitem{McT}
J.~B. McLeod and W.~C. Troy.
\newblock The {S}kyrme model for nucleons under spherical symmetry.
\newblock {\em Proc. Roy. Soc. Edinburgh Sect. A}, 118(3-4):271--288, 1991.

\bibitem{Shatah}
J.~Shatah.
\newblock Weak solutions and development of singularities of the {${\rm
  SU}(2)$} {$\sigma$}-model.
\newblock {\em Comm. Pure Appl. Math.}, 41(4):459--469, 1988.

\bibitem{STZ94}
Jalal Shatah and A.~Shadi Tahvildar-Zadeh.
\newblock On the {C}auchy problem for equivariant wave maps.
\newblock {\em Comm. Pure Appl. Math.}, 47(5):719--754, 1994.

\bibitem{Shen}
Ruipeng Shen.
\newblock On the energy subcritical nonlinear wave equation with radial data
  for 3<p<5.
\newblock {\em Preprint}, 2012.

\bibitem{Skyrme}
Tony Hilton~Royle Skyrme.
\newblock {\em Selected Papers With Commentary of Tony Hilton Royle Skyrme},
  volume~3.
\newblock World Scientific, 1994.

\bibitem{TS}
N.~Turok and D~Spergel.
\newblock Global texture and the microwave background.
\newblock {\em Physical Review Letters}, 64(23):2736--2739, 1990.

\end{thebibliography}

 \bigskip

\centerline{\scshape Andrew Lawrie}
\medskip
{\footnotesize
% please put the address of the first author
 \centerline{Department of Mathematics, The University of California, Berkeley}
\centerline{859 Evans Hall \#3840, Berkeley, CA 94720, U.S.A.}
\centerline{\email{ alawrie@math.berkeley.edu}}
} % Do not forget to end the {\footnotesize by the sign }

\end{document}